\documentclass[12pt,reqno]{amsart}

\usepackage{amsmath,amssymb,amsthm,graphicx,amsxtra,setspace}
\usepackage{amsfonts}
\usepackage{mathtools} 
\usepackage{xcolor}
\usepackage{soul} 
\usepackage{mathrsfs}

\usepackage[margin=0.80in]{geometry}

\usepackage{mathrsfs}
\usepackage{eucal}
\usepackage{hyperref}
\usepackage{upgreek}
\usepackage{multirow}
\usepackage{mathabx}
\usepackage{type1cm,eso-pic,color}
\usepackage{fouridx}

\usepackage{todonotes}

\allowdisplaybreaks

\usepackage{graphicx,eurosym}
\usepackage{hyperref}
\usepackage{mathtools}

\usepackage[cyr]{aeguill}

\colorlet{darkblue}{blue!50!black}

\hypersetup{
	colorlinks,%
	citecolor=blue,%
	filecolor=red,%
	linkcolor=darkblue,%
	urlcolor=blue,%
	pdfnewwindow=true,%
	pdfstartview={FitH}
}

\renewcommand{\d}{\/\mathrm{d}\/}

\def\L{\mathbb{L}}
\def\A{\mathrm{A}}
\def\I{\mathrm{I}}

\def\C{\mathrm{C}}
\def\f{\boldsymbol{f}}

\def\B{\mathrm{B}}
\def\D{\mathrm{D}}

\def\X{\mathbb{X}}
\def\x{\boldsymbol{x}}

\def\h{\boldsymbol{g}}

\def\v{\boldsymbol{y}}
\def\w{\boldsymbol{w}}
\def\W{\mathrm{W}}

\def\N{\mathbb{N}}

\def\V{\mathbb{V}}
\def\wi{\widetilde}

\def\u{\mathrm{U}}
\def\P{\mathrm{P}}
\def\u{\boldsymbol{u}}
\def\H{\mathbb{H}}

\newcommand{\R}{\mathbb{R}}

\renewcommand{\d}{\/\mathrm{d}\/}

\newtheorem{theorem}{Theorem}[section]
\newtheorem{lemma}[theorem]{Lemma}
\newtheorem{proposition}[theorem]{Proposition}
\newtheorem{assumption}[theorem]{Assumption}

\newtheorem{definition}[theorem]{Definition}

\newtheorem{remark}[theorem]{Remark}

\let\originalleft\left
\let\originalright\right
\renewcommand{\left}{\mathopen{}\mathclose\bgroup\originalleft}
\renewcommand{\right}{\aftergroup\egroup\originalright}

\numberwithin{equation}{section}

\newcommand{\Addresses}{{
		\footnote{
			\noindent \textsuperscript{1}Faculty of Primary Education, Hanoi Pedagogical University 2,  32 Nguyen Van Linh, Xuan Hoa, Phu Tho, Vietnam.\par\nopagebreak
			
			\noindent \textsuperscript{2,6}Department of Mathematics, Electric Power University, 235 Hoang Quoc Viet, Bac Tu Liem, Hanoi, Vietnam. \par\nopagebreak
			
			\noindent \textsuperscript{3}Tata Institute of Fundamental Research - Centre For Applicable Mathematics (TIFR-CAM), Bangalore 560065, Karnataka, India.\par\nopagebreak
			
			\noindent \textsuperscript{4}Center for Mathematics and Applications (NOVA Math), NOVA School of Science and Technology (NOVA FCT), Caparica,	Portugal. \par\nopagebreak
			
			\noindent \textsuperscript{5}Department of Mathematics, Indian Institute of Technology Roorkee-IIT Roorkee, Haridwar Highway, Roorkee, Uttarakhand 247667, India.\par\nopagebreak
			
			\noindent  
			\textit{e-mail:} \texttt{Bui Kim My: buikimmy@hpu2.edu.vn.}
			
			\textit{e-mail:} \texttt{Ho Thi Hang: hanght@epu.edu.vn.}
			
			\textit{e-mail:} \texttt{Kush Kinra: kushkinra@gmail.com;}
			
			\textit{e-mail:} \texttt{Manil T. Mohan: maniltmohan@ma.iitr.ac.in, maniltmohan@gmail.com.}
			
			\textit{e-mail:} \texttt{Pham Tri Nguyen: nguyenpt@epu.edu.vn.}
			
			\noindent \textsuperscript{*}Corresponding author.
			
			\textit{Key words:} Stochastic globally modified Navier-Stokes equations, asymptotic autonomous roubustness of random attractors, Kuratowski's measure of non-compactness, uniform-tail estimates, flattening property.
			
			Mathematics Subject Classification (2020): Primary: 37L55, 76D05; Secondary: 35B41, 37B55, 35B40.
}}}

\begin{document}


	\title[Asymptotic autonomy of random attractors for GMNSE]{Existence and asymptotic autonomous robustness of random attractors for three-dimensional stochastic globally modified Navier-Stokes equations on unbounded domains
		\Addresses}
	\author[B.K. My, H.T. Hang, K. Kinra, M.T. Mohan, P.T. Nguyen]{Bui Kim My\textsuperscript{1}, Ho Thi Hang\textsuperscript{2}, Kush Kinra\textsuperscript{3,4*}, Manil T. Mohan\textsuperscript{5} and Pham Tri Nguyen\textsuperscript{6}}
	
	\maketitle
	
	\begin{abstract}
		\textsf{In this article, we discuss the \emph{existence and asymptotically autonomous robustness} (AAR) (almost surely) of random attractors for 3D stochastic globally modified Navier-Stokes equations (SGMNSE) on Poincar\'e domains (which may be bounded or unbounded). Our aim is to investigate the existence and AAR of random attractors for 3D SGMNSE when the time-dependent forcing converges to a time-independent function under the perturbation of linear multiplicative noise as well as additive noise. The main approach is to provide a way to justify that, on some \emph{uniformly} tempered universe, the usual pullback asymptotic compactness of the solution operators is uniform across an \emph{infinite} time-interval $(-\infty,\tau]$. The backward uniform ``tail-smallness'' and ``flattening-property'' of the solutions over $(-\infty,\tau]$ have been demonstrated to achieve this goal. To the best of our knowledge, this is the first attempt to establish the existence as well as  AAR of random attractors for 3D SGMNSE on unbounded domains.}
	\end{abstract}

	\section{Introduction} \label{sec1}\setcounter{equation}{0}
	Let $\mathcal{O}$ be an open connected subset of $\R^3$. The time evolution of an incompressible fluid is described by the three-dimensional (3D) Navier-Stokes equations, which are given by
	\begin{equation}\label{3D-NSE}
		\left\{
		\begin{aligned}
			\frac{\partial \u}{\partial t}-\nu \Delta\u+(\u\cdot\nabla)\u+\nabla p&=\boldsymbol{f}, &&\text{ in }\  \mathcal{O}\times(\tau,\infty), \\ \nabla\cdot\u&=0, && \text{ in } \ \ \mathcal{O}\times[\tau,\infty), \\ \u&=0, && \text{ on } \ \ \partial\mathcal{O}\times[\tau,\infty), \\
			\u(\tau)&=\u_0, && \ x\in \mathcal{O},
		\end{aligned}
		\right.
	\end{equation}
	where $\tau\in\mathbb{R}$, $\u(x,t) \in \R^3$, $p(x,t)\in\R$ and $\f(x,t)\in \R^3$ denotes the velocity field, pressure and external forcing, respectively, and the positive constant $\nu$ represents the \emph{kinematic viscosity} of the fluid. The 2D Navier-Stokes equations are well studied in  literature, but there are significant challenges in examining the 3D Navier-Stokes equations. The lack of uniqueness of Leray-Hopf weak solutions and the existence of global strong solutions are some of them. To approach the original problem, there are a good number of modified versions of the 3D Navier-Stokes equations due to Leray and others, see for instance \cite{Constantin_2003}.
	
	In 2006, the authors in \cite{Caraballo+Real+Kloeden_2006} introduced a modified version of 3D Navier-Stokes equations by replacing the nonlinear term $(\u\cdot\nabla)\u$ with $F_{N}(\|\u\|_{\V})\left[(\u\cdot\nabla)\u\right]$, where the function $F_N:(0,+\infty)\to (0,1]$ is defined by
	\begin{align}\label{FN}
		F_{N}(r)=\min \left\{1,\frac{N}{r}\right\}, \;\;\; r\in(0,+\infty),
	\end{align}
	where $N>0$ is a given constant. The modifying factor $F_{N}(\|\u\|_{\V})$ depends on the norm of $\u$ in $\V$ (see Section \ref{FnO} below for function spaces). The resulting system 
	\begin{equation}\label{2}
		\left\{
		\begin{aligned}
			\frac{\partial \u}{\partial t}-\nu \Delta\u+F_{N}(\|\u\|_{\V})\left[(\u\cdot\nabla)\u\right]+\nabla p&=\boldsymbol{f}, &&\text{ in }\  \mathcal{O}\times(\tau,\infty), \\ \nabla\cdot\u&=0, && \text{ in } \ \ \mathcal{O}\times[\tau,\infty), \\ \u&=0, && \text{ on } \ \ \partial\mathcal{O}\times[\tau,\infty), \\
			\u(\tau)&=\u_0, && \ x\in \mathcal{O},
		\end{aligned}
		\right.
	\end{equation}
	is known as globally modified Navier-Stokes equations (GMNSE). The system \eqref{2} is indeed a globally modified version of the Navier-Stokes equations since the modifying factor $F_{N}(\|\u\|_{\V})$ depends on the norm $\|\u\|_{\V}=\|\nabla\u\|_{\L^2(\mathcal{O})}$, which in turn depends on $\nabla\u$ over the whole domain $\mathcal{O}$ and not just at or near the point $x \in \mathcal{O}$ under consideration. Essentially, it prevents large gradients dominating the dynamics and leading to explosions. An another modified version of the system \eqref{2} is addressed in Appendix \ref{ApA} where the modifying factor is $F_{N}(\|\u\|_{\mathbb{L}^4(\mathcal{O})})$. We hope that the modified system \eqref{A.1} can be explored much in future as the modifying factor depends on the norm $\|\u\|_{\mathbb{L}^4(\mathcal{O})}$ only. In this work, our focus is on the analysis of the stochastic counterpart of the system \eqref{2}.

	The global existence and uniqueness of weak as well as strong solutions of 3D GMNSE are well established in literature, see the works \cite{Caraballo+Kloeden_2013,Caraballo+Real+Kloeden_2006,Caraballo+Real+Kloeden_2010,Romito_2009}, etc. Similar to the case of Navier-Stokes equations, it is a challenging problem to study large time behavior of GMNSE on the whole space $\R^3$. Inspired by many interesting works such as \cite{BCLLLR,BL,CLR,CLR1,GLW,Wang+Kinra+Mohan_2023}, these kinds of analysis can be attempted on unbounded Poincar\'e domains. By a Poincar\'e  domain, we mean a domain in which the Poincar\'e inequality is satisfied.  A typical example of unbounded Poincar\'e domains in $\mathbb{R}^3$ is $\mathcal{O}=\R^2\times(-L,L)$ with $L>0$, see \cite[p.306]{R.Temam} and \cite[p.117]{Robinson2}. More precisely, we consider the following assumption on the domain $\mathcal{O}$:
	\begin{assumption}\label{assumpO}
		Let $\mathcal{O}$ be an open, connected and unbounded subset of $\R^3$, the boundary of which is uniformly of class $\mathrm{C}^3$ (see \cite{Heywood}). We assume that, there exists a positive constant $\lambda $ such that the following Poincar\'e inequality  is satisfied:
		\begin{align}\label{poin}
			\lambda\int_{\mathcal{O}} |\psi(x)|^2 \d x \leq \int_{\mathcal{O}} |\nabla \psi(x)|^2 \d x,  \ \text{ for all } \  \psi \in \H^{1}_0 (\mathcal{O}).
		\end{align}
	\end{assumption}
	\begin{remark}
		When $\mathcal{O}$ is a bounded domain, Poincar\'e inequality is satisfied automatically with $\lambda=\lambda_1$, where $\lambda_1$ is the first eigenvalue of the Stokes operator defined on bounded domains.
	\end{remark}


	The goal of this work is to investigate the existence and asymptotically autonomous robustness of pullback random attractors for the following 3D non-autonomous stochastic globally modified Navier-Stokes equations (SGMNSE) that is defined on the domain $\mathcal{O}$ satisfying Assumption \ref{assumpO} (which may be bounded or unbounded):
	\begin{equation}\label{1}
		\left\{
		\begin{aligned}
			\frac{\partial \u}{\partial t}-\nu \Delta\u+F_{N}(\|\u\|_{\V})\left[(\u\cdot\nabla)\u\right]+\nabla p&=\boldsymbol{f}+S(\u)\circ\frac{\d \W}{\d t}, && \text{ in }\  \mathcal{O}\times[\tau,\infty), \\ \nabla\cdot\u&=0,  &&  \text{ in } \ \ \mathcal{O}\times(\tau,\infty), \\ \u&=0, && \text{ on } \ \ \partial\mathcal{O}\times[\tau,\infty), \\
			\u(\tau)&=\u_0,&&  \ x\in \mathcal{O},
		\end{aligned}
		\right.
	\end{equation}
	where the term $S(\u)$ is referred as the diffusion  coefficient of the noise, and it is either independent of $\u$, that is, $S(\u)=\h\in\D(\A)$ (additive noise) or equal to $\u$ (linear multiplicative noise), the symbol $\circ$ means that  the stochastic integral is understood in the sense of Stratonovich, $\W=\W(t,\omega)$ is an one-dimensional two-sided Wiener process defined on a standard  probability space $(\Omega, \mathscr{F}, \mathbb{P})$, and $\D(\A)$ is the domain of the Stokes operator $\A$ defined in \eqref{Stokes}. The well-posedness of SGMNSE has been taken into consideration in the works \cite{Anh+Thanh+Tuyet_2023,Caraballo+Chen+Yang_2023_AMOP,Deugoue+Medjo_2018}, etc.
	

	\subsection{Literature survey}
	It is well-known that attractors plays an important role to analyze the large time behavior of dynamical systems, see \cite{Ball1997JNS,CLR2,CV2,Dia+Lappicy_2021,Li+Wang+Kloeden_2023,Robinson2,Robinson1,Sultanov_2023,R.Temam}, etc., and references therein. Large time behavior such as the existence of global/pullback/exponential/trajectory attractors of 3D globally modified Navier-Stokes equations has been well-discussed in literature, see \cite{Caraballo+Kloeden_2013,Caraballo+Real+Kloeden_2006,Kloeden+Langa+Real_2007,Ren_2014,Zhao+Yang_2017,Zhao+Caraballo_2019}. The works \cite{Arnold,BCF,CDF,CF,Schmalfussr}, etc.,  extended the concept of attractors for deterministic dynamical systems to random attractors for random dynamical systems.  Knowing that evolution equations originating from various fields of science and engineering  are frequently subject to simultaneous stochastic and non-autonomous forcing, the author in \cite{SandN_Wang} extended the concept of autonomous random dynamical systems to non-autonomous random dynamical systems. Several results on random attractors have been studied in light of these theoretical concepts, see \cite{chenp,GLW,KM2,KM3,KM7,PeriodicWang,Wang+Guo+Liu+Nguyen_2024,XC} and many others, for autonomous and non-autonomous stochastic equations. In particular, the works \cite{Anh+Thanh+Tuyet_2023,Caraballo+Chen+Yang_2023_SAM,Caraballo+Chen+Yang_2023_AMOP,Hang+My+Nguyen_2024}  deal with the random dynamics of 3D SGMNSE defined on bounded domains only. Very recently, in \cite{Hang+My+Nguyen_2024a, Hang+Nguyen_2024}, the authors established the existence of random attractors for 3D SGMNSE defined on unbounded domains. However, as far as we aware, this is the first attempt to establish the asymptotic autonomous robustness of random attractors for 3D SGMNSE defined on unbounded domains.


	\subsection{Aims, difficulties and approaches}

The non-autonomous characteristic of evolution systems is well-represented by the temporal dependence of forcing term. This may be the primary characteristic that sets autonomous evolution systems apart from non-autonomous systems. It makes sense that the non-autonomous dynamics of \eqref{1} grow more autonomous if the non-autonomous forcing term $\f(x,t)$ in \eqref{1} converges asymptotically to an autonomous forcing term in some way. \emph{Asymptotically autonomous dynamics} of \eqref{1} is the terminology used in the literature to describe this behavior in such cases.  Our main motivation is to investigate the asymptotically autonomous robustness of random attractors of  \eqref{1} driven by additive as well as linear multiplicative noises when $\f(\cdot,\cdot)$ converges to a time-independent function $\f_{\infty}(\cdot)$ in some sense (see Assumption \ref{Hypo_f-N} below). In particular, we aim to show the following convergence:
\begin{align}\label{MT11}
			\lim_{\tau\to -\infty}\mathrm{dist}_{\H}
			(\mathcal{A}(\tau,\omega),
			\mathcal{A}_{\infty}(\omega))=0, \  \text{a.e. } \omega\in\Omega,
		\end{align}
        where $\mathcal{A}
		=\{\mathcal{A}(\tau,\omega):\tau\in\mathbb{R},
		\omega\in\Omega\}$ is the pullback random attractor of system \eqref{1} corresponding to $\f(\cdot,\cdot)$ and 
$\mathcal{A}_{\infty}=
		\{\mathcal{A}(\omega):
		\omega\in\Omega\}$ is the random attractor of a stochastic globally modified Navier-Stokes equations with the autonomous forcing $\f_\infty$ (see the system \eqref{A-SNSE} below).

	A key step to demonstrate  \eqref{MT11} is how to prove the uniform precompactness of $\bigcup\limits_{s\in(-\infty,\tau]}\mathcal{A} (s,\omega)$ in $\H$. Because of the abstract theory introduced in \cite{SandN_Wang}, the pullback asymptotic compactness of $\Phi$ implies the compactness of every single time-section $\mathcal{A}(\tau,\omega)$. But, one cannot expect that the usual pullback asymptotic compactness of $\Phi$ leads to the
	precompactness of $\bigcup\limits_{s\in(-\infty,\tau]}\mathcal{A}(s,\omega)$ in $\H$ because $(-\infty,\tau]$ is an infinite interval. However, it is possible if we are able to prove that the usual pullback asymptotically compactness of $\Phi$ is uniform with respect to a uniformly tempered universe (see \eqref{D-NSE} below)
	over the infinite interval $(-\infty,\tau]$.

	On bounded domains, the \emph{uniform} pullback asymptotic compactness of $\Phi$ over $(-\infty,\tau]$ can be established via compact Sobolev embeddings, see for instance \cite{KRM,WL}, etc. When  $\mathcal{O}$ is an unbounded domain, similar to the one considered in this article, the Sobolev embeddings are no longer compact, proving such a \emph{uniform} asymptotic compactness is therefore harder than the bounded domain case. In this article, we use a result (see Lemma \ref{K-BAC} below) of \textit{Kuratowski's measure of non-compactness}, which requires the \textit{uniform tail-estimates} (\cite{UTE-Wang}) and \textit{flattening property} (\cite{Kloeden+Langa_2007}) of the solutions to the system \eqref{1}.
	
	The fluid dynamic equations like \eqref{1} contain the pressure term $p$, which is essentially different from the parabolic or hyperbolic equations as considered in \cite{CGTW,chenp,LGL,UTE-Wang}, etc. The pressure term $p$ cannot be eliminated by the divergence theorem when we prove uniform tail-estimates or flattening property. However, by taking the divergence in \eqref{1} and using the incompressibility condition ``$\nabla\cdot \u=0$'', we obtain the  rigorous expression of the pressure term in the weak sense (in $\mathrm{L}^2(\mathcal{O})$) as follows:
	\begin{align}\label{pressure}
		p=(-\Delta)^{-1}\bigg[F_{N}(\|\u\|_{\V}) \cdot \sum_{i,j=1}^{2} \frac{\partial^2}{\partial x_i\partial x_j}(u_iu_j)-\nabla\cdot\f\bigg].
	\end{align}
	Since $\mathcal{O}$ is a Poincar\'e domain with the boundary uniformly of class $\mathrm{C}^3$,  for the equation $-\Delta\psi=0$ with Dirichlet's boundary condition ($\psi\in\mathrm{H}^2(\mathcal{O})\cap\mathrm{H}_0^1(\mathcal{O})$),   \eqref{poin} yields
	\begin{align*}
		(-\Delta\psi,\psi)=	\|\nabla\psi\|_{\mathrm{L}^2(\mathcal{O})}^2=0 \Rightarrow \lambda \|\psi\|_{\mathrm{L}^2(\mathcal{O})}=0\Rightarrow \psi=0.
	\end{align*}
	Therefore, 	the operator $-\Delta$ is invertible in $\mathrm{L}^2(\mathcal{O})$  and its inverse $(-\Delta)^{-1}$ is bounded.

	It is worth mentioning here that the pressure term has not been taken into consideration in some of the works in  literature, see for instance \cite{Lin+Guo+Yang+Miranville_2024}. While proving uniform tail-estimates or flattening property for incompressible fluid dynamic models, one always encounter  a term like the one given in \eqref{pressure} which should be estimated in a proper way (see for instance \cite{KK+FC1,KK+FC2,KM7,Kinra+Mohan+Wang_2024,Wang+Kinra+Mohan_2023}, etc.). In this article, we have properly estimated the pressure term for both cases of additive as well as  linear multiplicative noises  (see estimates \eqref{p-value-N} and \eqref{p-value-N-M}, respectively below).
	
	Note that the non-compactness of Sobolev embeddings on unbounded domains can be avoided by using the widely accepted concept of energy equations (introduced in \cite{Ball1997JNS}); for examples, see \cite{BCLLLR,BL,CLR,CLR1,KM7,PeriodicWang}, etc. It should be noted that since $(-\infty,\tau]$ is an infinite time-interval, we are unable to use the concept of energy equations to show the uniform pullback asymptotic compactness of $\Phi$ in $\H$. 
	
	Another difficulty we face in this article is demonstrating the measurability (with respect to sample points) of the uniformly pullback compact attractor since the uniform pullback asymptotic compactness of $\Phi$ requires the consideration of a uniformly tempered universe (see Subsection \ref{BUTRS} below). In comparison to the usual tempered universe (see $\mathfrak{D}_{\infty}$ in Subsection \ref{CoRS} below), showing the measurability of a uniformly pullback compact attractor on uniformly tempered universe is more complicated because the radii of the uniform pullback absorbing sets contain a term with supremum over an uncountable set $(-\infty,\tau]$ (see \eqref{IRAS1-N} below). In \cite{SandN_Wang}, the author established the measurability of the usual pullback random attractor over usual tempered universe. To overcome this difficulty, we show that a uniformly pullback compact attractor is exactly the same as the usual pullback random attractor (see \textbf{Step VI} in Section \ref{thm1.4}). This idea has been successfully implemented in the literature; see the works \cite{CGTW,Wang+Kinra+Mohan_2023,WL}, etc.

	\subsection{Outline}
	In the next section, we first consider an abstract formulation of the system  \eqref{1},  discuss  the properties of an Ornstein-Uhlenbeck process and provide some details on Kuratowski's measure of non-compactness. We also present the assumption and state the main results (Theorems \ref{MT1-N} and \ref{MT1}) of this article in the same section. In Section \ref{sec3}, we prove Theorem \ref{MT1-N} for the problem \eqref{1} driven by additive noise. In the final section, we prove Theorem \ref{MT1} for the problem \eqref{1} driven by multiplicative noise. In Appendix \ref{ApA}, we address a different modified version of 3D Navier-Stokes equations by modifying the nonlinear term $(\u\cdot\nabla)\u$ with $F_{N}(\|\u\|_{\mathbb{L}^4(\mathcal{O})})\left[(\u\cdot\nabla)\u\right]$. We also discuss the existence and uniqueness of weak as well as strong solutions of the problem \eqref{A.1} (Theorems \ref{thm-weak} and \ref{thm-strong}) in the same section.

	\section{Mathematical Formulations and Auxiliary Results}\label{sec2}\setcounter{equation}{0}
	In this section, we first discuss some necessary function spaces which are needed for obtaining the main results of this work. Then we define linear and nonlinear operators which help us to obtain an abstract formulation of the stochastic system \eqref{1}. Next we provide the assumption and state the main results of this article. Further, we formulate an Ornstein-Uhlenbeck process, discuss its properties and define backward tempered random sets. Finally, we provide the definition of  Kuratowski's measure of non-compactness and some related results (see Lemma \ref{K-BAC}). Note that Lemma \ref{K-BAC} plays a crucial role in proving the time-semi-uniform asymptotic compactness (see Subsection \ref{thm1.4}).

	\subsection{Function spaces and operators}\label{FnO}
	Let the space $\mathcal{V}:=\{\u\in\C_0^{\infty}(\mathcal{O};\R^{3}):\nabla\cdot\u=0\},$ where $\C_0^{\infty}(\mathcal{O};\R^{3})$ denotes the space of all infinite times differentiable functions  ($\R^{3}$-valued) with compact support in $\mathcal{O}$. Let $\H$ and $\V$ denote the completion of $\mathcal{V}$ in 	$\mathrm{L}^2(\mathcal{O};\R^{3})$ and $\mathrm{H}^1(\mathcal{O};\R^{3})$ norms, respectively. The spaces  $\H$ and $\V$ are endowed with the norms $\|\u\|_{\H}^2:=\int_{\mathcal{O}}|\u(x)|^2\d x$ and $\|\u\|_{\V}^2:=\int_{\mathcal{O}}|\nabla\u(x)|^2\d x$ (using the Poincar\'e inequality),  respectively. The induced duality between the spaces $\V$ and $\V^*$ is denoted by $\langle\cdot,\cdot\rangle.$ Moreover, we have the continuous embedding $\V\hookrightarrow\H\equiv\H^*\hookrightarrow\V^*.$

\begin{remark}
    Note that the notations $\nabla\cdot\u$ and $\nabla\u$ represent the divergence and gradient of the vector-field $\u=(u_1,u_2,u_3)$, respectively. In particular, $\nabla\cdot\u$ and $\nabla\u$ are given by
    \begin{align}
        \nabla\cdot\u = \sum_{i=1}^3 \frac{\partial u_i}{\partial x_i} \;\;\; \text{ and }\;\;\; \nabla\u= 
\begin{pmatrix}
\frac{\partial u_1}{\partial x_1} & \frac{\partial u_1}{\partial x_2} & \frac{\partial u_1}{\partial x_3}   \\[4pt]
\frac{\partial u_2}{\partial x_1}  & \frac{\partial u_2}{\partial x_2} & \frac{\partial u_2}{\partial x_3}  \\[4pt]
\frac{\partial u_3}{\partial x_1}  & \frac{\partial u_3}{\partial x_2}  & \frac{\partial u_3}{\partial x_3} 
\end{pmatrix},
    \end{align}
    respectively.
\end{remark}

	\subsubsection{Linear operator}\label{LO}
	Let $\mathcal{P}: \L^2(\mathcal{O}) \to\H$ denote the Helmholtz-Hodge orthogonal projection (cf.  \cite{Farwig+Kozono+Sohr_2007}). Let us define the Stokes operator
	\begin{equation}\label{Stokes}
		\A\u:=-\mathcal{P}\Delta\u,\;\u\in\D(\A).
	\end{equation}
	The operator $\A:\V\to\V^*$ is a linear continuous operator.	Since the boundary of $\mathcal{O}$ is uniformly of class $\mathrm{C}^3$, it is inferred that $\D(\A)=\V\cap\H^2(\mathcal{O}),$ and $\|\A\u\|_{\H}$ defines a norm in $\D(\A),$ which is equivalent to the one in $\H^2(\mathcal{O})$ (cf. \cite[Lemma  1]{Heywood}). The above argument implies that $\mathcal{P}:\H^2(\mathcal{O})\to\H^2(\mathcal{O})$ is a bounded operator. Moreover, the operator $\A$ is  non-negative self-adjoint  in $\H$ and
	\begin{align}\label{2.7a}
		\langle\A\u,\u\rangle =\|\u\|_{\V}^2,\ \textrm{ for all }\ \u\in\V \ \text{ and }\ \|\A\u\|_{\V^*}\leq \|\u\|_{\V}.
	\end{align}
	
	\begin{remark}
		Since $\mathcal{O}$ is a Poincar\'e domain with the boundary of $\mathcal{O}$ is uniformly of class $\mathrm{C}^3$, the operator $\A$ is invertible and its inverse $\A^{-1}$ is bounded.
	\end{remark}
	
	\subsubsection{Nonlinear operator}\label{BO}
	Let us define the \emph{trilinear form} $b(\cdot,\cdot,\cdot):\V\times\V\times\V\to\R$ by $$b(\u,\boldsymbol{v},\w)=\int_{\mathcal{O}}(\u(x)\cdot\nabla)\boldsymbol{v}(x)\cdot\w(x)\d x=\sum_{i,j=1}^3\int_{\mathcal{O}}\u_i(x)
	\frac{\partial \boldsymbol{v}_j(x)}{\partial x_i}\w_j(x)\d x.$$ 
	If $\u, \boldsymbol{v}$ are two  elements such that the linear map $b(\u, \boldsymbol{v}, \cdot) $ is continuous on $\V$, then the corresponding element of $\V^*$ is denoted by $\B(\u, \boldsymbol{v})$. We also denote $\B(\u) = \B(\u, \u)=\mathcal{P}[(\u\cdot\nabla)\u]$. An integration by parts gives
	\begin{equation}\label{b0}
		\left\{
		\begin{aligned}
			b(\u,\boldsymbol{v},\boldsymbol{v}) &= 0, &&\text{ for all }\ \u,\boldsymbol{v} \in\V,\\
			b(\u,\boldsymbol{v},\w) &=  -b(\u,\w,\boldsymbol{v}), && \text{ for all }\ \u,\boldsymbol{v},\w\in \V.
		\end{aligned}
		\right.\end{equation}
	We also define
	\begin{align*}
		b_{N}(\u,\boldsymbol{v},\w)= F_{N}(\|\boldsymbol{v}\|_{\V})\cdot b(\u,\boldsymbol{v},\w), \; \text{ for all }\ \u,\boldsymbol{v},\w\in \V.
	\end{align*}
	The form $b_{N}$ is linear in $\u$ and $\w$ but it is nonlinear in $\boldsymbol{v}$, however, we have the identity
	\begin{align*}
		b_N(\u,\boldsymbol{v},\boldsymbol{v}) &= 0,\ \text{ for all }\ \u,\boldsymbol{v} \in\V.
	\end{align*}

	\begin{remark}\label{Rem2.1}
		In view of H\"older's and interpolation inequalities, and Sobolev embedding ($\|\u\|_{\L^6(\mathcal{O})}\leq C \|\u\|_{\V}$), we obtain
		\begin{align}\label{HI+SE}
			|b(\u,\boldsymbol{v},\w)| &\leq \|\u\|_{\L^6(\mathcal{O})} \|\nabla\boldsymbol{v}\|_{\L^2(\mathcal{O})}\|\w\|_{\L^3(\mathcal{O})}\leq  C \|\u\|_{\L^6(\mathcal{O})} \|\nabla\boldsymbol{v}\|_{\L^2(\mathcal{O})} \|\w\|^{\frac12}_{\L^2(\mathcal{O})} \|\w\|^{\frac12}_{\L^6(\mathcal{O})} 
			\nonumber\\ & \leq C \|\u\|_{\V} \|\boldsymbol{v}\|_{\V} \|\w\|^{\frac12}_{\H} \|\w\|^{\frac12}_{\V}, \ \text{ for all } \u, \boldsymbol{v}, \w\in \V.
		\end{align}
	\end{remark}
	Next, we define the operator $\B_{N}:\V\times\V\to\V^{\ast}$ by 
	\begin{align*}
		\left\langle \B_{N}(\u,\boldsymbol{v}), \w \right\rangle = b_{N}(\u,\boldsymbol{v},\w), \;\text{ for all } \u, \boldsymbol{v}, \w\in \V
	\end{align*}
	For simplicity of notation, we write $\B_{N}(\u) = \B_{N}(\u,\u)$. Moreover, we also have the following equality 
	\begin{align}\label{BN-diff}
		\B_{N}(\u,\u)-\B_{N}(\boldsymbol{v},\boldsymbol{v})
		& = F_{N}(\|\u\|_{\V}) \cdot  \B (\u-\boldsymbol{v},\u )
		+ \big[ F_{N}(\|\u\|_{\V}) - F_{N}(\|\boldsymbol{v}\|_{\V})\big]\cdot \B ( \boldsymbol{v}, \u )
		\nonumber\\ & \quad
		+ F_{N}(\|\boldsymbol{v}\|_{\V}) \cdot \B(\boldsymbol{v},\u-\boldsymbol{v}).
	\end{align}
	
	Let us now recall some properties of the function $F_{N}$ (cf. \cite{Romito_2009}) which will be useful in the sequel. 	
	\begin{lemma}[{\cite[Lemma 2.1]{Romito_2009}}]
		For any $\u,\boldsymbol{v}\in\V$ and each $N>0$, 
		\begin{align}
			0 \leq \|\u\|_{\V} F_{N}(\|\u\|_{\V}) & \leq N, \label{FN1}\\
			|F_{N}(\|\u\|_{\V})-F_{N}(\|\boldsymbol{v}\|_{\V})|&\leq \frac{1}{N} F_{N}(\|\u\|_{\V})F_{N}(\|\boldsymbol{v}\|_{\V})\cdot \|\u-\boldsymbol{v}\|_{\V}.\label{FN2} 
		\end{align}
	\end{lemma}	
	
	\subsection{Abstract formulation and Ornstein-Uhlenbeck process}\label{2.5}
	Taking the projection $\mathcal{P}$ on the 3D SGMNSE \eqref{1}, one obtains the following abstract form:
	\begin{equation}\label{SNSE}
		\left\{
		\begin{aligned}
			\frac{\d\u}{\d t}+\nu \A\u+\B_N(\u)&=\mathcal{P}\f +S(\u)\circ\frac{\d \W}{\d t} , \\
			\u(\tau)&=\u_{0},
		\end{aligned}
		\right.
	\end{equation}
	where $S(\u)=\u$ or independent of $\u$, $\W(t,\omega)$ is a  standard scalar Wiener process on the probability space $(\Omega, \mathscr{F}, \mathbb{P}),$ where $\Omega=\{\omega\in C(\R;\R):\omega(0)=0\}$ endowed with the compact-open topology given by the complete metric
	\begin{align*}
		d_{\Omega}(\omega,\omega'):=\sum_{m=1}^{\infty} \frac{1}{2^m}\frac{\|\omega-\omega'\|_{m}}{1+\|\omega-\omega'\|_{m}},\ \text{ where }\  \|\omega-\omega'\|_{m}:=\sup_{-m\leq t\leq m} |\omega(t)-\omega'(t)|,
	\end{align*}
	and $\mathscr{F}$ is the Borel sigma-algebra induced by the compact-open topology of $(\Omega,d_{\Omega}),$ $\mathbb{P}$ is the two-sided Wiener measure on $(\Omega,\mathscr{F})$. From \cite{FS}, it is clear that  the measure $\mathbb{P}$ is ergodic and invariant under the translation-operator group $\{\vartheta_t\}_{t\in\R}$ on $\Omega$ defined by
	\begin{align*}
		\vartheta_t \omega(\cdot) := \omega(\cdot+t)-\omega(t), \ \text{ for all }\ t\in\R, \ \omega\in \Omega.
	\end{align*}
	The operator $\vartheta(\cdot)$ is known as  the \emph{Wiener shift operator}. Furthermore, the quadruple $(\Omega,\mathscr{F},\mathbb{P},\vartheta)$ defines a metric dynamical system, cf. \cite{Arnold,BCLLLR}.

	\subsubsection{Ornstein-Uhlenbeck process}
	Let $\sigma>0$ be a constant (which will be specified later) and consider 
	\begin{align}\label{OU1}
		z(\vartheta_{t}\omega) =  \int_{-\infty}^{t} e^{-\sigma(t-\xi)}\d \W(\xi), \ \ \omega\in \Omega,
	\end{align} which is the stationary solution of the one dimensional Ornstein-Uhlenbeck equation
	\begin{align}\label{OU2}
		\d z(\vartheta_t\omega) + \sigma z(\vartheta_t\omega)\d t =\d\W(t).
	\end{align}
	It is known from \cite{FAN} that there exists a $\vartheta$-invariant subset $\widetilde{\Omega}\subset\Omega$ of full measure such that $z(\vartheta_t\omega)$ is continuous in $t$ for every $\omega\in \widetilde{\Omega},$ and
	\begin{align}
		\lim_{t\to +\infty} e^{-\delta t}|z(\vartheta_{-t}\omega)| &=0, \ \text{ for all } \ \delta>0,\label{Z5}\\
		\lim_{t\to \pm \infty} \frac{1}{t} \int_{0}^{t} z(\vartheta_{\xi}\omega)\d\xi &=\lim_{t\to \pm \infty} \frac{|z(\vartheta_t\omega)|}{|t|}=0.\label{Z3}
	\end{align} 
	For further analysis of this work, we do not distinguish between $\widetilde{\Omega}$ and $\Omega$. Since, $\omega(\cdot)$ has sub-exponential growth, $\Omega$ can be written as $\Omega=\bigcup\limits_{M\in\N}\Omega_{M}$, where
	\begin{align*}
		\Omega_{M}:=\{\omega\in\Omega:|\omega(t)|\leq Me^{|t|},\text{ for all }t\in\R\}, \text{ for every } M\in\N.
	\end{align*}
	
	The following result helps us to prove the Lusin continuity of solutions with respect to the sample points.
	
	\begin{lemma}\label{conv_z}
		For each $N\in\N$, suppose $\omega_k,\omega_0\in\Omega_{M}$ are such that $d_{\Omega}(\omega_k,\omega_0)\to0$ as $k\to+\infty$. Then, for each $\tau\in\R$ and $T\in\R^+$ ,
		\begin{align}
			&\sup_{t\in[\tau,\tau+T]}\bigg[|z(\vartheta_{t}\omega_k)-z(\vartheta_{t}\omega_0)|+|e^{ z(\vartheta_{t}\omega_k)}-e^{ z(\vartheta_{t}\omega_0)}|\bigg]\to 0 \ \text{ as } \ k\to+\infty,\nonumber\\
			&\sup_{k\in\N}\sup_{t\in[\tau,\tau+T]}|z(\vartheta_{t}\omega_k)|\leq C(\tau,T,\omega_0).\label{conv_z2}
		\end{align}
	\end{lemma}
	\begin{proof}
		See the proofs of Corollary 22 and Lemma 2.5 in \cite{CLL} and \cite{YR}, respectively.
	\end{proof}

	\subsubsection{Backward-uniformly tempered random set}\label{BUTRS}
	A bi-parametric set $\mathcal{D}=\{\mathcal{D}(\tau,\omega)\}$ in a Banach space $\X$ is said to be \emph{backward-uniformly tempered} if
	\begin{align}\label{BackTem}
		\lim_{t\to +\infty}e^{-ct}\sup_{s\leq \tau}\|\mathcal{D}(s-t,\vartheta_{-t}\omega)\|^2_{\X}=0\  \text{ for all } \  (\tau,\omega,c)\in\R\times\Omega\times\R^+,
	\end{align}
	where $\|\mathcal{D}\|_{\X}=\sup\limits_{\x\in \mathcal{D}}\|\x\|_{\X}.$
	\subsubsection{Class of random sets}\label{CoRS}
	\begin{itemize}
		\item Let ${\mathfrak{D}}$ be the collection of subsets of $\H$ defined as:
		\begin{align}\label{D-NSE}
			{\mathfrak{D}}=\left\{{\mathcal{D}}=\{{\mathcal{D}}(\tau,\omega):(\tau,\omega)\in\R\times\Omega\}:\mathcal{D} \text{ satisfies } \eqref{BackTem} \right\}.
		\end{align}
		\item Let ${\mathfrak{B}}$ be the collection of subsets of $\H$ defined as:
		\begin{align*}
			{\mathfrak{B}}=\left\{{\mathcal{B}}=\{{\mathcal{B}}(\tau,\omega):(\tau,\omega)\in\R\times\Omega\}:\lim_{t\to +\infty}e^{-ct}\|{\mathcal{B}}(\tau-t,\vartheta_{-t}\omega)\|^2_{\H}=0\right\},
		\end{align*}
		for all $c>0$.
		\item Let ${\mathfrak{D}}_{\infty}$ be the collection of subsets of $\H$ defined as:
		\begin{align*}
			{\mathfrak{D}}_{\infty}=\left\{\widehat{\mathcal{D}}=\{\widehat{\mathcal{D}}(\omega):\omega\in\Omega\}:\lim_{t\to +\infty}e^{-ct}\|\widehat{\mathcal{D}}(\vartheta_{-t}\omega)\|^2_{\H}=0,\ \text{ for all }  \ c>0\right\}.
		\end{align*}
	\end{itemize}

\subsection{Assumption and  main results}	
To investigate the asymptotically autonomous robustness of random attractors of  \eqref{1} driven by additive as well as linear multiplicative noises we assume that $\f$ satisfies the following conditions:
	\begin{assumption}\label{Hypo_f-N}
		$\f\in\mathrm{L}^{2}_{\emph{loc}}
		(\R;\L^2(\mathcal{O}))$ converges to a time-independent function $\f_{\infty}\in\L^2(\mathcal{O})$ as follows:
		\begin{align*}
			\lim_{\tau\to -\infty}\int^{\tau}_{-\infty}
			\|\f(t)-\f_{\infty}\|^2_{\L^2(\mathcal{O})}\d t=0.
		\end{align*}
	\end{assumption}

	\begin{remark}
		1.  Assumption \ref{Hypo_f-N} implies the following conditions (see Caraballo et al. \cite{CGTW}):
		\begin{itemize}
			\item Uniformness condition:
			\begin{align}\label{G3}
				&\sup_{s\leq \tau}\int_{-\infty}^{s}e^{\kappa(r-s)} \|\f(r)\|^2_{\L^2(\mathcal{O})}\d r<+\infty,  \ \mbox{$\text{ for all }\ \kappa>0$, $\tau\in\mathbb{R}$}.
			\end{align}
			\item The tails of the forcing $\f$ are backward-uniformly small:
			\begin{align}\label{f3-N}
				\lim_{k\rightarrow\infty}\sup_{s\leq \tau}\int_{-\infty}^{s}e^{\kappa(r-s)}
				\int_{\mathcal{O}\cap\{|x|\geq k\}}|\f(x,r)|^{2}\d x\d r=0,  \ \mbox{$\text{ for all }\  \kappa>0$, $\tau\in\mathbb{R}$.}
			\end{align}
		\end{itemize}
		
		2.  	We also deduce from Assumption \ref{Hypo_f-N} that for any $T>0$
		\begin{align}\label{BC8-A}
			\int_{0}^{T}\|\f(t+\tau)-\f_{\infty}\|^2_{\L^2(\mathcal{O})} \d t\leq \int_{-\infty}^{\tau+T}\|\f(t)-\f_{\infty}\|^2_{\L^2(\mathcal{O})} \d t\to 0 \ \text{ as } \ \tau\to -\infty.
		\end{align}
		
		3. An example of Assumption \ref{Hypo_f-N} is
		$\f(x,t)=\f_\infty(x)e^t+\f_\infty(x)$ with $\f_{\infty}\in\L^2(\mathcal{O})$.
		
	\end{remark}
	
	Now, we are ready to state our main results of this article, which are as follows:
	\begin{theorem}[{\texttt{Additive noise case}}]\label{MT1-N}
		Under Assumptions \ref{assumpO} and \ref{Hypo_f-N}, the non-autonomous random dynamical system $\Phi$ generated by \eqref{1} with $S(\u)=\h$ has a
		unique pullback random attractor
		$\mathcal{A}
		=\{\mathcal{A}(\tau,\omega):\tau\in\mathbb{R},
		\omega\in\Omega\}$ such that
		$\bigcup\limits_{s\in(-\infty,\tau]}\mathcal{A}
		(s,\omega)$ is precompact in $\H$ and  $$\lim_{t \to +\infty} e^{- \gamma t}\sup_{s\in(-\infty,\tau]
		}\|\mathcal{A}(s-t,\vartheta_{-t} \omega ) \|_{\H} =0,$$ for any
		$\gamma>0$, $\tau\in \mathbb{R}$ and  $\omega\in\Omega$.  In addition, the time-section  $\mathcal{A}(\tau,\omega)$ is asymptotically
		autonomous robust in $\H$, and the limiting set of $\mathcal{A}(\tau,\omega)$ as $\tau\rightarrow-\infty$ is just determined
		by the random attractor $\mathcal{A}_{\infty}=
		\{\mathcal{A}(\omega):
		\omega\in\Omega\}$ of a stochastic globally modified Navier-Stokes equations with the autonomous forcing $\f_\infty$ (see the system \eqref{A-SNSE} below) , that is,
		\begin{align}\label{MT2-N}
			\lim_{\tau\to -\infty}\mathrm{dist}_{\H}
			(\mathcal{A}(\tau,\omega),
			\mathcal{A}_{\infty}(\omega))=0, \  \text{a.e. } \omega\in\Omega.
		\end{align}
		Furthermore, we also justify the
		\texttt{asymptotically autonomous robustness in probability}:
		\begin{align}\label{MT3-N}
			\lim_{\tau\to -\infty}\mathbb{P}\Big(\omega\in\Omega:\mathrm{dist}_{\H}
			(\mathcal{A}(\tau,\omega),
			\mathcal{A}_{\infty}(\omega))\geq\delta\Big){=0},\ \ \ \text{ for all } \ \delta>0.
		\end{align}
	\end{theorem}
	
	\begin{remark}
		It is remarkable to note that the following condition on $\h$ has been considered in many works to investigate the random dynamics of 2D Navier-Stokes equations: There exists a constant ${\aleph}>0$ such that
		$\h\in\D(\A)$ satisfies
		\begin{align*}
			\bigg|
			\sum_{i,j=1}^2\int_{\mathcal{O}}\u_i(x)
			\frac{\partial \h_j(x)}{\partial x_i}\u_j(x)\d x\bigg|\leq {\aleph}\|\u\|^2_{\mathbb{L}^2(\mathcal{O})}, \ \ \text{ for all }\ \u\in\L^2(\mathcal{O}).
		\end{align*}  
	 We point out here that the above-mentioned condition on $\h$ is not required in this work due to the cut-off $F_N(\cdot)$ appearing in front of the nonlinearity $(\u\cdot\nabla)\u$ in \eqref{1}.
	\end{remark}
	
	\begin{theorem}[{\texttt{Multiplicative noise case}}]\label{MT1}
		Under Assumptions \ref{assumpO} and \ref{Hypo_f-N}, all results in  Theorem \ref{MT1-N} hold for
		the non-autonomous random dynamical system generated by \eqref{1} with $S(\u)=\u$.
	\end{theorem}

	\subsection{Kuratowski's measure of non-compactness} The first result on measure of non-compactness was defined and studied by Kuratowski in \cite{Kuratowski}. With the help of some vital implications of Kuratowski's measure of non-compactness, one can show the existence of a convergent subsequence for some arbitrary sequences. Therefore, several authors used such results to obtain the asymptotic compactness of random dynamical systems, cf. \cite{CGTW,Kinra+Mohan+Wang_2024,Wang+Kinra+Mohan_2023}, etc. and references therein.
	\begin{definition}[Kuratowski's measure of non-compactness, \cite{Rakocevic}]
		Let $(\mathbb{X}, d)$ be a metric space and $E$ a bounded subset of $\mathbb{X}$. Then the Kuratowski measure of non-compactness (the set-measure of non-compactness) of $E$ denoted by  $\kappa_{\mathbb{X}}(E)$ is defined by
		\begin{align*}
			\kappa_{\mathbb{X}}(E)=\inf\bigg\{\varepsilon>0:E\subset\bigcup\limits_{i=1}^{n}Q_{i}, \ Q_{i}\subset\mathbb{X},\ \mathrm{diam}(Q_{i})<\varepsilon\ \ (i=1,2\ldots,n; \ n\in\N)\bigg\}.
		\end{align*}
		The function $\kappa_{\mathbb{X}}$ is called \emph{Kuratowski's measure of non-compactness}.
	\end{definition}
	Note that $	\kappa_{\mathbb{X}}(E)=0$ if and only if $\overline{E}$ is compact (see  \cite[Lemma 1.2]{Rakocevic}). The following lemma is an application of Kuratowski's measure of non-compactness which is helpful in proving the time-semi-uniform asymptotic compactness of random dynamical systems.
	\begin{lemma}[Lemma 2.7, \cite{LGL}]\label{K-BAC}
		Let $\mathbb{X}$ be a Banach space and $x_n$ be an arbitrary sequence in $\mathbb{X}$. Then $\{x_n\}$ has a convergent subsequence if $\kappa_{\mathbb{X}}\{x_n:n\geq m\}\to 0 \text{ as } m\to \infty.$
	\end{lemma}

	\section{Asymptotically Autonomous Robustness of Random Attractors for \eqref{1}: Additive Noise}\label{sec3}\setcounter{equation}{0}
	In this section, we consider the 3D SGMNSE \eqref{SNSE} driven by additive white noise, that is, $S(\u)$ is independent of $\u,$ and establish the existence and asymptotic autonomous robustness of $\mathfrak{D}$-pullback random attractors. Let us consider the 3D SGMNSE perturbed by additive white noise for $t\geq \tau,$ $\tau\in\mathbb{R}$ and $\h\in\D(\A)$ as
	\begin{equation}\label{SNSE-A}
		\left\{
		\begin{aligned}
			\frac{\d\u}{\d t}+\nu \A\u+\B_N(\u)&=\mathcal{P}\f +\h\frac{\d \W}{\d t} , \\
			\u(\tau)&=\u_{0},
		\end{aligned}
		\right.
	\end{equation}
	where $\W(t,\omega)$ is the standard scalar Wiener process on the probability space $(\Omega, \mathscr{F}, \mathbb{P})$ (see Section \ref{2.5} above).
	
	Let us define $\v(t,\tau,\omega,\v_{\tau}):=\u(t,\tau,\omega,\u_{\tau})-\h z(\vartheta_{t}\omega)$, where $z$ is defined by \eqref{OU1} and satisfies \eqref{OU2}, and $\u$ is the solution of \eqref{1} with $S(\u)=\h$. Then $\v$ satisfies:
	\begin{equation}\label{2-A}
		\left\{
		\begin{aligned}
			\frac{\d\v}{\d t}-\nu \Delta\v&+F_N(\|\v+\h z\|_{\V})\cdot \left[\big((\v+\h z)\cdot\nabla\big)(\v+\h z)\right]+\nabla p\\&=\f +\sigma\h z+\nu z\Delta\h,  &&  \text{ in }\  \mathcal{O}\times(\tau,\infty), \\ \nabla\cdot\v&=0, && \text{ in } \ \ \mathcal{O}\times[\tau,\infty),\\ \v&=0, &&  \text{ on } \ \ \partial\mathcal{O}\times[\tau,\infty), \\
			\v(\tau)&=\v_{0}=\u_{0}-\h z(\vartheta_{\tau}\omega),  &&  \ x\in \mathcal{O},
		\end{aligned}
		\right.
	\end{equation}
	as well as the projected form in $\V^*$:
	\begin{equation}\label{CNSE-A}
		\left\{
		\begin{aligned}
			\frac{\d\v}{\d t} +\nu \A\v+ \B_N(\v+\h z)&= \mathcal{P}\boldsymbol{f} + \sigma\h z -\nu z\A\h \\
			\v(\tau)&=\v_{0}=\u_{0} -\h z(\vartheta_{\tau}\omega).
		\end{aligned}
		\right.
	\end{equation}
	
	In the next lemma, we discuss the solvability result of the system \eqref{CNSE-A}.
	
	\begin{lemma}\label{Soln-N}
		Suppose that $\f\in\mathrm{L}^2_{\mathrm{loc}}(\R;\L^2(\mathcal{O}))$. For each $(\tau,\omega,\v_{\tau})\in\R\times\Omega\times\H$, the system \eqref{CNSE-A} has a unique solution $\v(\cdot,\tau,\omega,\v_{\tau})\in\mathrm{C}([\tau,+\infty);\H)\cap\mathrm{L}^2_{\mathrm{loc}}(\tau,+\infty;\V)$ such that $\v$ is continuous with respect to the  initial data.
	\end{lemma}
	\begin{proof}
		Since the system \eqref{CNSE-A} is a deterministic system for each $\omega\in\Omega$, the proof of this lemma can be done by following similar  steps used in the work \cite{Kinra+Mohan_UP}, etc.
	\end{proof}

	\subsection{Lusin continuity and measurability of random dynamics system}
	The Lusin continuity of the solutions assist us to define the non-autonomous random dynamical system (NRDS). The following lemma gives us the energy inequality of solutions to the system \eqref{CNSE-A}, which will be frequently used.
	\begin{lemma}
		Let $\f\in\mathrm{L}^2_{\mathrm{loc}}(\R;\L^2(\mathcal{O}))$ and Assumption \ref{assumpO} holds. Then, the solution of the system \eqref{CNSE-A} satisfies the following inequality:
		\begin{align}\label{EI1-N}
			&\frac{\d}{\d t}\|\v(t)\|^2_{\H}+ \nu\lambda\|\v(t)\|^2_{\H}+\frac{\nu}{2}\|\v(t)\|^2_{\V} \leq \widehat{R}_4\left[\|\f(t)\|^{2}_{\L^2(\mathcal{O})}+\left|z(\vartheta_{t}\omega)\right|^2\right],
		\end{align}
		for $t\geq\tau$ a.e., where $\widehat{R}_4>0$ is some constant.
	\end{lemma}

	\begin{proof}
		From \eqref{CNSE-A}, we obtain
		\begin{align}\label{ue17-N}
			\frac{1}{2}\frac{\d}{\d t}\|\v\|^2_{\H} &=-\nu\|\v\|^2_{\V}- F_{N}(\|\v+\h z(\vartheta_{t}\omega)\|_{\V})\cdot b(\v+\h z(\vartheta_{t}\omega),\v+\h z(\vartheta_{t}\omega),\v)
			\nonumber\\ & \quad +\left(\f,\v\right)+z(\vartheta_{t}\omega)\left(\sigma\h-\nu \A\h,\v\right).
		\end{align}
		Making use of \eqref{b0},  we find the existence of a constant $\widehat{R}_1>0$ such that
		\begin{align}\label{ue18-N}
			&F_{N}(\|\v+\h z(\vartheta_{t}\omega)\|_{\V})\cdot	\left|b(\v+\h z(\vartheta_{t}\omega),\v+\h z(\vartheta_{t}\omega),\v)\right|
			\nonumber\\&= F_{N}(\|\v+\h z(\vartheta_{t}\omega)\|_{\V})\cdot	\left|b(\v+\h z(\vartheta_{t}\omega),\v+\h z(\vartheta_{t}\omega),\h z(\vartheta_{t}\omega))\right|
			\nonumber\\ & \leq C F_{N}(\|\v+\h z(\vartheta_{t}\omega)\|_{\V})\cdot	\|\v+\h z(\vartheta_{t}\omega)\|_{\V}\|\v+\h z(\vartheta_{t}\omega)\|_{\V} \|\h\|_{\V} |z(\vartheta_{t}\omega)|
			\nonumber\\ & \leq C N  \|\v+\h z(\vartheta_{t}\omega)\|_{\V} \|\h\|_{\V} |z(\vartheta_{t}\omega)|
			\nonumber\\ & \leq \frac{\nu}{16}  \|\v+\h z(\vartheta_{t}\omega)\|^2_{\V}+ CN^2 \|\h\|^2_{\V} |z(\vartheta_{t}\omega)|^2
			\nonumber\\ & \leq \frac{\nu}{8}  \|\v\|^2_{\V} + \widehat{R}_1(1+N^2) \|\h\|^2_{\V} |z(\vartheta_{t}\omega)|^2,
		\end{align}
		where we have also used \eqref{b0}, \eqref{HI+SE}, \eqref{poin}, \eqref{FN1} and Young's inequality. Using \eqref{poin}, H\"older's and Young's inequalities, there exist constants $\widehat{R}_2,\widehat{R}_3>0$ such that
		\begin{align}
			\left(\f,\v\right)+z(\vartheta_{t}\omega)\big(\sigma\h-\nu \A\h,\v\big)\leq\frac{\nu\lambda}{8}\|\v\|^2_{\H}+\widehat{R}_2\|\f\|^{2}_{\L^2(\mathcal{O})}+\widehat{R}_3 \left|z(\vartheta_{t}\omega)\right|^2.\label{ue19-N}
		\end{align}
		 Combining \eqref{ue17-N}-\eqref{ue19-N} we have
		\begin{align*}
		\frac{\d}{\d t}\|\v\|^2_{\H}+\frac{3\nu}{2}\|\v\|^2_{\V} &\leq 2\widehat{R}_2\|\f\|^{2}_{\L^2(\mathcal{O})}+2\left[\widehat{R}_1(1+N^2) \|\h\|^2_{\V}+\widehat{R}_3\right]|z(\vartheta_{t}\omega)|^2.
		\end{align*}
		Along with \eqref{poin}, the above inequality implies 
		\begin{align*}
		\frac{\d}{\d t}\|\v\|^2_{\H}+\nu\lambda\|\v\|^2_{\H}+\frac{\nu}{2}\|\v\|^2_{\V} &\leq \widehat{R}_4\left[\|\f\|^{2}_{\L^2(\mathcal{O})}+\left|z(\vartheta_{t}\omega)\right|^2\right], 
		\end{align*}	
		where $\widehat{R}_4=2\max\{\widehat{R}_1(1+N^2) \|\h\|^2_{\V} +\widehat{R}_3,\widehat{R}_2\}.$ This completes the proof of the estimate \eqref{EI1-N}.
	\end{proof}

	The next result shows the Lusin continuity of solution mappings of the system \eqref{CNSE-A} with respect to sample points.
	\begin{proposition}\label{LusinC-N}
		Let $\f\in\mathrm{L}^2_{\mathrm{loc}}(\R;\L^2(\mathcal{O}))$ and Assumption \ref{assumpO} holds. For each $N\in\N$, the mapping $\omega\mapsto\v(t,\tau,\omega,\v_{\tau})$ $($solution of \eqref{CNSE-A}$)$ is continuous from $(\Omega_{M},d_{\Omega_M})$ to $\H$, uniformly in $t\in[\tau,\tau+T]$ with $T>0.$
	\end{proposition}
	\begin{proof}
		Assume $\omega_k,\omega_0\in\Omega_M$ are such that $d_{\Omega_M}(\omega_k,\omega_0)\to0$ as $k\to\infty$. Let $\mathscr{Y}^k(\cdot):=\v^k(\cdot)-\v^0(\cdot),$ where $\v^k(\cdot)=\v(\cdot,\tau,\omega_k,\v_{\tau})$ and $\v_0(\cdot)=\v(\cdot,\tau,\omega_0,\v_{\tau})$. Then, $\mathscr{Y}^k(\cdot)$ satisfies:
		\begin{align}\label{LC1-N}
			\frac{\d\mathscr{Y}^k}{\d t}
			&=-\nu \A\mathscr{Y}^k -\left[\B_N\big(\v^k+z(\vartheta_{t}\omega_k)\h\big)-\B_N\big(\v^0+z(\vartheta_{t}\omega_0)\h\big)\right]\nonumber\\&\quad +\left\{\sigma\h-\nu\A\h\right\}\left[z(\vartheta_t\omega_k)-z(\vartheta_t\omega_0)\right]
			\nonumber\\ &=-\nu \A\mathscr{Y}^k - F_{N}(\|\v^k+z(\vartheta_{t}\omega_k)\h\|_{\V}) \cdot  \B\big(\mathscr{Y}^k+\left[z(\vartheta_t\omega_k)-z(\vartheta_t\omega_0)\right]\h, \v^k+z(\vartheta_{t}\omega_k)\h\big)
			\nonumber\\ & \quad - \big[ F_{N}(\|\v^k+z(\vartheta_{t}\omega_k)\h\|_{\V}) - F_{N}(\|\v^0+z(\vartheta_{t}\omega_0)\h\|_{\V})\big]\cdot \B\big( \v^k+z(\vartheta_{t}\omega_k)\h, \v^0+z(\vartheta_{t}\omega_0)\h \big)
			\nonumber\\ & \quad - F_{N}(\|\v^0+z(\vartheta_{t}\omega_0)\h\|_{\V}) \cdot \B\big(\v^0+z(\vartheta_{t}\omega_0)\h,\mathscr{Y}^k+\left[z(\vartheta_t\omega_k)-z(\vartheta_t\omega_0)\right]\h\big)
			\nonumber\\& \quad +\left\{\sigma\h-\nu\A\h\right\}\left[z(\vartheta_t\omega_k)-z(\vartheta_t\omega_0)\right],
		\end{align}
		in $\V^*$, where we have used \eqref{BN-diff} in the final equality.  Taking the  inner product with $\mathscr{Y}^k$ in \eqref{LC1-N} and using \eqref{b0}, we infer
		\begin{align}\label{LC2-N}
			&	\frac{1}{2}\frac{\d }{\d t}\|\mathscr{Y}^k\|^2_{\H}
			\nonumber\\ &=-\nu \|\mathscr{Y}^k\|_{\H}^2 - F_{N}(\|\v^k+z(\vartheta_{t}\omega_k)\h\|_{\V}) \cdot b\big(\mathscr{Y}^k+\left[z(\vartheta_t\omega_k)-z(\vartheta_t\omega_0)\right]\h, \v^k+z(\vartheta_{t}\omega_k)\h,\mathscr{Y}^k\big)
			\nonumber\\ & \quad - \big[ F_{N}(\|\v^k+z(\vartheta_{t}\omega_k)\h\|_{\V}) - F_{N}(\|\v^0+z(\vartheta_{t}\omega_0)\h\|_{\V})\big] \cdot b\big( \v^k+z(\vartheta_{t}\omega_k)\h, \v^0+z(\vartheta_{t}\omega_0)\h, \mathscr{Y}^k \big)
			\nonumber\\ & \quad - F_{N}(\|\v^0+z(\vartheta_{t}\omega_0)\h\|_{\V}) \cdot b\big(\v^0+z(\vartheta_{t}\omega_0)\h,\left[z(\vartheta_t\omega_k)-z(\vartheta_t\omega_0)\right]\h, \mathscr{Y}^k \big)
			\nonumber\\& \quad +\left\{\sigma\h-\nu\A\h\right\}\left[z(\vartheta_t\omega_k)-z(\vartheta_t\omega_0)\right].
		\end{align}
		Next we estimate each term on the right hand side of \eqref{LC2-N}. Using \eqref{HI+SE}, \eqref{FN1} and Young's inequality, we obtain 
		\begin{align}\label{LC3-N}
			& \big|F_{N}(\|\v^k+z(\vartheta_{t}\omega_k)\h\|_{\V}) \cdot b\big(\mathscr{Y}^k+\left[z(\vartheta_t\omega_k)-z(\vartheta_t\omega_0)\right]\h, \v^k+z(\vartheta_{t}\omega_k)\h,\mathscr{Y}^k\big)\big|
			\nonumber\\ & \leq F_{N}(\|\v^k+z(\vartheta_{t}\omega_k)\h\|_{\V}) \cdot \|\mathscr{Y}^k+\left[z(\vartheta_t\omega_k)-z(\vartheta_t\omega_0)\right]\h\|_{\V} \|\v^k+z(\vartheta_{t}\omega_k)\h\|_{\V} \|\mathscr{Y}^k\|_{\H}^{\frac{1}{2}}\|\mathscr{Y}^k\|_{\V}^{\frac{1}{2}}
			\nonumber\\ & \leq N \|\mathscr{Y}^k+\left[z(\vartheta_t\omega_k)-z(\vartheta_t\omega_0)\right]\h\|_{\V}  \|\mathscr{Y}^k\|_{\H}^{\frac{1}{2}}\|\mathscr{Y}^k\|_{\V}^{\frac{1}{2}}
			\nonumber\\ & \leq N \|\mathscr{Y}^k\|_{\H}^{\frac{1}{2}}\|\mathscr{Y}^k\|_{\V}^{\frac{3}{2}} + N \left|z(\vartheta_t\omega_k)-z(\vartheta_t\omega_0)\right|\|\h\|_{\V}  \|\mathscr{Y}^k\|_{\H}^{\frac{1}{2}}\|\mathscr{Y}^k\|_{\V}^{\frac{1}{2}}
			\nonumber\\ & \leq \frac{\nu}{8}\|\mathscr{Y}^k\|_{\V}^{2} + C N^2 \left[ N^2   +  \left|z(\vartheta_t\omega_k)-z(\vartheta_t\omega_0)\right|^2\|\h\|^2_{\V}\right]  \|\mathscr{Y}^k\|_{\H}^{2}.
		\end{align}
		Using \eqref{HI+SE}, \eqref{FN1}-\eqref{FN2} and Young's inequality, we find 
		\begin{align}\label{LC4-N}
			& \left|\big[ F_{N}(\|\v^k+z(\vartheta_{t}\omega_k)\h\|_{\V}) - F_{N}(\|\v^0+z(\vartheta_{t}\omega_0)\h\|_{\V})\big] \cdot b\big( \v^k+z(\vartheta_{t}\omega_k)\h, \v^0+z(\vartheta_{t}\omega_0)\h, \mathscr{Y}^k \big)\right|
			\nonumber\\ & \leq \frac{C}{N} F_{N}(\|\v^k+z(\vartheta_{t}\omega_k)\h\|_{\V}) F_{N}(\|\v^0+z(\vartheta_{t}\omega_0)\h\|_{\V}) \cdot \|\mathscr{Y}^k+\left[z(\vartheta_t\omega_k)-z(\vartheta_t\omega_0)\right]\h\|_{\V} 
			\nonumber\\ & \quad \times \|\v^k+z(\vartheta_{t}\omega_k)\h\|_{\V} \|\v^0+z(\vartheta_{t}\omega_0)\h\|_{\V} \|\mathscr{Y}^k\|_{\H}^{\frac{1}{2}}\|\mathscr{Y}^k\|_{\V}^{\frac{1}{2}}
			\nonumber\\ & \leq C N  \|\mathscr{Y}^k+\left[z(\vartheta_t\omega_k)-z(\vartheta_t\omega_0)\right]\h\|_{\V} 
			\|\mathscr{Y}^k\|_{\H}^{\frac{1}{2}}\|\mathscr{Y}^k\|_{\V}^{\frac{1}{2}}
			\nonumber\\ & \leq \frac{\nu}{8}\|\mathscr{Y}^k\|_{\V}^{2} + C N^2 \left[ N^2   +  \left|z(\vartheta_t\omega_k)-z(\vartheta_t\omega_0)\right|^2\|\h\|^2_{\V}\right]  \|\mathscr{Y}^k\|_{\H}^{2}.
		\end{align}
		Using \eqref{HI+SE}, \eqref{FN1} and Young's inequality, we get  
		\begin{align}\label{LC5-N}
			&\left|F_{N}(\|\v^0+z(\vartheta_{t}\omega_0)\h\|_{\V}) \cdot b\big(\v^0+z(\vartheta_{t}\omega_0)\h,\left[z(\vartheta_t\omega_k)-z(\vartheta_t\omega_0)\right]\h, \mathscr{Y}^k \big)\right|
			\nonumber\\ & \leq F_{N}(\|\v^0+z(\vartheta_{t}\omega_0)\h\|_{\V})\cdot \|\v^0+z(\vartheta_{t}\omega_0)\h\|_{\V} \left|z(\vartheta_t\omega_k)-z(\vartheta_t\omega_0)\right|\|\h\|_{\V} \|\mathscr{Y}^k\|_{\H}^{\frac{1}{2}}\|\mathscr{Y}^k\|_{\V}^{\frac{1}{2}}
			\nonumber\\ & \leq N \left|z(\vartheta_t\omega_k)-z(\vartheta_t\omega_0)\right|\|\h\|_{\V} \|\mathscr{Y}^k\|_{\H}^{\frac{1}{2}}\|\mathscr{Y}^k\|_{\V}^{\frac{1}{2}}
			\nonumber\\ & \leq \frac{\nu}{8}\|\mathscr{Y}^k\|_{\V}^{2} + C N^2 \left|z(\vartheta_t\omega_k)-z(\vartheta_t\omega_0)\right|^2\|\h\|^2_{\V}   \|\mathscr{Y}^k\|_{\H}^{2}.
		\end{align}
		Finally, H\"older's and Young's inequalities imply
		\begin{align}\label{LC6-N}
			&\left|z(\vartheta_t\omega_k)-z(\vartheta_t\omega_0)\right|	|\big(\sigma\h-\nu\A\h,\mathscr{Y}^k\big)|
			\nonumber\\ & \leq C  [\|\h\|_{\H} +\|\h\|_{\D(\A)}]\left|z(\vartheta_t\omega_k)-z(\vartheta_t\omega_0)\right|\|\mathscr{Y}^k\|_{\H}
			\nonumber\\ & \leq \frac{\nu\lambda}{8}\|\mathscr{Y}^k\|_{\H}^{2} + C \left|z(\vartheta_t\omega_k)-z(\vartheta_t\omega_0)\right|^2 [\|\h\|^2_{\H} +\|\h\|^2_{\D(\A)}]
			\nonumber\\ & \leq \frac{\nu}{8}\|\mathscr{Y}^k\|_{\V}^{2} + C \left|z(\vartheta_t\omega_k)-z(\vartheta_t\omega_0)\right|^2 [\|\h\|^2_{\H} +\|\h\|^2_{\D(\A)}].
		\end{align}
		Combining \eqref{LC2-N}-\eqref{LC6-N}, we reach at
		\begin{align}\label{LC8-N}
			\frac{\d }{\d t}\|\mathscr{Y}^k(t)\|^2_{\H}\leq P_k(t)\|\mathscr{Y}^k(t)\|^2_{\H}+Q_k(t),
		\end{align}
		for a.e. $t\in[\tau,\tau+T]$, $T>0$, where $P_k(t)=C N^2 \left[ N^2   +  \left|z(\vartheta_t\omega_k)-z(\vartheta_t\omega_0)\right|^2\|\h\|^2_{\V}\right]$ and
		\begin{align*}
			Q_k&=C \left|z(\vartheta_t\omega_k)-z(\vartheta_t\omega_0)\right|^2 [\|\h\|^2_{\H} +\|\h\|^2_{\D(\A)}].
		\end{align*}
		Now, from Lemma \ref{conv_z} and the fact that $\h\in\D(\A)$, we conclude that
		\begin{align}\label{LC13-N}
			\sup_{k\in\N}\int_{\tau}^{\tau+T}P_{k}(t)\d t\leq C(\tau,T,N,\omega_0)\ \  \text{ and }\ \ 	\lim_{k\to+\infty}\int_{\tau}^{\tau+T}Q_k(t)\d t=0.
		\end{align}
		In view of \eqref{LC13-N}, we apply Gronwall's inequality to \eqref{LC8-N} and complete the proof.
	\end{proof}
	
	Making use of Lemma \ref{Soln-N}, we can define a mapping $\Phi:\R^+\times\R\times\Omega\times\H\to\H$ by
	\begin{align}\label{Phi-N}
		\Phi(t,\tau,\omega,\u_{\tau})=\u(t+\tau,\tau,\vartheta_{-\tau}\omega,\u_{\tau})=\v(t+\tau,\tau,\vartheta_{-\tau}\omega,\v_{\tau})+\h z(\vartheta_{t}\omega).
	\end{align}
	The Lusin continuity in Proposition \ref{LusinC-N} provides the $\mathscr{F}$-measurability of $\Phi$. Consequently, Lemma \ref{Soln-N} and Proposition \ref{LusinC-N} imply that the mapping $\Phi$ defined by \eqref{Phi-N} is an NRDS on $\H$.

	\subsection{Backward convergence of NRDS}
	Consider the following autonomous 3D SGMNSE driven by additive white noise:
	\begin{equation}\label{A-SNSE}
		\left\{
		\begin{aligned}
			\frac{\d\widetilde{\u}}{\d t}+\nu \A\widetilde{\u}+\B_{N}(\widetilde{\u})&=\mathcal{P}\f_{\infty} +\h\frac{\d \W}{\d t}, \\
			\widetilde{\u}(0)&=\widetilde{\u}_{0}.
		\end{aligned}
		\right.
	\end{equation}
	We show that the solution of the system \eqref{CNSE-A} converges to the solution of the corresponding autonomous system \eqref{A-SNSE} as $\tau\to-\infty$. Let $\widetilde{\v}(t,\omega)=\widetilde{\u}(t,\omega)-\h(x)z(\vartheta_{t}\omega)$. Then, the pathwise deterministic system corresponding to the stochastic system \eqref{A-SNSE} is given by:
	\begin{equation}\label{A-CNSE}
		\left\{
		\begin{aligned}
			\frac{\d\widetilde{\v}}{\d t} +\nu \A\widetilde{\v} + \B_{N}(\widetilde{\v}+\h z(\vartheta_{t}\omega)) &= \mathcal{P}{\boldsymbol{f}}_{\infty} + \sigma\h z(\vartheta_{t}\omega) -\nu z(\vartheta_{t}\omega)\A\h, \\
			\widetilde{\v}(0)&=\widetilde{\v}_{0} =\widetilde{\u}_{0}-\h z(\omega), 
		\end{aligned}
		\right.
	\end{equation}
	in $\V^*$. 
	\begin{proposition}\label{Back_conver-N}
		Suppose that Assumptions \ref{assumpO} and \ref{Hypo_f-N} are satisfied. Then the solution $\v$ of the system \eqref{CNSE-A} backward converges to the solution $\widetilde{\v}$ of the system \eqref{A-CNSE}, that is,
		\begin{align*}
			\lim_{\tau\to -\infty}\|\v(T+\tau,\tau,\vartheta_{-\tau}\omega,\v_{\tau})-\widetilde{\v}(t,\omega,\widetilde{\v}_0)\|_{\H}=0, \ \ \text{ for all } T>0 \text{ and } \omega\in\Omega,
		\end{align*}
		whenever $\|\v_{\tau}-\widetilde{\v}_0\|_{\H}\to0$ as $\tau\to-\infty.$
	\end{proposition}
	
	\begin{proof}
		Let $\mathscr{Y}^{\tau}(t):=\v(t+\tau,\tau,\vartheta_{-\tau}\omega,\v_{\tau})-\widetilde{\v}(t,\omega,\widetilde{\v}_0)$ for $t\geq0$. From \eqref{CNSE-A} and \eqref{A-CNSE}, we have
		\begin{align}\label{BC1-A}
			\frac{\d\mathscr{Y}^{\tau}}{\d t}
			&=-\nu \A\mathscr{Y}^{\tau}-\left[\B_{N}\big(\v+\h z(\vartheta_{t}\omega)\big)-\B_{N}\big(\widetilde{\v}+\h z(\vartheta_{t}\omega)\big)\right]+\left[\mathcal{P}\f(t+\tau)-\mathcal{P}\f_{\infty}\right]
			\nonumber\\ &=-\nu \A\mathscr{Y}^{\tau}-F_{N}(\|\v+\h z(\vartheta_{t}\omega)\|_{\V}) \cdot  \B (\mathscr{Y}^{\tau},\v+\h z(\vartheta_{t}\omega) )
			\nonumber\\ & \quad - \big[ F_{N}(\|\v+\h z(\vartheta_{t}\omega)\|_{\V}) - F_{N}(\|\widetilde{\v}+\h z(\vartheta_{t}\omega)\|_{\V})\big]\cdot \B ( \widetilde{\v}+\h z(\vartheta_{t}\omega) , \v+\h z(\vartheta_{t}\omega) )
			\nonumber\\ & \quad
			- F_{N}(\|\widetilde{\v}+\h z(\vartheta_{t}\omega)\|_{\V}) \cdot \B(\widetilde{\v}+\h z(\vartheta_{t}\omega),\mathscr{Y}^{\tau})+\left[\mathcal{P}\f(t+\tau)-\mathcal{P}\f_{\infty}\right],
		\end{align}
		in $\V^{\ast}$ (in the weak sense), where we have also used \eqref{BN-diff} in the final equality. From \eqref{BC1-A}, we infer
		\begin{align}\label{BC2-A}
			& \frac{1}{2}\frac{\d }{\d t}\|\mathscr{Y}^{\tau}\|^2_{\H}
			\nonumber\\ &=-\nu\|\mathscr{Y}^{\tau}\|^2_{\V}    
			-F_{N}(\|\v+\h z(\vartheta_{t}\omega)\|_{\V}) \cdot  b (\mathscr{Y}^{\tau},\v+\h z(\vartheta_{t}\omega),\mathscr{Y}^{\tau} )
			\nonumber\\ & \quad - \big[ F_{N}(\|\v+\h z(\vartheta_{t}\omega)\|_{\V}) - F_{N}(\|\widetilde{\v}+\h z(\vartheta_{t}\omega)\|_{\V})\big]\cdot b ( \widetilde{\v}+\h z(\vartheta_{t}\omega) , \v+\h z(\vartheta_{t}\omega),\mathscr{Y}^{\tau} )
			\nonumber\\ & \quad
			+(\f(t+\tau)-\f_{\infty},\mathscr{Y}^{\tau}).
		\end{align}
		Next, we estimate all the terms on the right hand side of \eqref{BC2-A} one by one. Using \eqref{HI+SE}, \eqref{FN1} and Young's inequality, we find
		\begin{align}
			&\left|F_{N}(\|\v+\h z(\vartheta_{t}\omega)\|_{\V}) \cdot  b (\mathscr{Y}^{\tau},\v+\h z(\vartheta_{t}\omega),\mathscr{Y}^{\tau} )\right|
			\nonumber\\ & \leq F_{N}(\|\v+\h z(\vartheta_{t}\omega)\|_{\V})\cdot \|\mathscr{Y}^{\tau}\|_{\V} \|\v+\h z(\vartheta_{t}\omega)\|_{\V} \|\mathscr{Y}^{\tau}\|^{\frac12}_{\H} \|\mathscr{Y}^{\tau}\|^{\frac12}_{\V}
			\nonumber\\ & \leq N \|\mathscr{Y}^{\tau}\|^{\frac12}_{\H} \|\mathscr{Y}^{\tau}\|^{\frac32}_{\V}
			\leq \frac{\nu}{6}\|\mathscr{Y}^{\tau}\|^2_{\V} + C N^4 \|\mathscr{Y}^{\tau}\|^2_{\H}.
		\end{align}
		Using \eqref{HI+SE}, \eqref{FN1}-\eqref{FN2} and Young's inequality, we obtain
		\begin{align}
			&\left|\big[ F_{N}(\|\v+\h z(\vartheta_{t}\omega)\|_{\V}) - F_{N}(\|\widetilde{\v}+\h z(\vartheta_{t}\omega)\|_{\V})\big]\cdot b ( \widetilde{\v}+\h z(\vartheta_{t}\omega) , \v+\h z(\vartheta_{t}\omega),\mathscr{Y}^{\tau} )\right|
			\nonumber\\ & \leq \frac{C}{N} F_{N}(\|\v+\h z(\vartheta_{t}\omega)\|_{\V}) F_{N}(\| \widetilde{\v}+\h z(\vartheta_{t}\omega)\|_{\V}) \cdot \|\v+\h z(\vartheta_{t}\omega)\|_{\V}\| \widetilde{\v}+\h z(\vartheta_{t}\omega)\|_{\V} \|\mathscr{Y}^{\tau}\|^{\frac12}_{\H} \|\mathscr{Y}^{\tau}\|^{\frac12}_{\V}
			\nonumber\\ & \leq C N \|\mathscr{Y}^{\tau}\|^{\frac12}_{\H} \|\mathscr{Y}^{\tau}\|^{\frac12}_{\V} \leq \frac{\nu}{6}\|\mathscr{Y}^{\tau}\|^2_{\V} + C N^2 \|\mathscr{Y}^{\tau}\|^2_{\H}.
		\end{align}
		Applying  H\"older's and Young's inequalities, and \eqref{poin}, we deduce
		\begin{align}\label{BC5-A}
			\left|(\f(t+\tau)-\f_{\infty},\mathscr{Y}^{\tau})\right|&\leq C\|\f(t+\tau)-\f_{\infty}\|^2_{\L^2(\mathcal{O})}+\frac{\nu\lambda}{6}\|\mathscr{Y}^{\tau}\|^2_{\H} \nonumber\\&\leq C\|\f(t+\tau)-\f_{\infty}\|^2_{\L^2(\mathcal{O})}+\frac{\nu}{6}\|\mathscr{Y}^{\tau}\|^2_{\V}.
		\end{align}
		Combining \eqref{BC2-A}-\eqref{BC5-A}, we arrive at
		\begin{align}\label{BC6-A}
			\frac{\d }{\d t}\|\mathscr{Y}^{\tau}\|^2_{\H}+\nu\|\mathscr{Y}^{\tau}\|^2_{\V}\leq CN^2(1+N^2) \|\mathscr{Y}^{\tau}\|^2_{\H}+C\|\f(t+\tau)-\f_{\infty}\|^2_{\L^2(\mathcal{O})}.
		\end{align}
		Applying Gronwall's inequality in \eqref{BC6-A} over $(0,T)$, we obtain
		\begin{align*}
			\|\mathscr{Y}^{\tau}(T)\|^2_{\H}\leq \left[\|\mathscr{Y}^{\tau}(0)\|^2_{\H}+ \int_{0}^{T}\|\f(t+\tau)-\f_{\infty}\|^2_{\L^2(\mathcal{O})} \d t\right]e^{CN^2(1+N^2) T}.
		\end{align*}
		Using \eqref{BC8-A} and the fact that $\|\mathscr{Y}^{\tau}(0)\|^2_{\H}=\|\v_{\tau}-\widetilde{\v}_0\|_{\H}\to0$ as $\tau\to-\infty$, we conclude the proof.
	\end{proof}

	\subsection{Increasing random absorbing sets}
	This subsection provides the existence of an increasing ${\mathfrak{D}}$-random absorbing set for the 3D non-autonomous SGMNSE.
	
	\begin{lemma}\label{Absorbing-N}
		Suppose that $\f\in\mathrm{L}^2_{\mathrm{loc}}(\R;\L^2(\mathcal{O}))$ and Assumption \ref{assumpO} is satisfied. Then, for all $(\tau,\omega)\in\R\times\Omega,$ $s\leq\tau$, $\xi\geq s-t,$ $ t\geq0$ and $\v_{0}\in\H$,
		\begin{align}\label{AB2-N}
			&\|\v(\xi,s-t,\vartheta_{-s}\omega,\v_{0})\|^2_{\H}+\frac{\nu}{2}\int_{s-t}^{\xi}e^{\nu\lambda(\zeta-\xi)}\|\v(\zeta,s-t,\vartheta_{-s}\omega,\v_{0})\|^2_{\V}\d\zeta\nonumber\\&\leq e^{-\nu\lambda(\xi-s+t)}\|\v_{0}\|^2_{\H} +\widehat{R}_4\int_{-t}^{\xi-s}e^{\nu\lambda(\zeta+s-\xi)}\bigg\{\|\f(\zeta+s)\|^2_{\L^2(\mathcal{O})}+\left|z(\vartheta_{\zeta}\omega)\right|^2\bigg\}\d\zeta,
		\end{align}
		where $\widehat{R}_4$ is the same as in \eqref{EI1-N}. For each $(\tau,\omega,D)\in\R\times\Omega\times{\mathfrak{D}},$ there exists a time $\mathfrak{T}:=\mathfrak{T}(\tau,\omega,D)>0$ such that
		\begin{align}\label{AB1-N}
			&\sup_{s\leq\tau}\sup_{t\geq \mathfrak{T}}\sup_{\v_{0}\in D(s-t,\vartheta_{-t}\omega)}\bigg[\|\v(s,s-t,\vartheta_{-s}\omega,\v_{0})\|^2_{\H} +\frac{\nu}{2}\int_{s-t}^{s}e^{\nu\lambda(\zeta-s)}\|\v(\zeta,s-t,\vartheta_{-s}\omega,\v_{0})\|^2_{\V}\d\zeta\bigg]\nonumber\\&\leq 1+\widehat{R}_4 \sup_{s\leq \tau}R(s,\omega),
		\end{align}
		where $R(s,\omega)$ is given by
		\begin{align}\label{ABr-N}
			R(s,\omega):=\int_{-\infty}^{0}e^{\nu\lambda\zeta}\bigg\{\|\f(\zeta+s)\|^2_{\L^2(\mathcal{O})}+\left|z(\vartheta_{\zeta}\omega)\right|^2\bigg\}\d\zeta.
		\end{align}
	\end{lemma}
	\begin{proof}
		Let us rewrite the energy inequality \eqref{EI1-N} for $\v(\zeta)=\v(\zeta,s-t,\vartheta_{-s}\omega,\v_{0})$ as
		\begin{align}\label{AB-EI}
			&\frac{\d}{\d \zeta}\|\v(\zeta)\|^2_{\H}+ \nu\lambda \|\v(\zeta)\|^2_{\H}+\frac{\nu}{2}\|\v(\zeta)\|^2_{\V} \leq \widehat{R}_4\left[\|\f(\zeta)\|^{2}_{\L^2(\mathcal{O})}+\left|z(\vartheta_{\zeta-s}\omega)\right|^2\right].
		\end{align}
		In view of the variation of constants formula with respect to $\zeta\in(s-t,\xi)$, we obtain \eqref{AB2-N} immediately. Putting $\xi=s$ in \eqref{AB2-N}, we obtain
		\begin{align}\label{AB4-N}
			&\|\v(s,s-t,\vartheta_{-s}\omega,\v_{0})\|^2_{\H}+\frac{\nu}{2}\int_{s-t}^{s}e^{\nu\lambda(\zeta-s)}\|\v(\zeta,s-t,\vartheta_{-s}\omega,\v_{0})\|^2_{\V}\d\zeta\nonumber\\&\leq e^{-\nu\lambda t}\|\v_{0}\|^2_{\H}  +\widehat{R}_4\int_{-t}^{0}e^{\nu\lambda\zeta}\bigg\{\|\f(\zeta+s)\|^2_{\L^2(\mathcal{O})}+\left|z(\vartheta_{\zeta}\omega)\right|^2\bigg\}\d\zeta,
		\end{align}
		for all $s\leq\tau$. 
		Since $\v_0\in D(s-t,\vartheta_{-t}\omega)$ and $D$ is backward tempered, it implies from the definition of backward temperedness \eqref{BackTem} that there exists a time ${\mathfrak{T}}:={\mathfrak{T}}(\tau,\omega,D)$ such that for all $t\geq {\mathfrak{T}}>0$,
		\begin{align}\label{v_0-1}
			e^{-\nu\lambda t}\sup_{s\leq \tau}\|D(s-t,\vartheta_{-t}\omega)\|^2_{\H}\leq1.
		\end{align}
		Taking supremum over $s\in(-\infty,\tau]$ in \eqref{AB4-N}, one obtains \eqref{AB1-N}.
	\end{proof}

	\begin{proposition}\label{IRAS-N}
		Suppose that $\f\in\mathrm{L}^2_{\mathrm{loc}}(\R;\L^2(\mathcal{O}))$ and Assumption \ref{assumpO} is satisfied. Then, we have the following results: 
		\vskip 2mm
		\noindent
		\emph{\textbf{(i)}} There is an increasing $\mathfrak{D}$-pullback absorbing set $\mathcal{R}$ given by,  for all $ \tau\in\R$
		\begin{align}\label{IRAS1-N}
			\mathcal{R}(\tau,\omega):=\left\{\u\in\H:\|\u\|^2_{\H}\leq 2+2\widehat{R}_4\sup_{s\leq \tau} R(s,\omega)+2\|\h\|^2_{\H}\left|z(\omega)\right|^2\right\}, 
		\end{align}
		where $\widehat{R}_4$ and $R(\cdot,\cdot)$ are the same as in \eqref{EI1-N} and \eqref{ABr-N}, respectively.
		Moreover, $\mathcal{R}$ is backward-uniformly  tempered with an arbitrary rate, that is, $\mathcal{R}\in{\mathfrak{D}}$.
		\vskip 2mm
		\noindent
		\emph{\textbf{(ii)}} There is a $\mathfrak{B}$-pullback \textbf{random} absorbing set $\widetilde{\mathcal{R}}$ given by, for all $ \tau\in\R$
		\begin{align}\label{IRAS11-N}
			\widetilde{\mathcal{R}}(\tau,\omega):=\left\{\u\in\H:\|\u\|^2_{\H}\leq 2+2\widehat{R}_4 R(\tau,\omega)+2\|\h\|^2_{\H}\left|z(\omega)\right|^2\right\}\in{\mathfrak{B}}. 
		\end{align}
	\end{proposition}
	\begin{proof}
		\textbf{(i)} Using \eqref{G3} and \eqref{Z3}, we derive
		\begin{align}\label{IRAS2-N}
			\sup_{s\leq \tau}R(s,\omega)&= \sup_{s\leq \tau}\int_{-\infty}^{0}
			e^{\nu\lambda\zeta}
			\bigg\{\|\f(\zeta+s)\|^2_{\L^2(\mathcal{O})}
			+\left|z(\vartheta_{\zeta}\omega)
			\right|^2\bigg\}\d\zeta <\infty.
		\end{align}
		Hence, the pullback absorption follows from Lemma \ref{Absorbing-N}. Due to the fact that $\tau\mapsto \sup\limits_{s\leq\tau}R(s,\omega)$ is an increasing function, $\mathcal{R}(\tau,\omega)$ is an increasing $\mathfrak{D}$-pullback absorbing set. For $c>0$, let $c_1=\min\{\frac{c}{2}, \nu\lambda \}$ and consider
\begin{align}\label{IRAS3-N}
			&	\lim_{t\to+\infty}e^{-ct}\sup_{s\leq \tau}\|\mathcal{R}(s-t,\vartheta_{-t}\omega)\|^2_{\H} \nonumber\\&\leq \lim_{t\to+\infty}e^{-ct}\left[2+2\widehat{R}_4\sup_{s\leq \tau} R(s-t,\vartheta_{-t}\omega)+2\|\h\|^2_{\H}\left|z(\vartheta_{-t}\omega)\right|^2\right] \nonumber\\ &= 2\widehat{R}_4\lim_{t\to+\infty}e^{-ct}\sup_{s\leq \tau}\int\limits_{-\infty}^{-t}e^{\nu\lambda(\zeta+t)}\bigg\{\|\f(\zeta+s)\|^2_{\L^2(\mathcal{O})}+\left|z(\vartheta_{\zeta}\omega)\right|^2\bigg\}\d\zeta \nonumber\\ &\leq 2\widehat{R}_4\lim_{t\to+\infty}e^{-(c-c_1)t}\sup_{s\leq \tau}\int\limits_{-\infty}^{0}e^{c_1\zeta}\bigg\{\|\f(\zeta+s)\|^2_{\L^2(\mathcal{O})}+\left|z(\vartheta_{\zeta}\omega)\right|^2\bigg\}\d\zeta\nonumber\\ &\ =0,
		\end{align}
		where we have used  \eqref{G3}, \eqref{Z3} and \eqref{IRAS2-N}. It follows from \eqref{IRAS3-N} that $\mathcal{R}\in{\mathfrak{D}}$.
		\vskip 2mm
		\noindent
		\textbf{(ii)} Since $\widetilde{\mathcal{R}}\subseteq\mathcal{R}\in\mathfrak{D}\subseteq\mathfrak{B}$ and the mapping $\omega\mapsto R(\tau,\omega)$ is $\mathscr{F}$-measurable, using \eqref{AB4-N} (for $s=\tau$), we obtain that $\widetilde{\mathcal{R}}$ is a $\mathfrak{B}$-pullback \textbf{random} absorbing set.
	\end{proof}
	\subsection{Backward uniform tail-estimates and backward flattening estimates}
	The following backward uniform tail-estimates and backward flattening estimates for the solution of the system \eqref{CNSE-A} play a key role in establishing the time-semi-uniform asymptotic compactness of the NRDS \eqref{Phi-N}. We obtain these estimates by using a proper cut-off function.
	\begin{lemma}\label{largeradius-N}
		Suppose that Assumption \ref{assumpO} is satisfied. Then, for any $(\tau,\omega,D)\in\R\times\Omega\times{\mathfrak{D}},$ the solution of \eqref{CNSE-A}  satisfies
		\begin{align}\label{ep-N}
			&\lim_{k,t\to+\infty}\sup_{s\leq \tau}\sup_{\v_{0}\in D(s-t,\vartheta_{-t}\omega)}\|\v(s,s-t,\vartheta_{-s}\omega,\v_{0})\|^2_{\L^2(\mathcal{O}^{c}_{k})}=0,
		\end{align}
		where $\mathcal{O}^c_{k}=\mathcal{O}\backslash\mathcal{O}_k$ and $\mathcal{O}_{k}=\{x\in\mathcal{O}:|x|\leq k\}$. In addition, for any $(\tau,\omega,B)\in\R\times\Omega\times{\mathfrak{B}},$ the solution of \eqref{CNSE-A}  satisfies
		\begin{align}\label{ep-N-tau}
			&\lim_{k,t\to+\infty}\sup_{\v_{0}\in B(\tau-t,\vartheta_{-t}\omega)}\|\v(\tau,\tau-t,\vartheta_{-\tau}\omega,\v_{0})\|^2_{\L^2(\mathcal{O}^{c}_{k})}=0.
		\end{align}
	\end{lemma}
	\begin{proof}
		Let $\uprho$ be a smooth function such that $0\leq\uprho(\xi)\leq 1$ for $\xi\in[0,\infty)$ and
		\begin{align}\label{337}
			\uprho(\xi)=\begin{cases*}
				0, \text{ for }0\leq \xi\leq 1,\\
				1, \text{ for } \xi\geq2.
			\end{cases*}
		\end{align}
		By the hypothesis $\uprho$ is constant (hence all derivatives vanish) on
$[0,1]$ and on $[2,\infty)$. Thus the only interval where $\uprho'$ and
$\uprho''$ can be nonzero in the compact interval $[1,2]$. Since $\uprho'\in \mathrm{C}([1,2])$ and $\uprho''\in \mathrm{C}([1,2])$, each attains a finite maximum
on $[1,2]$. Let us define
\begin{align*}
M_1:=\sup_{\xi\in[1,2]}|\rho'(\xi)|,\ 
M_2:=\sup_{\xi\in[1,2]}|\rho''(\xi)|,
\end{align*}
and set $C:=\max\{M_1,M_2\}$. For $\xi\notin[1,2]$ we have $\uprho'(\xi)=\rho''(\xi)=0$,
so the inequalities $|\uprho'(\xi)|\le C$ and $|\uprho''(\xi)|\le C$ hold for all
$\xi\in[0,\infty)$. 
Taking the divergence to the first equation of \eqref{2-A},   we obtain
		\begin{align*}
			-\Delta p&=\nabla\cdot\left[F_N(\|\v+\h z(\vartheta_{t}\omega)\|_{\V}) \cdot\big((\v+{\boldsymbol{g}}z(\vartheta_{t}\omega))\cdot\nabla\big)(\v+{\boldsymbol{g}}z(\vartheta_{t}\omega))\right]-\nabla\cdot\f \\&= F_N(\|\v+\h z(\vartheta_{t}\omega)\|_{\V})\cdot \left\{\nabla\cdot\left[\nabla\cdot\big((\v+\boldsymbol{g}z(\vartheta_{t}\omega))\otimes(\v+\boldsymbol{g}z(\vartheta_{t}\omega))\big)\right]\right\}-\nabla\cdot\f\\
			&= F_N(\|\v+\h z(\vartheta_{t}\omega)\|_{\V}) \sum_{i,j=1}^{3}\frac{\partial^2}{\partial x_i\partial x_j}\big((y_i+{g}_i z(\vartheta_{t}\omega))(y_j+{g}_j z(\vartheta_{t}\omega))\big)-\nabla\cdot\f,
		\end{align*}
		in the weak sense,	which implies that
		\begin{align}\label{p-value}
			p=(-\Delta)^{-1}\left[F_N(\|\v+\h z(\vartheta_{t}\omega)\|_{\V}) \sum_{i,j=1}^{3}\frac{\partial^2}{\partial x_i\partial x_j}\big((y_i+g_i z(\vartheta_{t}\omega))(y_j+g_j z(\vartheta_{t}\omega))\big)-\nabla\cdot\f\right].
		\end{align}
		It follows from \eqref{p-value} and \eqref{FN1} that
		\begin{align}
	\|p\|_{\mathrm{L}^2(\mathcal{O})}&=F_N(\|\v+\h z(\vartheta_{t}\omega)\|_{\V})\cdot\left\|\left[\sum_{i,j=1}^{3}\frac{\partial^2}{\partial x_i\partial x_j}(-\Delta)^{-1}\big((y_i+g_i z(\vartheta_{t}\omega))(y_j+g_j z(\vartheta_{t}\omega))\big)\right]\right\|_{\mathrm{L}^2(\mathcal{O})}\nonumber\\&\quad+\|(-\Delta)^{-1}(\nabla\cdot\f)\|_{\mathrm{L}^2(\mathcal{O})}\nonumber\\  &\leq CF_N(\|\v+\h z(\vartheta_{t}\omega)\|_{\V})\cdot\left\|\sum_{i,j=1}^{3}(-\Delta)^{-1}\big((y_i+g_i z(\vartheta_{t}\omega))(y_j+g_j z(\vartheta_{t}\omega))\big)
			\right\|_{\mathrm{H}^{2}(\mathcal{O})}
			\nonumber\\  &\quad +C\|\nabla\cdot\f\|_{\mathrm{H}^{-2}(\mathcal{O})}\nonumber\\  &\leq CF_N(\|\v+\h z(\vartheta_{t}\omega)\|_{\V})\cdot \left\|\Delta\sum_{i,j=1}^{3}(-\Delta)^{-1}\big((y_i+g_i z(\vartheta_{t}\omega))(y_j+g_j z(\vartheta_{t}\omega))\big)
			\right\|_{\mathrm{L}^{2}(\mathcal{O})} \nonumber\\&\quad+C\|\f\|_{\mathbb{H}^{-1}(\mathcal{O})}
			\nonumber\\  &\leq C F_N(\|\v+\h z(\vartheta_{t}\omega)\|_{\V})\cdot \|\v+\h z(\vartheta_{t}\omega)\|^2_{\L^4(\mathcal{O})}+C\|\f\|_{\mathbb{L}^2(\mathcal{O})}
			\nonumber\\  &\leq C F_N(\|\v+\h z(\vartheta_{t}\omega)\|_{\V})\cdot \|\v+\h z(\vartheta_{t}\omega)\|^{\frac12}_{\H} \|\v+\h z(\vartheta_{t}\omega)\|^{\frac32}_{\V}+C\|\f\|_{\mathbb{L}^2(\mathcal{O})}
			\nonumber\\  &\leq C N  \|\v+\h z(\vartheta_{t}\omega)\|_{\V}+C\|\f\|_{\mathbb{L}^2(\mathcal{O})}
			\nonumber\\  &\leq C N  (\|\v\|_{\V}+\|\h\|_{\V}\left|z(\vartheta_{t}\omega)\right|)+C\|\f\|_{\mathbb{L}^2(\mathcal{O})}\label{p-value-N},
		\end{align}
		where we have also used the elliptic regularity for the Poincar\'e domains with uniformly smooth boundary of class $\mathrm{C}^3$ (cf. Lemma 1, \cite{Heywood}), Ladyzhenskaya's inequality, \eqref{FN1} and \eqref{poin}. Taking the inner product of the first equation in \eqref{2-A} by $\uprho^2\left(\frac{|x|^2}{k^2}\right)\v$ in $\mathbb{L}^2(\mathcal{O})$, we have
		\begin{align}\label{ep1-N}
			&\frac{1}{2} \frac{\d}{\d t}\int_{\mathcal{O}}\uprho^2\left(\frac{|x|^2}{k^2}\right)|\v|^2\d x\nonumber \\&= \underbrace{\nu\int_{\mathcal{O}}(\Delta\v) \uprho^2\left(\frac{|x|^2}{k^2}\right) \v \d x}_{:=I_1(k,t)}-\underbrace{F_N(\|\v+\h z(\vartheta_{t}\omega)\|_{\V})\cdot b\left(\v+\h z,\v+\h z,\uprho^2\left(\frac{|x|^2}{k^2}\right)\v\right)}_{:=I_2(k,t)}\nonumber\\&\quad-\underbrace{\int_{\mathcal{O}}(\nabla p)\uprho^2\left(\frac{|x|^2}{k^2}\right)\v\d x}_{:=I_3(k,t)} + \underbrace{\int_{\mathcal{O}}\f\uprho^2\left(\frac{|x|^2}{k^2}\right)\v\d x + z\int_{\mathcal{O}}
				(\sigma\h+\nu\Delta\h)\uprho^2\left(\frac{|x|^2}{k^2}\right)\v\d x}_{:=I_4(k,t)}.
		\end{align}
		Let us now estimate each term on the right hand side of \eqref{ep1-N}. By \eqref{poin}, integration by parts and  divergence free condition of $\v$, we have
		\begin{align}
			I_{1}(k,t)&= -\nu \int_{\mathcal{O}}\left|\nabla\left(\uprho\left(\frac{|x|^2}{k^2}\right) \v\right)\right|^2  \d x+\nu \int_{\mathcal{O}}\v\nabla\left(\uprho\left(\frac{|x|^2}{k^2}\right)\right)\nabla\left(\uprho\left(\frac{|x|^2}{k^2}\right) \v \right) \d x \nonumber\\&\quad-\nu \int_{\mathcal{O}}\nabla\v \nabla\left(\uprho\left(\frac{|x|^2}{k^2}\right)\right)\uprho\left(\frac{|x|^2}{k^2}\right) \v  \d x\nonumber\\&\leq-\nu \int_{\mathcal{O}}\left|\nabla\left(\uprho\left(\frac{|x|^2}{k^2}\right) \v\right)\right|^2\d x+\frac{\nu}{8} \int_{\mathcal{O}}\left|\nabla\left(\uprho\left(\frac{|x|^2}{k^2}\right) \v\right)\right|^2\d x\nonumber\\&\quad+\frac{\nu\lambda}{8} \int_{\mathcal{O}}\left|\left(\uprho\left(\frac{|x|^2}{k^2}\right) \v\right)\right|^2\d x+\frac{C}{k}\left[\|\v\|^2_{\H}+\|\v\|^2_{\V}\right]\nonumber\\&\leq-\frac{3\nu\lambda}{4} \int_{\mathcal{O}}\left|\left(\uprho\left(\frac{|x|^2}{k^2}\right) \v\right)\right|^2\d x+\frac{C}{k}\|\v\|^2_{\V},\label{ep2-N}
		\end{align}
		and
		\begin{align}
			\left|I_{2}(k,t)\right|
			&= F_N(\|\v+\h z(\vartheta_{t}\omega)\|_{\V}) \cdot \left|4\int_{\mathcal{O}} \uprho\left(\frac{|x|^2}{k^2}\right)\uprho'\left(\frac{|x|^2}{k^2}\right)\frac{x}{k^2}\cdot\v |\v+\h z(\vartheta_{t}\omega)|^2 \d x\right|\nonumber\\&= F_N(\|\v+\h z(\vartheta_{t}\omega)\|_{\V})\cdot \left|4 \int\limits_{\mathcal{O}\cap\{k\leq|x|\leq \sqrt{2}k\}}\uprho\left(\frac{|x|^2}{k^2}\right) \uprho'\left(\frac{|x|^2}{k^2}\right)\frac{x}{k^2}\cdot\v |\v+\h z(\vartheta_{t}\omega)|^2 \d x\right|\nonumber\\&\leq \frac{4\sqrt{2}}{k} F_N(\|\v+\h z(\vartheta_{t}\omega)\|_{\V})  \cdot \int\limits_{\mathcal{O}\cap\{k\leq|x|\leq \sqrt{2}k\}} \left|\uprho'\left(\frac{|x|^2}{k^2}\right)\right| \cdot |\v|\cdot|\v+\h z(\vartheta_{t}\omega)|^2 \d x
			\nonumber\\ & \leq \frac{C}{k} F_N(\|\v+\h z(\vartheta_{t}\omega)\|_{\V}) \cdot \|\v\|_{\H}  \|\v+\h z(\vartheta_{t}\omega)\|^2_{\L^4(\mathcal{O})} 
			\nonumber\\ & \leq \frac{C}{k} F_N(\|\v+\h z(\vartheta_{t}\omega) \|_{\V})\cdot \|\v\|_{\V} \|\v+\h z(\vartheta_{t}\omega)\|^{\frac12}_{\H} \|\v+\h z(\vartheta_{t}\omega)\|^{\frac32}_{\V} 
			\nonumber\\&\leq\frac{CN}{k} \|\v\|_{\V} \|\v+\h z(\vartheta_{t}\omega)\|_{\V}
			\leq\frac{CN}{k}
			\left[\|\v\|^2_{\V}
			+\|\h\|^2_{\V}\left|z(\vartheta_{t}\omega)\right|^2\right],\label{ep3-N}
		\end{align}
		where we have also used H\"older's and Ladyzhenskaya's inequalities, \eqref{FN1} and Young's inequality. Using integration by parts, divergence free condition, \eqref{poin}, \eqref{p-value-N} and Young's inequality, we obtain
		\begin{align}\label{ep9-N} -I_3(k,t)&=2\int_{\mathcal{O}}p\uprho\left(\frac{|x|^2}{k^2}\right)\uprho'\left(\frac{|x|^2}{k^2}\right)\frac{2}{k^2}(x\cdot\v)\d x\nonumber\\&\leq\frac{C}{k} \int\limits_{\mathcal{O}\cap\{k\leq|x|\leq \sqrt{2}k\}}\left|p\right|\left|\v\right|\d x
			\leq \frac{C}{k}\bigg[\|p\|_{\mathrm{L}^2(\mathcal{O})}\|\v\|_{\H}\bigg] \leq \frac{C}{k}\bigg[\|p\|_{\mathrm{L}^2(\mathcal{O})}\|\v\|_{\V}\bigg]
			\nonumber\\&\leq \frac{C}{k}\bigg[\|p\|^2_{\mathrm{L}^2(\mathcal{O})}+\|\v\|^2_{\V}\bigg]
			\leq \frac{C}{k}\bigg[ (1+N^2)  \|\v\|^2_{\V}+N^2\|\h\|^2_{\V}\left|z(\vartheta_{t}\omega)\right|^2+\|\f\|^2_{\mathbb{L}^2(\mathcal{O})}\bigg].
		\end{align}
		Finally, we estimate the remaining term of \eqref{ep1-N} by using  H\"older's and Young's inequalities as follows:
		\begin{align}
			I_4(k,t)&\leq\frac{\nu\lambda}{4}\int_{\mathcal{O}} \left|\uprho\left(\frac{|x|^2}{k^2}\right)\v\right|^2 \d x+C\int_{\mathcal{O}}\uprho^2\left(\frac{|x|^2}{k^2}\right)\bigg[|\f|^2+\left|z(\vartheta_{t}\omega)\right|^2 |\h|^2 +\left|z(\vartheta_{t}\omega)\right|^2|\Delta\h|^2\bigg]\d x.	\label{ep4-N}
		\end{align}
		Combining \eqref{ep1-N}-\eqref{ep4-N}, we get
		\begin{align}\label{ep5-N}
			&\frac{\d}{\d t}\|\v\|^2_{\mathbb{L}^2(\mathcal{O}_k^{c})}+\nu\lambda\|\v\|^2_{\mathbb{L}^2(\mathcal{O}_k^c)} \nonumber\\ &\leq\frac{C}{k} \bigg[ (1+N+N^2)  \|\v\|^2_{\V}
			+ N\|\h\|^2_{\V}\left|z(\vartheta_{t}\omega)\right|^2 +N^2\|\h\|^2_{\V}\left|z(\vartheta_{t}\omega)\right|^2+\|\f\|^2_{\mathbb{L}^2(\mathcal{O})}\bigg]\nonumber\\&\quad+C \int_{\mathcal{O}\cap\{|x|\geq k\}}|\f(x)|^2\d x  +C\left|z(\vartheta_{t}\omega)\right|^2 \int_{\mathcal{O}\cap\{|x|\geq k\}}\left[|\h (x)|^2+|\Delta\h (x)|^2\right]\d x
			\nonumber\\ &\leq\frac{C(1+N^2)}{k} \bigg[   \|\v\|^2_{\V}
			+\|\h\|^2_{\V}\left|z(\vartheta_{t}\omega)\right|^2+\|\f\|^2_{\mathbb{L}^2(\mathcal{O})}\bigg]\nonumber\\&\quad+C \int_{\mathcal{O}\cap\{|x|\geq k\}}|\f(x)|^2\d x  +C\left|z(\vartheta_{t}\omega)\right|^2 \int_{\mathcal{O}\cap\{|x|\geq k\}}\left[|\h (x)|^2+|\Delta\h (x)|^2\right]\d x.
		\end{align}
		Making use of variation of constants formula to the  inequality \eqref{ep5-N} on $(s-t,s)$ and replacing $\omega$ by $\vartheta_{-s}\omega$, we find, for $s\leq\tau,\ t\geq 0$ and $\omega\in\Omega$,
		\begin{align}\label{ep6-N}
			\|\v(s,s-t,\vartheta_{-s}\omega,\v_{0})\|^2_{\mathbb{L}^2(\mathcal{O}_k^{c})} & \leq e^{-\nu\lambda t}\|\v_{0}\|^2_{\H} +\frac{C}{k}\bigg[\int_{s-t}^{s}e^{\nu\lambda(\zeta-s)}\|\v(\zeta,s-t,\vartheta_{-s}\omega,\v_{0})\|^2_{\V}\d\zeta \nonumber\\&\quad+\int_{-t}^{0}e^{\nu\lambda\zeta}\bigg\{\|\h\|^2_{\V}\left|z(\vartheta_{\zeta}\omega)\right|^{2}+\|\f(\zeta+s)\|^2_{\L^2(\mathcal{O})}\bigg\}\d\zeta\bigg]
			\nonumber\\&\quad+C\int_{-t}^{0}e^{\nu\lambda\zeta}\left|z(\vartheta_{\zeta}\omega)\right|^2\d\zeta\left[\int\limits_{\mathcal{O}\cap\{|x|\geq k\}}(|\h (x)|^2+|\Delta\h (x)|^2)\d x\right]
			\nonumber\\&\quad+ C\int_{-t}^{0}e^{\nu\lambda\zeta} \int\limits_{\mathcal{O}\cap\{|x|\geq k\}}|\f(x,\zeta+s)|^2\d x\d\zeta.
		\end{align}
		Now, using the fact that $\h\in\D(\A)$, \eqref{f3-N}, the definition of collections $\mathfrak{D}$ and $\mathfrak{B}$ (see Subsection \ref{CoRS}), \eqref{Z3}, \eqref{IRAS2-N} and Lemma \ref{Absorbing-N}, one can complete the proof.
	\end{proof}

	The following lemma provides the backward flattening estimates for the solution of the system \eqref{2-A}. For each $k\geq1$, we let
	\begin{align}\label{varrho_k}
		\varrho_k(x):= 1-\uprho\left(\frac{|x|^2}{k^2}\right),  \ \ x\in\mathcal{O}.
	\end{align}
	Let $\bar{\v}:=\varrho_k\v$ for $\v:=\v(s,s-t,\omega,\v_{\tau})\in\H$. Then $\bar{\v}\in\L^2(\mathcal{O}_{\sqrt{2}k})$ which has the orthogonal decomposition:
	\begin{align}\label{DirectProd}
		\bar{\v}=\P_{i}\bar{\v}\oplus(\I-\P_{i})\bar{\v}=:\bar{\v}_{i,1}+\bar{\v}_{i,2},  \ \ \text{ for each } \ i\in\N,
	\end{align}
	where $\P_i:\L^2(\mathcal{O}_{\sqrt{2}k})\to\H_{i}:=\mathrm{span}\{\boldsymbol{e}_1,\boldsymbol{e}_2,\ldots,\boldsymbol{e}_i\}\subset\L^2(\mathcal{O}_{\sqrt{2}k})$ is a canonical projection and $\{\boldsymbol{e}_m\}_{m=1}^{\infty}$ is a family of eigenfunctions for $-\Delta$ in $\L^2(\mathcal{O}_{\sqrt{2}k})$ with corresponding  eigenvalues $0<\lambda_1\leq\lambda_2\leq\cdots\leq\lambda_m\to\infty$ as $m\to\infty$. We also have that $$\varrho_k \Delta\v=\Delta\bar{\v}-\v\Delta\varrho_k-2\nabla\varrho_k\cdot\nabla\v.$$ Furthermore, for  $\boldsymbol{\psi}\in\H_0^1(\mathcal{O}_{\sqrt{2}k})$, we have
	\begin{align}\label{poin-i}
		\mathrm{P}_i\boldsymbol{\psi}&=\sum_{m=1}^{i}(\boldsymbol{\psi},\boldsymbol{e}_m)\boldsymbol{e}_m,\  \A^{1/2}\mathrm{P}_i\boldsymbol{\psi}=\sum_{m=1}^{i}\lambda^{1/2}_j(\boldsymbol{\psi},\boldsymbol{e}_m)\boldsymbol{e}_m,\nonumber\\ (\I-\mathrm{P}_i)\boldsymbol{\psi}&=\sum_{m=i+1}^{\infty}(\boldsymbol{\psi},\boldsymbol{e}_m)\boldsymbol{e}_m,\ \A^{1/2}(\I-\P_i)\boldsymbol{\psi}=\sum_{m=i+1}^{\infty}\lambda^{1/2}_m(\boldsymbol{\psi},\boldsymbol{e}_m)\boldsymbol{e}_m, \nonumber\\
		\|\nabla(\I-\P_i)\boldsymbol{\psi}\|_{\L^2(\mathcal{O}_{\sqrt{2}k})}^2&=\|\A^{1/2}(\I-\P_i)\boldsymbol{\psi}\|_{\L^2(\mathcal{O}_{\sqrt{2}k})}^2=\sum_{m=i+1}^{\infty}\lambda_m|(\boldsymbol{\psi},\boldsymbol{e}_m)|^2\nonumber\\&\geq \lambda_{i+1}\sum_{m=i+1}^{\infty}|(\boldsymbol{\psi},\boldsymbol{e}_m)|^2=\lambda_{i+1}\|(\I-\P_i)\boldsymbol{\psi}\|_{\L^2(\mathcal{O}_{\sqrt{2}k})}^2.
	\end{align}

	\begin{lemma}\label{Flattening-N}
		Suppose that $\f\in\mathrm{L}^2_{\mathrm{loc}}(\R;\L^2(\mathcal{O}))$ and Assumption \ref{assumpO} is satisfied. Let $(\tau,\omega,D)\in\R\times\Omega\times\mathfrak{D}$ and $k\geq1$ be fixed. Then
		\begin{align}\label{FL-P}
			\lim_{i,t\to+\infty}\sup_{s\leq \tau}\sup_{\v_{0}\in D(s-t,\vartheta_{-t}\omega)}\|(\I-\P_{i})\bar{\v}(s,s-t,\vartheta_{-s}\omega,\bar{\v}_{0,2})\|^2_{\L^2(\mathcal{O}_{\sqrt{2}k})}=0,
		\end{align}
		where $\bar{\v}_{0,2}=(\I-\P_{i})(\varrho_k\v_{0})$. In addition, for any $(\tau,\omega,B)\in\R\times\Omega\times\mathfrak{B}$, we have
        \begin{align}\label{FL-P-tau}
			\lim_{i,t\to+\infty}\sup_{\v_{0}\in B(\tau-t,\vartheta_{-t}\omega)}\|(\I-\P_{i})\bar{\v}(\tau,\tau-t,\vartheta_{-\tau}\omega,\bar{\v}_{0,2})\|^2_{\L^2(\mathcal{O}_{\sqrt{2}k})}=0.
		\end{align}
	\end{lemma}

	\begin{proof}
		Multiplying the first equation of (3.2) by $\varrho_k$, we find
		\begin{align}\label{FL1}
			&\frac{\d\bar{\v}}{\d t}-\nu\Delta\bar{\v}+\varrho_k\cdot F_N(\|\v+\h z(\vartheta_{t}\omega)\|_{\V}) \cdot \left[\big((\v+\h z)\cdot\nabla\big)(\v+\h z)\right]+\varrho_k\nabla p\nonumber\\&=-\nu\v\Delta\varrho_k-2\nu\nabla\varrho_k\cdot\nabla\v+\varrho_k\f +\sigma\varrho_k\h z+\nu z\varrho_k\Delta\h.
		\end{align}
		Applying $(\I-\P_i)$ to the equation \eqref{FL1} and taking the inner product of the resulting equation with $\bar{\v}_{i,2}$ in $\L^2(\mathcal{O}_{\sqrt{2}k})$, we have
		\begin{align}\label{FL2}
			&\frac{1}{2}\frac{\d}{\d t}\|\bar{\v}_{i,2}\|^2_{\L^2(\mathcal{O}_{\sqrt{2}k})} +\nu\|\nabla\bar{\v}_{i,2}\|^2_{\L^2(\mathcal{O}_{\sqrt{2}k})}\nonumber\\&=-\underbrace{F_N(\|\v+\h z(\vartheta_{t}\omega)\|_{\V})\cdot\sum_{q,m=1}^{3}\int\limits_{\mathcal{O}_{\sqrt{2}k}}\left(\I-\P_i\right)\bigg[(y_{q}+g_{q}z)\frac{\partial(y_{m}+g_{m}z)}{\partial x_q}\varrho^2_k(x)y_{m}\bigg]\d x}_{:=J_1}\nonumber\\&\quad-\underbrace{\left\{\nu\big(\v\Delta\varrho_k,\bar{\v}_{i,2}\big)+2\nu\big(\nabla\varrho_k\cdot\nabla\v,\bar{\v}_{i,2}\big)-\big(\varrho_k\f,\bar{\v}_{i,2}\big) - z\varrho_k\big(\sigma\h+\nu\Delta\h,\bar{\v}_{i,2}\big)\right\}}_{:=J_2}\nonumber\\&\quad-\underbrace{\big(\varrho_k(x)\nabla p, \bar{\v}_{i,2}\big)}_{:=J_3}.
		\end{align}
		
		Next, we estimate each term of \eqref{FL2} in the following way: Using integration by parts, divergence free condition of $\v$, \eqref{poin-i} (one can assume that $\lambda_{i}\geq1$), H\"older's, Ladyzhenskaya's inequalities, \eqref{FN1} and Young's inequality, we find
		\begin{align}
			\left|J_1\right| &=F_N(\|\v+\h z(\vartheta_{t}\omega)\|_{\V})\left|2\int_{\mathcal{O}_{\sqrt{2}k}}\left(\I-\P_i\right)\bigg[\uprho'\left(\frac{|x|^2}{k^2}\right)\frac{x}{k^2}\cdot\left\{\varrho_k(x)\v\right\}|\v+\boldsymbol{g}z(\vartheta_{t}\omega)|^2\bigg]\d x\right|
			\nonumber\\&\leq C F_N(\|\v+\h z(\vartheta_{t}\omega)\|_{\V})\cdot \|\bar{\v}_{i,2}\|_{\L^2(\mathcal{O}_{\sqrt{2}k})} \|\v+\boldsymbol{g}z(\vartheta_{t}\omega)\|^2_{\L^4(\mathcal{O})}
			\nonumber\\&\leq C \lambda_{i+1}^{-\frac12} F_N(\|\v+\h z(\vartheta_{t}\omega)\|_{\V})\cdot \|\nabla\bar{\v}_{i,2}\|_{\L^2(\mathcal{O}_{\sqrt{2}k})} \|\v+\boldsymbol{g}z(\vartheta_{t}\omega)\|^{\frac12}_{\H} \|\v+\boldsymbol{g}z(\vartheta_{t}\omega)\|^{\frac32}_{\V}
			\nonumber\\&\leq C N \lambda_{i+1}^{-\frac12} \|\nabla\bar{\v}_{i,2}\|_{\L^2(\mathcal{O}_{\sqrt{2}k})} \|\v+\boldsymbol{g}z(\vartheta_{t}\omega)\|_{\V} 
			\nonumber\\ & \leq \frac{\nu}{6}\|\nabla\bar{\v}_{i,2}\|^2_{\L^2(\mathcal{O}_{\sqrt{2}k})} + C N^2 \lambda_{i+1}^{-1} \left[\|\v\|^2_{\V} + |z(\vartheta_{t}\omega)|^2\|\boldsymbol{g}\|^2_{\V}\right],\\   
			\left|J_2\right|  &\leq C\bigg[\|\v\|_{\H}+ \|\v\|_{\V}+\|\f\|_{\L^2(\mathcal{O})}+\left|z(\vartheta_{t}\omega)\right|(\|\h\|_{\H}+\|\h\|_{\D(\A)})\bigg]\|\bar{\v}_{i,2}\|_{\L^2(\mathcal{O}_{\sqrt{2}k})}
			\nonumber\\ &\leq C\lambda^{-1/2}_{i+1}\bigg[\|\v\|_{\V}+\|\f\|_{\L^2(\mathcal{O})}+\left|z(\vartheta_{t}\omega)\right|(\|\h\|_{\H}+\|\h\|_{\D(\A)})\bigg]\|\nabla\bar{\v}_{i,2}\|_{\L^2(\mathcal{O}_{\sqrt{2}k})}
			\nonumber\\ & \leq\frac{\nu}{6}\|\nabla\bar{\v}_{i,2}\|^2_{\L^2(\mathcal{O}_{\sqrt{2}k})}+ C\lambda^{-1}_{i+1}\bigg[\|\v\|^2_{\V}+\|\f\|^2_{\L^2(\mathcal{O})} + \left|z(\vartheta_{t}\omega)\right|^2(\|\h\|^2_{\H}+\|\h\|^2_{\D(\A)})\bigg],\label{FL5}\\  
			\left|J_3\right| &=\left|\int_{\mathcal{O}_{\sqrt{2}k}}(\I-\P_i)\bigg[\nabla p\left\{\varrho_k(x)\right\}^2\v\bigg]\d x\right|=2\left|\int_{\mathcal{O}_{\sqrt{2}k}}(\I-\P_i)\bigg[\uprho'\left(\frac{|x|^2}{k^2}\right)p\varrho_k(x)\v\bigg]\d x\right|
			\nonumber\\&\leq C\|p\|_{\mathrm{L}^2(\mathcal{O})}\|\bar{\v}_{i,2}\|_{\L^2(\mathcal{O}_{\sqrt{2}k})}
			\nonumber\\&\leq C \lambda_{i+1}^{-\frac12}\left[ N  (\|\v\|_{\V}+\left|z(\vartheta_{t}\omega)\right|\|\h\|_{\V})+ \|\f\|_{\mathbb{L}^2(\mathcal{O})}\right]\|\nabla\bar{\v}_{i,2}\|_{\L^2(\mathcal{O}_{\sqrt{2}k})}
			\nonumber\\ & \leq\frac{\nu}{6}\|\nabla\bar{\v}_{i,2}\|^2_{\L^2(\mathcal{O}_{\sqrt{2}k})}+ C\lambda^{-1}_{i+1}\bigg[N^2\|\v\|^2_{\V}+ N^2\left|z(\vartheta_{t}\omega)\right|^2\|\h\|^2_{\V}+\|\f\|^2_{\L^2(\mathcal{O})} \bigg],\label{FL6}
		\end{align}
		where we have also used the estimate \eqref{p-value-N} in \eqref{FL6}. Now, combining \eqref{FL2}-\eqref{FL6} and using \eqref{poin} in the resulting inequality, we arrive at
		\begin{align}\label{FL7}
			&\frac{\d}{\d t}\|\bar{\v}_{i,2}\|^2_{\L^2(\mathcal{O}_{\sqrt{2}k})} + \nu\lambda\|\bar{\v}_{i,2}\|^2_{\L^2(\mathcal{O}_{\sqrt{2}k})} \leq  C (1+N^2)\lambda_{i+1}^{-1}\bigg[\|\v\|^2_{\V}+ \left|z(\vartheta_{t}\omega)\right|^2\|\h\|^2_{\D(\A)}+\|\f\|^2_{\L^2(\mathcal{O})} \bigg].
		\end{align}
		By using the  variation of constants formula in \eqref{FL7}, we find
		\begin{align}\label{FL8}
			&\|(\I-\P_{i})\bar{\v}(s,s-t,\vartheta_{-s}\omega,\bar{\v}_{0,2})\|^2_{\L^2(\mathcal{O}_{\sqrt{2}k})}\nonumber\\&\leq e^{-\nu\lambda t}\|(\I-\P_i)(\varrho_k\v_{0})\|^2_{\L^2(\mathcal{O}_{\sqrt{2}k})} +C (1+N^2)\lambda_{i+1}^{-1}\bigg[\underbrace{\int_{s-t}^{s}e^{\nu\lambda(\zeta-s)}\|\v(\zeta,s-t,\vartheta_{-s}\omega,\v_{0})\|^2_{\V}\d\zeta}_{L_1(s,t)}\nonumber\\&\quad+\underbrace{\|\h\|^2_{\D(\A)}\int_{s-t}^{s}e^{\nu\lambda(\zeta-s)}\left|z(\vartheta_{\zeta}\omega)\right|^{2}\d\zeta}_{L_2(t)} + \underbrace{\int_{-t}^{0}e^{\nu\lambda\zeta}\|\f(\zeta+s)\|^2_{\L^2(\mathcal{O})}\d\zeta}_{L_3(s,t)}.
		\end{align}
		We have from \eqref{Z3}, \eqref{G3}, \eqref{AB1-N} and \eqref{IRAS2-N} that
		\begin{equation}\label{FL9}
			\sup_{s\leq \tau}L_1(s,t)<+\infty,\	\ \sup_{s\leq \tau}L_3(s,t)<+\infty \ \ \text{ and }\ \  L_2(t)<+\infty,
		\end{equation}
		for sufficiently large $t>0$. Further,
		\begin{align}\label{FL11}
			\|(\I-\P_i)(\varrho_k\v_{0})\|^2_{\L^2(\mathcal{O}_{\sqrt{2}k})}\leq C\|\v_{0}\|^2_{\H},
		\end{align}
		for all $\v_{0}\in D(s-t,\vartheta_{-t}\omega)$ and $s\leq\tau$. Now, using the definition of collections $\mathfrak{D}$ and $\mathfrak{B}$ (see Subsection \ref{CoRS}), and \eqref{FL9}-\eqref{FL11} in \eqref{FL8}, we obtain \eqref{FL-P} and \eqref{FL-P-tau}, as desired, which  completes the proof.
	\end{proof}

	\subsection{Proof of Theorem \ref{MT1-N}}\label{thm1.4}
	In this subsection, we demonstrate the main results of this section, that is, the existence of $\mathfrak{D}$-pullback random attractors and their asymptotic autonomous robustness for the system \eqref{SNSE-A}. The proof of this theorem is  divided into the following seven steps:
	\vskip 2mm
	\noindent
	\textbf{Step I:} \textit{$\mathfrak{D}$-pullback time-semi-uniform asymptotic compactness of $\Phi$.} It is enough to prove that for each $(\tau,\omega,D)\in\R\times\Omega\times\mathfrak{D}$, arbitrary sequences $s_n\leq\tau$, $t_{n}\to+\infty$ and $\v_{0,n}\in D(s_n-t_n,\vartheta_{-t_n}\omega)$,  the sequence $$\v_n:=\v(s_n,s_n-t_n,\vartheta_{-s_n}\omega,\v_{0,n})$$ is pre-compact in $\H$. Let $E_{N}=\{\v_n:n\geq N\}, \ N=1,2,\ldots.$ In order to prove the pre-compactness of the sequence $\v_n$, it is enough to show that the Kuratowski measure $\kappa_{\H}(E_{N})\to0$ and $N\to+\infty$, (see Lemma \ref{K-BAC}).
	
	Since, $t_n\to +\infty$, there exists $\mathcal{N}\in\N$ such that $t_n\geq \mathfrak{T}$ for all $n\geq \mathcal{N}$. For each $\eta>0$, by Lemma \ref{largeradius-N}, there exist $\mathcal{N}_1\geq \mathcal{N}$ and $K_0\geq1$ such that
	\begin{align}\label{MTA1}
		\left\|\uprho\left(\frac{|x|^2}{k^2}\right)\v_n\right\|_{\L^2(\mathcal{O}^{c}_{K})}\leq\frac{\eta}{2}, 
	\end{align}
	for all $n\geq \mathcal{N}_1$ and for all  $K\geq K_0$, where $\mathcal{O}^{c}_{K}=\mathcal{O}\backslash\mathcal{O}_{K}$ and  $\mathcal{O}_{K}=\{x\in\mathcal{O}:|x|\leq K\}.$  By Lemma \ref{Flattening-N}, there exist $i_0\in\N$ and $\mathcal{N}_2\geq\mathcal{N}_1$ such that
	\begin{align}\label{MTA2}
		\|(\I-\P_i)(\varrho_{K}\v_n)\|_{\L^2(\mathcal{O}_{\sqrt{2}K})}\leq\frac{\eta}{2}, 
	\end{align}
	for all $n\geq \mathcal{N}_2$ and for all  $i\geq i_0$.
	
	Now, Lemma \ref{Absorbing-N} provides us that the set $E_{\mathcal{N}_2}$ is bounded in $\H$. Then, the set $\{\varrho_{K}\v_n:n\geq\mathcal{N}_2\}$ is bounded in $\L^2(\mathcal{O}_{\sqrt{2}K})$. Hence, by the finite-dimensional range of $\P_i$, $\P_{i}\{\varrho_{K}\v_n:n\geq\mathcal{N}_2\}$ is pre-compact in $\L^2(\mathcal{O}_{\sqrt{2}K})$, from which we conclude that
	\begin{align}\label{MTA3}
		\kappa_{\L^2(\mathcal{O}_{\sqrt{2}K})}\left(\P_{i}\{\varrho_{K}\v_n:n\geq\mathcal{N}_2\}\right)=0.
	\end{align}
	It follows from \eqref{MTA2}-\eqref{MTA3} and \cite[Theorem 1.4]{Rakocevic} that
	\begin{align}\label{MTA4}
		&\kappa_{\L^2(\mathcal{O}_{\sqrt{2}K})}\left(\{\varrho_{K}\v_n:n\geq\mathcal{N}_2\}\right)\nonumber\\&\leq\kappa_{\L^2(\mathcal{O}_{\sqrt{2}K})}\left(\P_{i}\{\varrho_{K}\v_n:n\geq\mathcal{N}_2\}\right)+\kappa_{\L^2(\mathcal{O}_{\sqrt{2}K})}\left((\I-\P_{i})\{\varrho_{K}\v_n:n\geq\mathcal{N}_2\}\right)\nonumber\\& \leq\frac{\eta}{2}.
	\end{align}
	We infer from \eqref{MTA1} and \eqref{MTA4} that
	\begin{align*}
		\kappa_{\H}(E_{\mathcal{N}_2})\leq\kappa_{\L^2(\mathcal{O}_{\sqrt{2}K})}(\varrho_{K}E_{\mathcal{N}_2})+\kappa_{\L^2(\mathcal{O}^{c}_{K})}\left(\uprho\left(\frac{|x|^2}{k^2}\right)E_{\mathcal{N}_2}\right)\leq \eta,
	\end{align*}
	which shows that $\Phi$ is time-semi-uniformly asymptotically compact in $\H$.
	\vskip 2mm
	\noindent
	\textbf{Step II:} \textit{$\mathfrak{B}$-pullback asymptotically compactness of $\Phi$.} 
     Using arguments similar to those in \textbf{Step I}, we  prove that for each $(\tau,\omega,B)\in\R\times\Omega\times\mathfrak{B}$, arbitrary sequence $t_{n}\to+\infty$ and $\v_{0,n}\in B(\tau-t_n,\vartheta_{-t_n}\omega)$,  the sequence 
	$$\v^n:=\v(\tau,\tau-t_n,\vartheta_{-\tau}\omega,\v_{0,n})$$ 
	is pre-compact in $\H$. We denote $E^{N}=\{\v^n:n\geq N\}, \ N=1,2,\ldots.$ 
	In view of Lemma \ref{largeradius-N}, for each $\eta>0$, there exist $\mathcal{N}_3\in\N$ and $K_1\geq 1$ such that
	\begin{align*}
		\left\|\uprho\left(\frac{|x|^2}{k^2}\right)\v^n\right\|_{\L^2(\mathcal{O}^{c}_{K})}\leq\frac{\eta}{2}, 
	\end{align*}
	for all $n\geq \mathcal{N}_3$ and for all  $K\geq K_1$. 
	 In view of Lemma \ref{Flattening-N}, there exist $i_1\in\N$ and $\mathcal{N}_4\geq\mathcal{N}_3$ such that
	\begin{align*}
		\|(\I-\P_i)(\varrho_{K}\v^n)\|_{\L^2(\mathcal{O}_{\sqrt{2}K})}\leq\frac{\eta}{2}, 
	\end{align*}
	for all $n\geq \mathcal{N}_4$ and for all  $i\geq i_1$. Since the set $E^{\mathcal{N}_4}$ is bounded in $\H$. Then, the set $\{\varrho_{K}\v^n:n\geq\mathcal{N}_2\}$ is bounded in $\L^2(\mathcal{O}_{\sqrt{2}K})$. Hence, $\P_{i}\{\varrho_{K}\v^n:n\geq\mathcal{N}_2\}$ is pre-compact in $\L^2(\mathcal{O}_{\sqrt{2}K})$, from which we deduce 
	\begin{align*}
		\kappa_{\L^2(\mathcal{O}_{\sqrt{2}K})}\left(\P_{i}\{\varrho_{K}\v^n:n\geq\mathcal{N}_4\}\right)=0.
	\end{align*}
	Therefore, we have 
	\begin{align*}
		&\kappa_{\L^2(\mathcal{O}_{\sqrt{2}K})}\left(\{\varrho_{K}\v^n:n\geq\mathcal{N}_4\}\right)\nonumber\\&\leq\kappa_{\L^2(\mathcal{O}_{\sqrt{2}K})}\left(\P_{i}\{\varrho_{K}\v^n:n\geq\mathcal{N}_4\}\right)+\kappa_{\L^2(\mathcal{O}_{\sqrt{2}K})}\left((\I-\P_{i})\{\varrho_{K}\v^n:n\geq\mathcal{N}_4\}\right)\nonumber\\&\leq\frac{\eta}{2}.
	\end{align*}
	From the above estimates, we obtain the following.
	\begin{align*}
		\kappa_{\H}(E^{\mathcal{N}_4})\leq\kappa_{\L^2(\mathcal{O}_{\sqrt{2}K})}(\varrho_{K}E^{\mathcal{N}_4})+\kappa_{\L^2(\mathcal{O}^{c}_{K})}\left(\uprho\left(\frac{|x|^2}{k^2}\right)E^{\mathcal{N}_4}\right)\leq \eta,
	\end{align*}
	which shows that $\Phi$ is asymptotically compact in $\H$.
	\vskip 2mm
	\noindent
	\textbf{Step III:} \textit{$\mathfrak{D}$-pullback attractor $\mathcal{A}(\tau,\omega)$.} Proposition \ref{IRAS-N}  and \textbf{Step I} ensure that $\Phi$ has a $\mathfrak{D}$-pullback absorbing set and $\Phi$ is $\mathfrak{D}$-pullback asymptotically compact in $\H$, respectively. Hence, by the abstract theory established in \cite{SandN_Wang}, $\Phi$ has a unique $\mathfrak{D}$-pullback attractor $\mathcal{A}$ which is given by
	\begin{align}\label{A1}
		\mathcal{A}=\bigcap\limits_{t_0>0}\overline{\bigcup\limits_{t\geq t_0}\Phi(t,\tau-t,\vartheta_{-t}\omega)\mathcal{R}(\tau-t,\vartheta_{-t}\omega)}^{\H}.
	\end{align}
	However, we remark that the $\mathscr{F}$-measurability of $\mathcal{A}$ is unknown, therefore we say that $\mathcal{A}$ is a $\mathfrak{D}$-pullback attractor instead of $\mathfrak{D}$-pullback random attractor.
	\vskip 2mm
	\noindent
	\textbf{Step IV:} \textit{$\mathfrak{B}$-pullback attractor $\widetilde{\mathcal{A}}(\tau,\omega)$.} Proposition \ref{IRAS-N} (part (ii)) and \textbf{Step II} ensure that $\Phi$ has a $\mathfrak{B}$-pullback random absorbing set and $\Phi$ is $\mathfrak{B}$-pullback asymptotically compact, respectively. Hence, by the abstract theory established in \cite{SandN_Wang}, $\Phi$ has a unique $\mathfrak{B}$-pullback random attractor $\widetilde{\mathcal{A}}$ which is given by
	\begin{align}\label{A2}
		\widetilde{\mathcal{A}}=\bigcap\limits_{t_0>0}\overline{\bigcup\limits_{t\geq t_0}\Phi(t,\tau-t,\vartheta_{-t}\omega)\widetilde{\mathcal{R}}(\tau-t,\vartheta_{-t}\omega)}^{\H}.
	\end{align}
	\vskip 2mm
	\noindent
	\textbf{Step V:} \textit{Time-semi-uniformly compactness of $\mathcal{A}(\tau,\omega)$.} We prove that $\bigcup\limits_{s\leq\tau}\mathcal{A}(s,\omega)$ is pre-compact in $\H$. Let $\{\u_n\}_{n=1}^{\infty}$ be an arbitrary sequence extracted from $\bigcup\limits_{s\leq\tau}\mathcal{A}(s,\omega)$. Then, we can find a sequence $s_n\leq\tau$ such that $\u_n\in\mathcal{A}(s_n,\omega)$ for each $n\in\N$. Now, for the sequence $t_n\to +\infty$, by the invariance property of $\mathcal{A},$ we have $\u_n\in\Phi(t_n,s_n-t_n,\vartheta_{-t_{n}}\omega)\mathcal{A}(s_n-t_n,\vartheta_{-t_{n}}\omega)$. It implies that we can find $\u_{0,n}\in\mathcal{A}(s_n-t_n,\vartheta_{-t_{n}}\omega)$ such that $\u_n=\Phi(t_n,s_n-t_n,\vartheta_{-t_{n}}\omega,\u_{0,n})$, where $\u_{0,n}\in\mathcal{A}(s_n-t_n,\vartheta_{-t_{n}}\omega)\subseteq\mathcal{R}(s_n-t_n,\vartheta_{-t_{n}}\omega)$ with $s_n\leq\tau$ and $\mathcal{R}\in\mathfrak{D}$.  Then,  it follows from  the $\mathfrak{D}$-pullback time-semi-uniform asymptotic compactness of $\Phi$ that the sequence $\{\u_n\}_{n=1}^{\infty}$ is pre-compact in $\H$. Hence, $\bigcup\limits_{s\leq\tau}\mathcal{A}(s,\omega)$ is pre-compact in $\H$.
	\vskip 2mm
	\noindent
	\textbf{Step VI:} \textit{$\mathcal{A}(\tau,\omega)=\widetilde{\mathcal{A}}(\tau,\omega)$. This  implies that $\Phi$ has a unique pullback \textbf{random} attractor which is time-semi-uniformly compact in $\H$.} Let us fix $(\tau,\omega)\in\R\times\Omega$. Since, by Proposition \ref{IRAS-N}, $\mathcal{R}(\tau,\omega)\supseteq\widetilde{\mathcal{R}}(\tau,\omega)$, it follows from \eqref{A1} and \eqref{A2} that $\mathcal{A}(\tau,\omega)\supseteq\widetilde{\mathcal{A}}(\tau,\omega)$. At the same time, since $\mathcal{A}\in\mathfrak{B}\subseteq\mathfrak{D}$, the invariance property of $\mathcal{A}$ and the attraction property of $\widetilde{\mathcal{A}}$ imply that
	\begin{align*}
		\text{dist}_{\H}(\mathcal{A}(\tau,\omega),\widetilde{\mathcal{A}}(\tau,\omega))=\text{dist}_{\H}(\Phi(t,\tau-t,\vartheta_{-t}\omega)\mathcal{A}(\tau-t,\vartheta_{-t}\omega),\widetilde{\mathcal{A}}(\tau,\omega))\to 0,
	\end{align*}
	as $t\to+\infty$. This indicates that $\mathcal{A}(\tau,\omega)\subseteq\widetilde{\mathcal{A}}(\tau,\omega)$. Hence $\mathcal{A}(\tau,\omega)=\widetilde{\mathcal{A}}(\tau,\omega)$, which, in view of the $\mathscr{F}$-measurability of $\widetilde{\mathcal{A}}(\tau,\omega)$, shows that $\mathcal{A}(\tau,\omega)$ is $\mathscr{F}$-measurable.
	\vskip 2mm
	\noindent
	\textbf{Step VII:} \textit{Proof of \eqref{MT2-N} and \eqref{MT3-N}.} By using {Propositions} \ref{IRAS-N} and \ref{Back_conver-N} and \emph{time-semi-uniformly compactness} of $\mathcal{A}$, and applying similar arguments as in the proof of \cite[Theorem 5.2]{CGTW}, one can complete the proof. Since, the arguments are similar to the proof of \cite[Theorem 5.2]{CGTW}, we are not repeating here.

	\section{Asymptotically Autonomous Robustness of Random Attractors for \eqref{1}: Linear Multiplicative Noise}\label{sec4}\setcounter{equation}{0}
	In this section, we consider 3D SGMNSE driven by a linear multiplicative white noise ($S(\u)=\u$ in \eqref{1}) and establish the existence and asymptotic autonomous robustness of $\mathfrak{D}$-pullback random attractors. Let us define $\v(t,\tau,\omega,\v_{\tau}):=e^{-z(\vartheta_{t}\omega)}\u(t,\tau,\omega,\u_{\tau})\ \text{ with  }\  \v_{\tau}=e^{-z(\vartheta_{\tau}\omega)}\u_{\tau},$ where $z$ satisfies \eqref{OU2} and $\u$ is the solution of \eqref{1} with $S(\u)=\u$. Then $\v$ satisfies:
	\begin{equation}\label{2-M}
		\left\{
		\begin{aligned}
			\frac{\d\v}{\d t}-\nu \Delta\v&+F_{N}(e^{z(\vartheta_{t}\omega)}\|\v\|_{\V})\cdot  e^{z(\vartheta_{t}\omega)}(\v\cdot\nabla)\v+e^{-z(\vartheta_{t}\omega)}\nabla p\\&= \f e^{-z(\vartheta_{t}\omega)} +\sigma z(\vartheta_t\omega)\v,  && \text{ in }\  \mathcal{O}\times(\tau,\infty), \\ \nabla\cdot\v&=0, &&  \text{ in } \ \ \mathcal{O}\times[\tau,\infty), \\
			\v&=0, && \text{ on } \ \ \partial\mathcal{O}\times[\tau,\infty),\\
			\v(\tau)&=\v_{0}=e^{-z(\vartheta_{\tau}\omega)}\u_{0},  && \ x\in \mathcal{O},
		\end{aligned}
		\right.
	\end{equation}
	as well as (projected form) for $ t> \tau, \tau\in\R ,$
	\begin{equation}\label{CNSE-M}
		\left\{
		\begin{aligned}
			\frac{\d\v}{\d t}+\nu \A\v+ F_{N}(e^{z(\vartheta_{t}\omega)}\|\v\|_{\V}) \cdot e^{z(\vartheta_{t}\omega)}\B\big(\v\big)&=\mathcal{P}\f e^{-z(\vartheta_{t}\omega)} + \sigma z(\vartheta_t\omega)\v , \\
			\v(\tau)&=\v_{0}=e^{-z(\vartheta_{\tau}\omega)}\u_{0},
		\end{aligned}
		\right.
	\end{equation}
	in $\V^*$. Next, we consider the 3D autonomous SGMNSE with a linear multiplicative white noise corresponding to the non-autonomous system \eqref{SNSE} with $S(\u)=\u$  as
	\begin{equation}\label{A-SNSE-M}
		\left\{
		\begin{aligned}
			\frac{\d\widetilde{\u}}{\d t}+\nu \A\widetilde{\u}+\B_{N}(\widetilde{\u})&=\mathcal{P}\f_{\infty} +\widetilde{\u}\circ\frac{\d \W}{\d t} , \\
			\widetilde{\u}(0)&=\widetilde{\u}_{0}.
		\end{aligned}
		\right.
	\end{equation}
	Let $\widetilde{\v}(t,\omega)=e^{-z(\vartheta_{t}\omega)}\widetilde{\u}(t,\omega)$. Then, the system \eqref{A-SNSE-M} can be written in the following pathwise deterministic system: for $ t> \tau, \tau\in\R ,$
	\begin{equation}\label{A-CNSE-M}
		\left\{
		\begin{aligned}
			\frac{\d\widetilde{\v}}{\d t}+\nu \A\widetilde{\v}+ F_{N}(e^{z(\vartheta_{t}\omega)}\|\wi \v\|_{\V}) \cdot e^{z(\vartheta_{t}\omega)}\B\big(\widetilde{\v}\big)&=\mathcal{P}\f_{\infty} e^{-z(\vartheta_{t}\omega)} + \sigma z(\vartheta_t\omega)\widetilde{\v}, \\
			\widetilde{\v}(0)&=\widetilde{\v}_{0} = e^{-z(\omega)}\widetilde{\u}_{0},
		\end{aligned}
		\right.
	\end{equation}
	in $\V^*$. The following lemma shows the well-posedness result for the system \eqref{CNSE-M}.
	\begin{lemma}\label{Soln}
		Suppose that $\f\in\mathrm{L}^2_{\mathrm{loc}}(\R;\L^2(\mathcal{O}))$. For each $(\tau,\omega,\v_{\tau})\in\R\times\Omega\times\H$, the system \eqref{CNSE-M} has a unique weak solution $\v(\cdot,\tau,\omega,\v_{\tau})\in\mathrm{C}([\tau,+\infty);\H)\cap\mathrm{L}^2_{\mathrm{loc}}(\tau,+\infty;\V)$ such that $\v$ is continuous with respect to the  initial data.
	\end{lemma}
	\begin{proof}
		Since the system \eqref{CNSE-M} is a deterministic system for each $\omega\in\Omega$, the proof of this lemma can be completed by following similar steps as used in the work \cite{Kinra+Mohan_UP}.  For the convenience of the readers, we provide the details of the proof.
         The proof is divided into several steps:
				\vskip 2mm
                \noindent
                \textbf{Step I:} \textit{Existence of weak solutions.} Let $T>\tau$ and let $\{\boldsymbol{w}_1,\ldots,\boldsymbol{w}_n,\ldots\}$ be an orthonormal basis in $\H$ belonging to $\V$. For any $n\geq 1$, let $\widehat{\H}_n:=\mathrm{span}\{\boldsymbol{w}_1,\ldots,\boldsymbol{w}_n\}$ and $\widehat{\mathrm{P}}_n$ denote the orthogonal projection from $\H$ onto $\widehat{\H}_n$ given by
		$$\widehat{\mathrm{P}}_n \x=\sum\limits_{i=1}^n(\x,\boldsymbol{w}_i)\boldsymbol{w}_i, \ \ \x\in\H.$$
		Let us consider the following system of ODEs in the finite dimensional space $\widehat{\H}_n$:
		\begin{equation}\label{SL1}
		\left\{
		\begin{aligned}
		\frac{\d\v^n(t)}{\d t}&=-\mathcal{G}_{N,n}(\v^n(t))+\f_n e^{-z(\vartheta_{t}\omega)} + \sigma z(\vartheta_t\omega)\v^n(t) , \\
		\v^n(\tau)&=\v_{0n},
		\end{aligned}
		\right.
		\end{equation}
		where $\mathcal{G}_{n,N}(\v^n)=\widehat{\mathrm{P}}_n\mathcal{G}_{N}(\v^n)$ and $\mathcal{G}_{N}(\v^n)=\A\v^n+F_{N}(e^{z(\vartheta_{t}\omega)}\|\v^n\|_{\V}) \cdot e^{z(\vartheta_{t}\omega)}\B\big(\v^n\big)$, $\f_n=\widehat{\mathrm{P}}_n[\mathcal{P}\f]$. We see that the system \eqref{SL1} has a unique local solution $\v^n\in\mathrm{C}([\tau, T^*_n];\H_n)$, for some $\tau<T^*_n\leq T$. The following a priori estimates show that the time $T^*_n$ can be extended to time $T$. Taking the inner product with $\v^n(\cdot)$ to the first equation of \eqref{SL1}, we obtain for a.e.	$t\in [\tau, T]$ 
		\begin{align}\label{SL2}
		&\frac{1}{2}\frac{\d}{\d t} \|\v^n(t)\|^2_{\H} +\frac{3\nu}{4}\|\v^n(t)\|^2_{\V}+\frac{\nu}{4}\|\v^n(t)\|^2_{\V}\nonumber\\&= e^{-z(\vartheta_{t}\omega)}\left(\f(t),\v^n(t)\right)+\sigma z(\vartheta_{t}\omega)\|\v^n(t)\|^2_{\H}\nonumber\\&\leq \frac{\nu\lambda}{4}\|\v^n(t)\|^2_{\H}+\frac{e^{2\left|z(\vartheta_{t}\omega)\right|}}{\nu\lambda}\|\f(t)\|^2_{\L^2(\mathcal{O})}+\sigma z(\vartheta_{t}\omega)\|\v^n(t)\|^2_{\H},
		\end{align}	
		where we have used Young's inequality. Now, by \eqref{poin} and \eqref{SL2}, we know
		\begin{align}\label{SL3}
		\frac{\d}{\d t} \|\v^n(t)\|^2_{\H}+ \left(\nu\lambda-2\sigma z(\vartheta_{t}\omega)\right)\|\v^n(t)\|^2_{\H}+\frac{\nu}{2}\|\v^n(t)\|^2_{\V} \leq Ce^{2\left|z(\vartheta_{t}\omega)\right|}\|\f(t)\|^2_{\L^2(\mathcal{O})}.
		\end{align}	
		Applying Gronwall's lemma to \eqref{SL3}, for all $t\in [\tau, T]$, we get
		\begin{align}\label{SL4}
		&\|\v^n(t)\|^2_{\H} +\frac{\nu}{2}\int^{t}_{\tau}e^{\int^{s}_{t}\left(\nu\lambda-2\sigma z(\vartheta_{r}\omega)\right)dr}\|\v^n(s)\|^2_{\V}ds \nonumber\\ 
		&\leq e^{-\int^{t}_{\tau}\left(\nu\lambda-2\sigma z(\vartheta_{s}\omega)\right)ds}\|\v^n(\tau)\|^2_{\H}+C\int^{t}_{\tau}e^{\int^{s}_{t}\left(\nu\lambda-2\sigma z(\vartheta_{r}\omega)\right)dr}e^{2\left|z(\vartheta_{t}\omega)\right|}\|\f(s)\|^2_{\L^2{\mathcal(O)}}ds.
		\end{align}	
		Since $\|\v^n(\tau)\|^2_{\H}\leq \|\v(\tau)\|^2_{\H}$, it follows from \eqref{SL4} that 
		\begin{align}\label{SL5}
		\{\v^n\}_{n\in\N} \text{ is a bounded sequence in } \mathrm{L}^{\infty}(\tau,T;\H)\cap \mathrm{L}^2(\tau,T;\V).
		\end{align}		
		For any arbitrary element $\psi\in\mathrm{L}^2(\tau,T;\V)$, using \eqref{HI+SE}, \eqref{FN1} and \eqref{poin}, we have from \eqref{SL1} that
		\begin{align*}
		&\left|\int^{T}_{\tau}\left\langle \frac{\d \v^n(t)}{\d t},\psi(t)\right\rangle \d t\right| \\
		&\leq \int^{T}_{\tau}\Big[\nu\left|(\nabla\v^n(t),\nabla\psi(t))\right|+\left|F_{N}(e^{z(\vartheta_{t}\omega)}\|\v^n\|_{\V}) \cdot e^{z(\vartheta_{t}\omega)}b\big(\v^n(t),\psi(t),\v^n(t)\big)\right|+\left|(\f_n(t),\psi(t))\right| \\
		&\quad +\left|(\sigma z(\vartheta_{t}\omega)\v^n(t),\psi(t))\right|\Big]\d t \\
		&\leq C\int^{T}_{\tau}\Big[\|\v^n(t)\|_{\V}\|\psi(t)\|_{\V}+F_{N}(e^{z(\vartheta_{t}\omega)}\|\v^n\|_{\V}) \cdot \|\v^n(t)\|^2_{\V}\|\psi(t)\|_{\V}+\|\f(t)\|_{\L^2(\mathcal{O})}\|\psi(t)\|_{\V} \\
		&\qquad +\left|z(\vartheta_{t}\omega)\right|\|\v^n\|_{\V}\|\psi(t)\|_{\V}\Big]\d t \\
		&\leq C\int^{T}_{\tau}\Big[\|\v^n(t)\|_{\V}+Ne^{-z(\vartheta_{t}\omega)}\|\v^n(t)\|_{\V}+\|\f(t)\|_{\L^2(\mathcal{O})} +\left|z(\vartheta_{t}\omega)\right|\|\v^n\|_{\V}\Big]\|\psi(t)\|_{\V}\d t \\
		&\leq C\Big[\|\v^n\|_{\L^2(\tau,T;\V)}+\|\f\|_{\L^2(\tau,T;\L^2(\mathcal{O}))}\Big]\|\psi\|_{\L^2(\tau,T;\V)},
		\end{align*}		
		which implies $\frac{\d \v^n}{\d t}\in\mathrm{L}^2(\tau,T;\V')$ and $\mathcal{G}_{n,N}(\v^n)\in\mathrm{L}^2(\tau,T;\V')$. Using \eqref{SL5} and the \textit{Banach-Alaoglu theorem}, we infer the existence of an element $\v\in\mathrm{L}^{\infty}(\tau,T;\H)\cap\mathrm{L}^2(\tau,T;\V)$ with $\frac{\d \v}{\d t}\in\mathrm{L}^2(\tau,T;\V')$, $\mathcal{G}_{0,N}\in\mathrm{L}^2(\tau,T;\V')$ and a subsequence of $\{\v^n\}$ (which is not
		relabeled) such that
		\begin{align}
		\v^n &\stackrel{w^*}{\rightharpoonup} \v && \text { in } \quad \mathrm{L}^{\infty}(\tau,T;\H), \label{SL6}\\
		\v^n &\stackrel{w}{\rightharpoonup} \v && \text { in } \quad \mathrm{L}^2(\tau,T;\V), \label{SL7}\\
		\frac{\d \v^n}{\d t} &\stackrel{w}{\rightharpoonup} \frac{\d \v}{\d t} && \text { in } \quad \mathrm{L}^2(\tau,T;\V'), \label{SL8}\\
		\mathcal{G}_{n,N}\left(\v^n\right) &\stackrel{w}{\rightharpoonup} \mathcal{G}_{0,N} && \text { in } \quad \mathrm{L}^2(\tau,T;\V'). \label{SL9}
		\end{align}
		Note that $\f_n \to \mathcal{P}\f$ in $\mathrm{L}^2(\tau,T;\H)$. Therefore, passing to the limit of \eqref{SL1} as $n \to +\infty$, the limit $\v(\cdot)$ satisfies:
		\begin{equation}\label{SL10}
		\left\{
		\begin{aligned}
		\frac{\d\v}{\d t}&=-\mathcal{G}_{0,N}+\mathcal{P}\f e^{-z(\vartheta_{t}\omega)} + \sigma z(\vartheta_t\omega)\v && \text { in } \ \mathrm{L}^2(\tau,T;\V'), \\
		\v(\tau)&=\v_{0} && \text { in } \ \H.
		\end{aligned}
		\right.
		\end{equation}
		Now, we need to prove that $\mathcal{G}_{0,N}=\mathcal{G}_{N}(\v)$. Since $\v\in\mathrm{L}^2(\tau,T;\V)$ and $\frac{\d\v}{\d t}\in\mathrm{L}^2(\tau,T;\V')$, we then deduce (cf. \cite[Ch. III, Lemma 1.2]{Temam2}) $\v\in\mathrm{C}([\tau,T];\H)$, the real-valued function $t\mapsto \|\v(t)\|^2_{\H}$ is absolutely continuous and the following	equality is satisfied: 
		\begin{align}\label{SL11}
		\frac{\d}{\d t}\|\v(t)\|^2_{\H}=2\left\langle \frac{\d \v(t)}{\d t},\v(t)\right\rangle, \ \  \text{ for a.e. } t\in [\tau,T].
		\end{align}
		Hence, we have the following equality for any $\eta>0$:
		\begin{align}\label{SL12}
		&e^{-2\eta t}\|\v(t)\|^2_{\H}+2\int^{t}_{\tau}e^{-2\eta s}\langle \mathcal{G}_{0, N}(s)-\mathcal{P}\f e^{-z(\vartheta_{s}\omega)}+(\eta-\sigma z(\vartheta_s\omega))\v(s),\v(s)\rangle \d s  =e^{-2\eta \tau}\|\v(\tau)\|^2_{\H},
		\end{align}
		for all $t\in[\tau,T]$. Similarly to \eqref{SL12}, for the system \eqref{SL1}, we obtain 
		\begin{align}\label{SL13}
		& e^{-2\eta t}\|\v^n(t)\|^2_{\H}+2\int^{t}_{\tau}e^{-2\eta s}\langle \mathcal{G}_{N}(\v^n(t))-\f_n e^{-z(\vartheta_{s}\omega)}+(\eta-\sigma z(\vartheta_s\omega))\v^n(s),\v^n(s)\rangle \d s \nonumber\\
		&=e^{-2\eta \tau}\|\v^n(\tau)\|^2_{\H},
		\end{align}	
		for all $t\in[\tau,T]$.
        
				\noindent
		\textbf{Claim I:} \textit{The operator $\mathcal{G}_{N}: \V \to \V'$ is demicontinuous.} Let us take a sequence $\u^n\to \u$ in $\V$, that is, $\|\u^n-\u\|_{\V} \to 0$ as $n\to +\infty$. For any $\psi\in\V$, we consider
		\begin{align}\label{SL14}
		& \langle \mathcal{G}_{N}(\u^n)-\mathcal{G}_{N}(\u),\psi\rangle \nonumber\\
		&=\nu\langle \A\u^n-\A\u,\psi\rangle +\langle F_{N}(e^{z(\vartheta_{t}\omega)}\|\u^n\|_{\V}) \cdot e^{z(\vartheta_{t}\omega)}\B\big(\u^n\big)-F_{N}(e^{z(\vartheta_{t}\omega)}\|\u\|_{\V}) \cdot e^{z(\vartheta_{t}\omega)}\B\big(\u\big),\psi\rangle.
		\end{align}	
		We estimate the terms on the right-hand side of \eqref{SL14} as follows:
		\begin{align}\label{SL15}
		\nu|\langle \A\u^n-\A\u,\psi\rangle|=\nu|(\nabla(\u^n-\u),\nabla\psi)| \leq \nu\|\u^n-\u\|_{\V}\|\psi\|_{\V} \to 0, \ \text{ as } n\to+\infty.
		\end{align}	
		For the second term, by \eqref{HI+SE}, \eqref{FN1}, \eqref{FN2} and \eqref{poin}, we have
		\begin{align}\label{SL16}
		&\left|\langle F_{N}(e^{z(\vartheta_{t}\omega)}\|\u^n\|_{\V}) \cdot e^{z(\vartheta_{t}\omega)}\B\big(\u^n\big)-F_{N}(e^{z(\vartheta_{t}\omega)}\|\u\|_{\V}) \cdot e^{z(\vartheta_{t}\omega)}\B\big(\u\big),\psi\rangle\right| \nonumber\\
		&=\left|F_{N}(e^{z(\vartheta_{t}\omega)}\|\u^n\|_{\V}) \cdot e^{z(\vartheta_{t}\omega)}b\big(\u^n,\u^n,\psi\big)-F_{N}(e^{z(\vartheta_{t}\omega)}\|\u\|_{\V}) \cdot e^{z(\vartheta_{t}\omega)}b\big(\u,\u,\psi\big)\right| \nonumber\\
		&\leq \left|\left[F_{N}(e^{z(\vartheta_{t}\omega)}\|\u^n\|_{\V})-F_{N}(e^{z(\vartheta_{t}\omega)}\|\u\|_{\V})\right]e^{z(\vartheta_{t}\omega)}b\big(\u^n,\psi,\u^n\big)\right| \nonumber\\
		&\quad + \left|F_{N}(e^{z(\vartheta_{t}\omega)}\|\u\|_{\V}) \cdot e^{z(\vartheta_{t}\omega)} \left[b\big(\u^n-\u,\psi,\u^n-\u\big)+b\big(\u^n-\u,\psi,\u\big)+b\big(\u,\psi,\u^n-\u\big)\right]\right| \nonumber\\
		&\leq \frac{C}{N}F_{N}(e^{z(\vartheta_{t}\omega)}\|\u^n\|_{\V})F_{N}(e^{z(\vartheta_{t}\omega)}\|\u\|_{\V}) \cdot e^{2z(\vartheta_{t}\omega)}\|\u^n-\u\|_{\V}\|\u^n\|^2_{\V}\|\psi\|_{\V} \nonumber\\
		&\quad + CF_{N}(e^{z(\vartheta_{t}\omega)}\|\u\|_{\V}) \cdot e^{z(\vartheta_{t}\omega)}\left[\|\u^n-\u\|^2_{\V}+2\|\u^n-\u\|_{\V}\|\u\|_{\V}\right]\|\psi\|_{\V}  \nonumber\\
		&\leq Ce^{z(\vartheta_{t}\omega)}\|\u^n-\u\|_{\V}\|\u^n\|_{\V}\|\psi\|_{\V} + Ce^{z(\vartheta_{t}\omega)}\left[\|\u^n-\u\|^2_{\V}+2\|\u^n-\u\|_{\V}\|\u\|_{\V}\right]\|\psi\|_{\V} \nonumber\\
		& \to 0, \ \text{ as } n\to+\infty.
		\end{align}
		Then \eqref{SL14}-\eqref{SL16} imply that $\langle \mathcal{G}_{N}(\u^n)-\mathcal{G}_{N}(\u),\psi\rangle \to 0$ as $n\to+\infty$, for all $\psi\in\V$. Hence the operator $\mathcal{G}_{N}: \V \to \V'$ is demicontinuous, which implies that $\mathcal{G}_{N}(\cdot)$ is hemicontinuous also.
		
		\noindent
		 \textbf{Claim II:} \textit{The operator $\mathcal{G}_{N}(\cdot)$ satisfies}
		\begin{align}\label{SL17}
		\langle \mathcal{G}_{N}(\v_1)-\mathcal{G}_{N}(\v_2),\v_1-\v_2\rangle + \frac{(2\widehat{C}N)^4}{\nu^3}\|\v_1-\v_2\|^2_{\H} \geq \frac{\nu}{2}\|\v_1-\v_2\|^2_{\V} \geq 0,
		\end{align}	
		\textit{for any $\v_1,\v_2\in\V$, where $\widehat{C}$ is the constant in \eqref{HI+SE}.} 
		
		 We find that 
		\begin{align}\label{SL18}
		\langle \nu\A\v_1-\nu\A\v_2,\v_1-\v_2\rangle = \nu\|\v_1-\v_2\|^2_{\V}.
		\end{align}	
		Now, using \eqref{b0}, \eqref{HI+SE}, \eqref{FN1}, \eqref{FN2} and Young's inequality, we estimate
		\begin{align}\label{SL19}
		&\left|\langle F_{N}(e^{z(\vartheta_{t}\omega)}\|\v_1\|_{\V}) \cdot e^{z(\vartheta_{t}\omega)}\B\big(\v_1\big)-F_{N}(e^{z(\vartheta_{t}\omega)}\|\v_2\|_{\V}) \cdot e^{z(\vartheta_{t}\omega)}\B\big(\v_2\big),\v_1-\v_2\rangle\right| \nonumber\\
		&\leq \left|\left[F_{N}(e^{z(\vartheta_{t}\omega)}\|\v_1\|_{\V})-F_{N}(e^{z(\vartheta_{t}\omega)}\|\v_2\|_{\V})\right]e^{z(\vartheta_{t}\omega)}b\big(\v_1,\v_2,\v_1-\v_2\big)\right| \nonumber\\
		&\quad + \left|F_{N}(e^{z(\vartheta_{t}\omega)}\|\v_2\|_{\V}) \cdot e^{z(\vartheta_{t}\omega)}b\big(\v_1-\v_2,\v_2,\v_1-\v_2\big)\right| \nonumber\\
		&\leq \frac{\widehat{C}}{N}F_{N}(e^{z(\vartheta_{t}\omega)}\|\v_1\|_{\V})F_{N}(e^{z(\vartheta_{t}\omega)}\|\v_2\|_{\V}) \cdot e^{2z(\vartheta_{t}\omega)}\|\v_1\|_{\V}\|\v_2\|_{\V}\|\v_1-\v_2\|^{\frac{3}{2}}_{\V}\|\v_1-\v_2\|^{\frac{1}{2}}_{\H} \nonumber\\
		&\quad + \widehat{C}|F_{N}(e^{z(\vartheta_{t}\omega)}\|\v_2\|_{\V}) \cdot e^{z(\vartheta_{t}\omega)}\|\v_2\|_{\V}\|\v_1-\v_2\|^{\frac{3}{2}}_{\V}\|\v_1-\v_2\|^{\frac{1}{2}}_{\H} \nonumber\\
		&\leq 2\widehat{C}N\|\v_1-\v_2\|^{\frac{3}{2}}_{\V}\|\v_1-\v_2\|^{\frac{1}{2}}_{\H} \leq \frac{\nu}{2}\|\v_1-\v_2\|^2_{\V}+\frac{(2\widehat{C}N)^4}{\nu^3}\|\v_1-\v_2\|^2_{\H},
		\end{align}	
		where $\widehat{C}$ is the constant in \eqref{HI+SE}. From \eqref{SL19}, we deduce
		\begin{align*}
		&\langle F_{N}(e^{z(\vartheta_{t}\omega)}\|\v_1\|_{\V}) \cdot e^{z(\vartheta_{t}\omega)}\B\big(\v_1\big)-F_{N}(e^{z(\vartheta_{t}\omega)}\|\v_2\|_{\V}) \cdot e^{z(\vartheta_{t}\omega)}\B\big(\v_2\big),\v_1-\v_2\rangle \nonumber\\
		&\geq -\frac{\nu}{2}\|\v_1-\v_2\|^2_{\V}-\frac{(2\widehat{C}N)^4}{\nu^3}\|\v_1-\v_2\|^2_{\H}.
		\end{align*}
		This along with \eqref{SL18} implies \eqref{SL17}.
		
		\noindent
		 \textbf{Minty-Browder technique:} Note that $\v^n(\tau)=\mathrm{P}_n\v(\tau)$, and hence the initial value $\v^n(\tau)$	converges strongly in $\H$, that is, we have
		\begin{align}\label{SL20}
		\lim\limits_{n\to+\infty}\|\v^n(\tau)-\v(\tau)\|_{\H}=0.
		\end{align}
		For any $\psi\in\mathrm{L}^{\infty}(\tau,T;\widehat{\H}_m)$ with $m<n$, by \eqref{SL17}, we obtain for $\eta=\frac{(2\widehat{C}N)^4}{\nu^3},$
		\begin{align}\label{SL21}
		\int^{T}_{\tau} e^{-2\eta t}\left[\langle \mathcal{G}_{N}(\psi(t))-\mathcal{G}_{N}(\v^n(t)), \psi(t)-\v^n(t)\rangle + \eta(\psi(t)-\v^n(t),\psi(t)-\v^n(t))\right]\d t \geq 0.
		\end{align}		
		Due to \eqref{SL13} and \eqref{SL21}, we have
		\begin{align}\label{SL22}
		&\int^{T}_{\tau} e^{-2\eta t}\langle \mathcal{G}_{N}(\psi(t))+\eta\psi(t),\psi(t)-\v^n(t)\rangle \d t \nonumber\\
		&\geq \int^{T}_{\tau} e^{-2\eta t}\langle \mathcal{G}_{N}(\v^n(t))+\eta\v^n(t),\psi(t)-\v^n(t)\rangle \d t \nonumber\\
		&= \int^{T}_{\tau} e^{-2\eta t}\langle \mathcal{G}_{N}(\v^n(t))+\eta\v^n(t),\psi(t)\rangle \d t + \frac{1}{2}\left[e^{-2\eta T}\|\v^n(T)\|^2_{\H}-e^{-2\eta\tau}\|\v^n(\tau)\|^2_{\H}\right] \nonumber\\
		&\quad -\int^{T}_{\tau} e^{-2\eta t}\langle \f_n e^{-z(\vartheta_{t}\omega)}+\sigma z(\vartheta_s\omega)\v^n(t),\v^n(t)\rangle \d t.
		\end{align}		
		Taking limit infimum on both sides of \eqref{SL22}, we deduce
		\begin{align}\label{SL23}
		&\int^{T}_{\tau} e^{-2\eta t}\langle \mathcal{G}_{N}(\psi(t))+\eta\psi(t),\psi(t)-\v(t)\rangle \d t \nonumber\\
		&\geq \int^{T}_{\tau} e^{-2\eta t}\langle \mathcal{G}_{0,N}(t)+\eta\v(t),\psi(t)\rangle \d t + \frac{1}{2}\liminf\limits_{n\to+\infty}\left[e^{-2\eta T}\|\v^n(T)\|^2_{\H}-e^{-2\eta\tau}\|\v^n(\tau)\|^2_{\H}\right] \nonumber\\
		&\quad -\int^{T}_{\tau} e^{-2\eta t}\langle \mathcal{P}\f e^{-z(\vartheta_{t}\omega)}+\sigma z(\vartheta_s\omega)\v(t),\v(t)\rangle \d t \nonumber\\
		&\geq \int^{T}_{\tau} e^{-2\eta t}\langle \mathcal{G}_{0,N}(t)+\eta\v(t),\psi(t)\rangle \d t + \frac{1}{2}\left[e^{-2\eta T}\|\v(T)\|^2_{\H}-e^{-2\eta\tau}\|\v(\tau)\|^2_{\H}\right] \nonumber\\
		&\quad -\int^{T}_{\tau} e^{-2\eta t}\langle \mathcal{P}\f e^{-z(\vartheta_{t}\omega)}+\sigma z(\vartheta_s\omega)\v(t),\v(t)\rangle \d t,
		\end{align}
		where we have used the lower semicontinuity property of the $\H$-norm and the strong convergence of the initial data \eqref{SL20} in the final inequality. Furthermore, using the equality \eqref{SL12} in \eqref{SL23}, we have
		\begin{align}\label{SL24}
		&\int^{T}_{\tau} e^{-2\eta t}\langle \mathcal{G}_{N}(\psi(t))+\eta\psi(t),\psi(t)-\v(t)\rangle \d t \nonumber\\
		&\geq \int^{T}_{\tau} e^{-2\eta t}\langle \mathcal{G}_{0,N}(t)+\eta\v(t),\psi(t)\rangle \d t - \int^{T}_{\tau} e^{-2\eta t}\langle \mathcal{G}_{0,N}(t)+\eta\v(t),\v(t)\rangle \d t \nonumber\\
		&\geq \int^{T}_{\tau} e^{-2\eta t}\langle \mathcal{G}_{0,N}(t)+\eta\v(t),\psi(t)-\v(t)\rangle \d t.
		\end{align}
		Note that the estimate \eqref{SL24} holds true for any $\psi\in\mathrm{L}^{\infty}(\tau,T;\widehat{\H}_m)$, $m\in\N$, since the inequality given in \eqref{SL24} is independent of both $m$ and $n$. Using a density argument, one can show	that the inequality \eqref{SL24} remains true for any $\psi\in\mathrm{L}^{\infty}(\tau,T;\H)\cap\mathrm{L}^{2}(\tau,T;\V)$.	
		
		 Taking $\psi=\v+\delta w$, $\delta>0$, where $w\in\mathrm{L}^{\infty}(\tau,T;\H)\cap\mathrm{L}^{2}(\tau,T;\V)$, and substituting for $\psi$ in \eqref{SL24}, we get
		\begin{align*}
		\int^{T}_{\tau} e^{-2\eta t}\langle \mathcal{G}_{N}(\v+\delta w)-\mathcal{G}_{0,N}(t)+\delta\eta w,\delta w\rangle \d t \geq 0,
		\end{align*}	
		and consequently
		\begin{align}\label{SL25}
		\int^{T}_{\tau} e^{-2\eta t}\langle \mathcal{G}_{N}(\v+\delta w)-\mathcal{G}_{0,N}(t)+\delta\eta w,w\rangle \d t \geq 0,
		\end{align}	
		Using the hemicontinuity property of the operator $\mathcal{G}_{N}(\cdot)$, and then passing $\delta\to 0$ in \eqref{SL25}, we obtain
		\begin{align}\label{SL26}
		\int^{T}_{\tau} e^{-2\eta t}\langle \mathcal{G}_{N}(\v)-\mathcal{G}_{0,N}(t),w\rangle \d t \geq 0,
		\end{align}		
		for any $w\in\mathrm{L}^{\infty}(\tau,T;\H)\cap\mathrm{L}^{2}(\tau,T;\V)$. Therefore, by \eqref{SL26} we find that $\mathcal{G}_{N}(\v(\cdot))=\mathcal{G}_{0,N}(\cdot)$, this along with \eqref{SL10} implies that $\v$ is the weak solution of the system \eqref{CNSE-M}.

		\noindent
 \textbf{Step II.} \textit{Uniqueness:} Let $\v_1$ and $\v_2$ be two weak solutions of the system \eqref{CNSE-M} with the same initial conditions. We write $\mathfrak{y} = \v_1-\v_2$, then by \textbf{Step I}, we have that $\mathfrak{y}\in \C([\tau,T]; \H) \cap \mathrm{L}^2(\tau,T;\mathbb{V})$ and satisfies
\begin{equation}\label{3.50}
\begin{cases}
\displaystyle{\frac{\d\mathfrak{y}}{\d t} }= -\nu \A\mathfrak{y} - F_N(e^{z(\vartheta_t\omega)}\|\v_1\|_{\V})\cdot e^{z(\vartheta_t\omega)}\B(\v_1) + F_N(e^{z(\vartheta_t\omega)}\|\v_2\|_{\V})\cdot e^{z(\vartheta_t\omega)}\B(\v_2),\\
\quad\quad- \sigma z(\vartheta_t\omega)(\v_1 - \v_2),\\
\mathfrak{y}(\tau) = \boldsymbol{0}, 
\end{cases}
\end{equation}
in the weak sense. 

 Taking the inner product in \eqref{3.50} by $\mathfrak{y}$ and using the relation in \eqref{BN-diff}, we arrive at
\begin{align}
\dfrac{1}{2}\dfrac{\d}{\d t}\|\mathfrak{y}\|^2_{\mathbb{H}}& = -\nu \|\mathfrak{y}\|^2_{\V}  + \sigma z(\vartheta_t\omega)\|\mathfrak{y}\|^2_{\mathbb{H}} - F_N\big(e^{z(\vartheta_t\omega)}\|\v_1\|_{\V}\big) e^{z(\vartheta_t\omega)}b(\mathfrak{y}, \v_1,\mathfrak{y})\notag\\
&\quad\, - \bigg[F_N\big(e^{z(\vartheta_t\omega)}\|\v_1\|_{\V}\big) - F_N\big(e^{z(\vartheta_t\omega)}\|\v_2\|_{\V}\big)\bigg]e^{z(\vartheta_t\omega)}b(\v_2,\v_1, \mathfrak{y}).\label{3.50a}
\end{align}
By using \eqref{FN1} and Young's inequality, we have
\begin{align}
 \bigg| F_N\big(e^{z(\vartheta_t\omega)}\|\v_1\|_{\V}\big) e^{z(\vartheta_t\omega)}b(\mathfrak{y}, \v_1,\mathfrak{y})\bigg|&\leq C F_N\big(e^{z(\vartheta_t\omega)}\|\v_1\|_{\V}\big) e^{z(\vartheta_t\omega)}\|\mathfrak{y}\|_{\V} \|\v_1\|_{\V}\|\mathfrak{y}\|^{\frac{1}{2}}_{\H}\|\mathfrak{y}\|^{\frac{1}{2}}_{\V}\notag\\ 
&\leq CN \|\mathfrak{y}\|^{\frac{1}{2}}_{\H}\|\mathfrak{y}\|^{\frac{3}{2}}_{\V}  \leq \dfrac{\nu}{2} \|\mathfrak{y}\|^{2}_{\V} + CN^4\|\mathfrak{y}\|^{2}_{\H},\label{3.52}
\end{align}
and by using ~\eqref{FN2} and Young's inequality again, we also obtain
\begin{align}
& \bigg|\bigg[F_N\big(e^{z(\vartheta_t\omega)}\|\v_1\|_{\V}\big) - F_N\big(e^{z(\vartheta_t\omega)}\|\v_2\|_{\V}\big)\bigg]e^{z(\vartheta_t\omega)}b(\v_2,\v_1, \mathfrak{y}) \bigg|\notag\\ 
&\leq \dfrac{C}{N}F_N\big(e^{z(\vartheta_t\omega)}\|\v_1\|_{\V}\big)\cdot F_N\big(e^{z(\vartheta_t\omega)}\|\v_2\|_{\V}\big) e^{2z(\vartheta_t\omega)}\|\v_1 -\v_2\|_{\V}\|\v_2\|_{\V}\|\v_1\|_{\V}\|\mathfrak{y}\|^{\frac{1}{2}}_{\H}\|\mathfrak{y}\|^{\frac{1}{2}}_{\V}\notag\\
&\leq CN \|\mathfrak{y}\|^{\frac{1}{2}}_{\H}\|\mathfrak{y}\|^{\frac{3}{2}}_{\V}  \leq \dfrac{\nu}{2} \|\mathfrak{y}\|^{2}_{\V} + CN^4\|\mathfrak{y}\|^{2}_{\H},\label{3.53}
\end{align}
		Thus, it follows from ~\eqref{3.50} and ~\eqref{3.50a}-\eqref{3.53}, we deduce that
\begin{equation*}
\dfrac{1}{2}\dfrac{\d}{\d t}\|\mathfrak{y}\|^2_{\H} \leq \left(\sigma z(\vartheta_t\omega) + 2CN^4\right)\|\mathfrak{y}\|^2_{\H},
\end{equation*}
which together with the Gronwall inequality shows that
\begin{equation}\label{4355}
\|\mathfrak{y}(t)\|^2_{\H} \leq \|\mathfrak{y}(\tau)\|^2_{\H} \cdot e^{2(\sigma z(\vartheta_t\omega) + 2CN^4)t},\ \text{ for all } t\in (\tau,+\infty).
\end{equation}
If $\mathfrak{y}(\tau)=\boldsymbol{0}$, from \eqref{4355} we have the uniqueness of solutions of the problem ~\eqref{CNSE-M} as desired. Equation \eqref{4355} also tells us that the weak solutions are continuous with respect to the  initial data.  The proof of the lemma is completed.	
\end{proof}

    The energy inequality in the following lemma will be used frequently in the rest of the paper.
	\begin{lemma}
		For $\f\in\mathrm{L}^2_{\mathrm{loc}}(\R;\L^2(\mathcal{O}))$, the solution of \eqref{CNSE-M} satisfies the following inequality:
		\begin{align}\label{EI1}
			\frac{\d}{\d t} \|\v\|^2_{\H}+ \left(\nu\lambda-2\sigma z(\vartheta_{t}\omega)\right)\|\v\|^2_{\H}+\frac{\nu}{2}\|\v\|^2_{\V} \leq \frac{2e^{2\left|z(\vartheta_{t}\omega)\right|}}{\nu\lambda}\|\f\|^2_{\L^2(\mathcal{O})}.
		\end{align}
	\end{lemma}
	\begin{proof}
		From the first equation of the system \eqref{CNSE-M} and with the help of   \eqref{b0}, we obtain
		\begin{align*}
			\frac{1}{2}\frac{\d}{\d t} \|\v\|^2_{\H} +\frac{3\nu}{4}\|\v\|^2_{\V}+\frac{\nu}{4}\|\v\|^2_{\V}&= e^{-z(\vartheta_{t}\omega)}\left(\f,\v\right)+\sigma z(\vartheta_{t}\omega)\|\v\|^2_{\H}\nonumber\\&\leq \frac{\nu\lambda}{4}\|\v\|^2_{\H}+\frac{e^{2\left|z(\vartheta_{t}\omega)\right|}}{\nu\lambda}\|\f\|^2_{\L^2(\mathcal{O})}+\sigma z(\vartheta_{t}\omega)\|\v\|^2_{\H}.
		\end{align*}	
		Now, using \eqref{poin} in the second term on the left hand side of the above inequality, one can conclude the proof.
	\end{proof}

	The next result shows the Lusin continuity of the solution mapping of \eqref{CNSE-M} with respect to $\omega\in\Omega$.
	\begin{proposition}\label{LusinC}
		Suppose that $\f\in\mathrm{L}^2_{\mathrm{loc}}(\R;\L^2(\mathcal{O}))$ and Assumption \ref{assumpO} is satisfied. For each $N\in\N$, the mapping $\omega\mapsto\v(t,\tau,\omega,\v_{\tau})$ $($the solution of \eqref{CNSE-M}$)$ is continuous from $(\Omega_{M},d_{\Omega_M})$ to $\H$, uniformly in $t\in[\tau,\tau+T]$ with $T>0$.
	\end{proposition}
	
	\begin{proof}
		Assume that $\omega_k,\omega_0\in\Omega_M,$  $N\in\mathbb{N}$ are  such that $d_{\Omega_M}(\omega_k,\omega_0)\to0$ as $k\to+\infty$. Let $\mathscr{Y}^k(\cdot):=\v^k(\cdot)-\v^0(\cdot),$ where $\v^k(t)=\v(t,\tau,\omega_k,\v_{\tau})$ and $\v^0(t)=\v(t,\tau,\omega_0,\v_{\tau})$ for $t\in[\tau,\tau+T]$. Then, $\mathscr{Y}^k$ satisfies:
		\begin{align}\label{LC1}
			\frac{\d\mathscr{Y}^k}{\d t}
			&=-\nu \A\mathscr{Y}^k
			-F_{N}(e^{z(\vartheta_{t}\omega_k)}\|\v^k\|_{\V})\cdot e^{z(\vartheta_{t}\omega_k)}\B\big(\v^k\big)
			+F_{N}(e^{z(\vartheta_{t}\omega_0)}\|\v^0\|_{\V})\cdot e^{z(\vartheta_{t}\omega_0)}\B\big(\v^0\big)
			\nonumber\\&\quad +\f \left[e^{-z(\vartheta_{t}\omega_k)}-e^{-z(\vartheta_{t}\omega_0)}\right] + \sigma z(\vartheta_t\omega_k)\v^k-\sigma z(\vartheta_t\omega_0)\v^0
			\nonumber\\ &=-\nu \A\mathscr{Y}^k
			-F_{N}(e^{z(\vartheta_{t}\omega_k)}\|\v^k\|_{\V})\cdot e^{z(\vartheta_{t}\omega_k)}\left[\B\big(\mathscr{Y}^k,\v^k\big)+\B\big(\v^0,\mathscr{Y}^k\big)\right]
			\nonumber\\& \quad
			-F_{N}(e^{z(\vartheta_{t}\omega_k)}\|\v^k\|_{\V})\cdot \left[e^{z(\vartheta_{t}\omega_k)}-e^{z(\vartheta_{t}\omega_0)}\right]\cdot\left[\B\big(\v^0,\v^k\big)-\B\big(\v^0,\mathscr{Y}^k\big)\right]
			\nonumber\\ & \quad -\left[F_{N}(e^{z(\vartheta_{t}\omega_k)}\|\v^k\|_{\V})-F_{N}(e^{z(\vartheta_{t}\omega_0)}\|\v^0\|_{\V})\right]\cdot e^{z(\vartheta_{t}\omega_0)}\cdot\left[\B\big(\v^0,\v^k\big)-\B\big(\v^0,\mathscr{Y}^k\big)\right]
			\nonumber\\&\quad +\f \left[e^{-z(\vartheta_{t}\omega_k)}-e^{-z(\vartheta_{t}\omega_0)}\right] + \sigma z(\vartheta_t\omega_k)\v^k-\sigma z(\vartheta_t\omega_0)\v^0,
		\end{align}
		in $\V^{\ast}$ (in the weak sense), where we have also used \eqref{BN-diff} in the final equality. Taking the inner product with $\mathscr{Y}^k(\cdot)$ in \eqref{LC1}, and using \eqref{b0}, we get
		\begin{align}\label{LC2}
			\frac{1}{2}\frac{\d }{\d t}\|\mathscr{Y}^k\|^2_{\H}&=-\nu\|\mathscr{Y}^k\|^2_{\V}+\sigma z(\vartheta_t\omega_k)\|\mathscr{Y}^k\|^2_{\H}-F_{N}(e^{z(\vartheta_{t}\omega_k)}\|\v^k\|_{\V})\cdot e^{z(\vartheta_{t}\omega_k)}\cdot b\big(\mathscr{Y}^k,\v^k,\mathscr{Y}^k\big)
			\nonumber\\& \quad
			-F_{N}(e^{z(\vartheta_{t}\omega_k)}\|\v^k\|_{\V})\cdot \left[e^{z(\vartheta_{t}\omega_k)}-e^{z(\vartheta_{t}\omega_0)}\right]\cdot b\big(\v^0,\v^k,\mathscr{Y}^k\big)
			\nonumber\\ & \quad -\left[F_{N}(e^{z(\vartheta_{t}\omega_k)}\|\v^k\|_{\V})-F_{N}(e^{z(\vartheta_{t}\omega_0)}\|\v^0\|_{\V})\right]\cdot e^{z(\vartheta_{t}\omega_0)}\cdot b\big(\v^0,\v^k,\mathscr{Y}^k\big) \nonumber\\&\quad+\left[e^{-z(\vartheta_{t}\omega_k)}-e^{-z(\vartheta_{t}\omega_0)}\right](\f,\mathscr{Y}^k)+\sigma\left[z(\vartheta_t\omega_k)-z(\vartheta_t\omega_0)\right](\v^0,\mathscr{Y}^k).
		\end{align}
		Using H\"older's and Young's inequalities, and \eqref{poin}, we obtain
		\begin{align}
			\left|\left[e^{-z(\vartheta_{t}\omega_k)}-e^{-z(\vartheta_{t}\omega_0)}\right](\f,\mathscr{Y}^k)\right|& \leq C\left|e^{-z(\vartheta_{t}\omega_k)}-e^{-z(\vartheta_{t}\omega_0)}\right|^2\|\f\|^2_{\L^2(\mathcal{O})}+\frac{\nu}{10}\|\mathscr{Y}^k\|^2_{\V},\label{LC4}\\
			\left|\sigma\left[z(\vartheta_t\omega_k)-z(\vartheta_t\omega_0)\right](\v^0,\mathscr{Y}^k)\right|& \leq C\left|z(\vartheta_t\omega_k)-z(\vartheta_t\omega_0)\right|^2\|\v^0\|^2_{\H}+\frac{\nu}{10}\|\mathscr{Y}^k\|^2_{\V}.\label{LC5}
		\end{align}
		Applying \eqref{HI+SE}, \eqref{FN1}-\eqref{FN2} and Young's inequalities, we estimate
		\begin{align}
			&	\left|F_{N}(e^{z(\vartheta_{t}\omega_k)}\|\v^k\|_{\V})\cdot e^{z(\vartheta_{t}\omega_k)}\cdot b\big(\mathscr{Y}^k,\v^k,\mathscr{Y}^k\big)\right|
			\nonumber\\ &\leq C F_{N}(e^{z(\vartheta_{t}\omega_k)}\|\v^k\|_{\V})\cdot e^{z(\vartheta_{t}\omega_k)}\cdot \|\mathscr{Y}^k\|_{\V}\|\v^k\|_{\V}\|\mathscr{Y}^k\|_{\H}^{\frac12}\|\mathscr{Y}^k\|_{\V}^{\frac12}
			\nonumber\\ &\leq C N \|\mathscr{Y}^k\|_{\H}^{\frac12}\|\mathscr{Y}^k\|_{\V}^{\frac32}
			\leq  \frac{\nu}{10}\|\mathscr{Y}^k\|^2_{\V} +C N^4 \|\mathscr{Y}^k\|_{\H}^{2},\label{LC7}
		\end{align}
		\begin{align}
			&\left|F_{N}(e^{z(\vartheta_{t}\omega_k)}\|\v^k\|_{\V})\cdot \left[e^{z(\vartheta_{t}\omega_k)}-e^{z(\vartheta_{t}\omega_0)}\right]\cdot b\big(\v^0,\v^k,\mathscr{Y}^k\big)\right|
			\nonumber\\ & \leq C F_{N}(e^{z(\vartheta_{t}\omega_k)}\|\v^k\|_{\V}) \cdot \left|1-e^{z(\vartheta_{t}\omega_0)-z(\vartheta_{t}\omega_k)}\right|\cdot e^{z(\vartheta_{t}\omega_k)} \cdot \|\v^0\|_{\V}\|\v^k\|_{\V}\|\mathscr{Y}^k\|_{\H}^{\frac12}\|\mathscr{Y}^k\|_{\V}^{\frac12}
			\nonumber\\ & \leq C N \left|1-e^{z(\vartheta_{t}\omega_0)-z(\vartheta_{t}\omega_k)}\right| \|\v^0\|_{\V}\|\mathscr{Y}^k\|_{\V}
			\leq \frac{\nu}{10}\|\mathscr{Y}^k\|^2_{\V} + C N^2 \left|1-e^{z(\vartheta_{t}\omega_0)-z(\vartheta_{t}\omega_k)}\right|^2 \|\v^0\|_{\V}^2,\label{LC8}
		\end{align}
		and
		\begin{align}
			&\left|\left[F_{N}(e^{z(\vartheta_{t}\omega_k)}\|\v^k\|_{\V})-F_{N}(e^{z(\vartheta_{t}\omega_0)}\|\v^0\|_{\V})\right]\cdot e^{z(\vartheta_{t}\omega_0)}\cdot b\big(\v^0,\v^k,\mathscr{Y}^k\big)\right|
			\nonumber\\&\leq \frac{C}{N} F_{N}(e^{z(\vartheta_{t}\omega_k)}\|\v^k\|_{\V}) F_{N}(e^{z(\vartheta_{t}\omega_0)}\|\v^0\|_{\V})\|e^{z(\vartheta_{t}\omega_k)}\v^k-e^{z(\vartheta_{t}\omega_0)}\v^0\|_{\V}   e^{z(\vartheta_{t}\omega_0)}  \|\v^0\|_{\V}\|\v^k\|_{\V}\|\mathscr{Y}^k\|^{\frac12}_{\H} \|\mathscr{Y}^k\|^{\frac12}_{\V}
			\nonumber\\&\leq C N \|e^{z(\vartheta_{t}\omega_k)}\mathscr{Y}^k+\left(e^{z(\vartheta_{t}\omega_k)}-e^{z(\vartheta_{t}\omega_0)}\right)\v^0\|_{\V}   e^{-z(\vartheta_{t}\omega_k)}  \|\mathscr{Y}^k\|^{\frac12}_{\H} \|\mathscr{Y}^k\|^{\frac12}_{\V}
			\nonumber\\&\leq C N  \|\mathscr{Y}^k\|^{\frac12}_{\H} \|\mathscr{Y}^k\|^{\frac32}_{\V} +C N \left|1-e^{z(\vartheta_{t}\omega_0)-z(\vartheta_{t}\omega_k)}\right| \|\v^0\|_{\V}\|\mathscr{Y}^k\|_{\H}^{\frac12}\|\mathscr{Y}^k\|_{\V}^{\frac12}
			\nonumber\\&\leq C N  \|\mathscr{Y}^k\|^{\frac12}_{\H} \|\mathscr{Y}^k\|^{\frac32}_{\V} +C N \left|1-e^{z(\vartheta_{t}\omega_0)-z(\vartheta_{t}\omega_k)}\right| \|\v^0\|_{\V} \|\mathscr{Y}^k\|_{\V}
			\nonumber\\ & \leq  \frac{\nu}{10}\|\mathscr{Y}^k\|^2_{\V} +C N^4 \|\mathscr{Y}^k\|_{\H}^{2}+ C N^2 \left|1-e^{z(\vartheta_{t}\omega_0)-z(\vartheta_{t}\omega_k)}\right|^2 \|\v^0\|_{\V}^2. \label{LC9}
		\end{align}
		Combining \eqref{LC2}-\eqref{LC9}, we arrive at
		\begin{align}\label{LC13}
			\frac{\d }{\d t}\|\mathscr{Y}^k(t)\|^2_{\H}+\frac{\nu}{2}\|\mathscr{Y}^k(t)\|^2_{\V}\leq
			[C N^4+\sigma |z(\vartheta_t\omega_k)|] \|\mathscr{Y}^k(t)\|^2_{\H}+Q_1(t),
		\end{align}
		for a.e. $t\in[\tau,\tau+T]$, $T>0$, where
		\begin{align*}
			Q_1&=C\left|e^{-z(\vartheta_{t}\omega_k)}-e^{-z(\vartheta_{t}\omega_0)}\right|^2\|\f\|^2_{\L^2(\mathcal{O})}+C\left|z(\vartheta_t\omega_k)-z(\vartheta_t\omega_0)\right|^2\|\v^0\|^2_{\H}\nonumber\\&\quad+ C N^2 \left|1-e^{z(\vartheta_{t}\omega_0)-z(\vartheta_{t}\omega_k)}\right|^2 \|\v^0\|_{\V}^2.
		\end{align*}
		Using the fact that $\f\in\mathrm{L}^2_{\text{loc}}(\R;\L^2(\mathcal{O}))$, $\v^0\in\mathrm{C}([\tau,+\infty);\H)\cap\mathrm{L}^2_{\mathrm{loc}}(\tau,+\infty;\V)$ and Lemma \ref{conv_z}, we deduce that
		\begin{align}\label{LC16}
			\lim_{k\to+\infty}\int_{\tau}^{\tau+T}Q_1(r)\d r=0.
		\end{align}
		Applying Gronwall's inequality in \eqref{LC13} and using \eqref{conv_z2}, we obtain  for all $t\in[\tau,\tau+T]$,
		\begin{align}\label{LC17}
			\|\mathscr{Y}^k(t)\|^2_{\H}\leq e^{C N^4 T + \sigma C(\tau,T,\omega_0)}\left[\int_{\tau}^{\tau+T}Q_1(r)\d r\right].
		\end{align}
		In view of \eqref{LC16} and \eqref{LC17}, we have for all $t\in[\tau,\tau+T]$,
		\begin{align*}
			\|\mathscr{Y}^k(t)\|^2_{\H}\to0 \ \text{ as } \ k \to +\infty,
		\end{align*}
		which completes the proof.
	\end{proof}

	In view of Lemma \ref{Soln}, we can define a mapping $\Psi:\R^+\times\R\times\Omega\times\H\to\H$ by
	\begin{align}\label{Phi}
		\Psi(t,\tau,\omega,\u_{\tau})=\u(t+\tau,\tau,\vartheta_{-\tau}\omega,\u_{\tau})=e^{z(\vartheta_{t}\omega)}\v(t+\tau,\tau,\vartheta_{-\tau}\omega,\v_{\tau}).
	\end{align}
	The Lusin continuity in Proposition \ref{LusinC} provides the $\mathscr{F}$-measurability of $\Psi$. Consequently, the mapping $\Psi$ defined by \eqref{Phi} is a NRDS on $\H$. The following proposition demonstrates the backward convergence of NRDS \eqref{Phi}.
	\begin{proposition}\label{Back_conver}
		Suppose that Assumptions \ref{assumpO} and \ref{Hypo_f-N} are satisfied. Then, the solution $\v$ of the system \eqref{CNSE-M} backward converges to the solution $\widetilde{\v}$ of the system \eqref{A-CNSE-M}, that is,
		\begin{equation}
			\lim_{\tau\to -\infty}\|\v(T+\tau,\tau,\vartheta_{-\tau}\omega,\v_{\tau})-\widetilde{\v}(t,\omega,\widetilde{\v}_0)\|_{\H}=0, \ \ \text{ for all } \ T>0 \ \text{ and } \ \omega\in\Omega,
		\end{equation}
		whenever $\|\v_{\tau}-\widetilde{\v}_0\|_{\H}\to0$ as $\tau\to-\infty.$
	\end{proposition}
	
	\begin{proof}
		Let $\mathscr{Y}^{\tau}(t):=\v(t+\tau,\tau,\vartheta_{-\tau}\omega,\v_{\tau})-\widetilde{\v}(t,\omega,\widetilde{\v}_0)$ for $t\geq0$. From \eqref{CNSE-M} and \eqref{A-CNSE-M}, we obtain
		\begin{align}\label{BC1}
			\frac{\d\mathscr{Y}^{\tau}}{\d t} &=-\nu \A\mathscr{Y}^{\tau}-e^{z(\vartheta_{t}\omega)}\left[F_{N}(e^{z(\vartheta_{t}\omega)}\|\v\|_{\V}) \cdot \B\big(\v\big)- F_{N}(e^{z(\vartheta_{t}\omega)}\|\wi \v\|_{\V}) \cdot \B\big(\widetilde{\v}\big)\right] \nonumber\\ & \quad +e^{-z(\vartheta_{t}\omega)}\left[\mathcal{P}\f(t+\tau)-\mathcal{P}\f_{\infty}\right] + \sigma z(\vartheta_t\omega)\mathscr{Y}^{\tau}
			\nonumber\\ &=-\nu \A\mathscr{Y}^{\tau} - e^{z(\vartheta_{t}\omega)} F_{N}(e^{z(\vartheta_{t}\omega)}\|\v\|_{\V})\cdot \left[ \B\big(\wi\v,\mathscr{Y}^{\tau}\big)  + \B\big(\mathscr{Y}^{\tau}, \v\big)\right] 
			\nonumber\\ & \quad - e^{z(\vartheta_{t}\omega)} \left[ F_{N}(e^{z(\vartheta_{t}\omega)}\| \v\|_{\V}) - F_{N}(e^{z(\vartheta_{t}\omega)}\|\wi \v\|_{\V})  \right]\cdot \left[ \B\big(\widetilde{\v},\v\big)-\B\big(\widetilde{\v},\mathscr{Y}^{\tau}\big)\right] \nonumber\\ & \quad +e^{-z(\vartheta_{t}\omega)}\left[\mathcal{P}\f(t+\tau)-\mathcal{P}\f_{\infty}\right] + \sigma z(\vartheta_t\omega)\mathscr{Y}^{\tau},
		\end{align}
		in $\V^{\ast}$ (in the weak sense).  Taking the inner product with $\mathscr{Y}^{\tau}(\cdot)$ in \eqref{BC1}, and using \eqref{b0}, we get
		\begin{align}\label{BC2}
			\frac{1}{2}\frac{\d }{\d t}\|\mathscr{Y}^{\tau}\|^2_{\H}+\nu\|\mathscr{Y}^{\tau}\|^2_{\V}&=\sigma z(\vartheta_t\omega)\|\mathscr{Y}^{\tau}\|^2_{\H} - e^{z(\vartheta_{t}\omega)} F_{N}(e^{z(\vartheta_{t}\omega)}\|\v\|_{\V})\cdot  b\big(\mathscr{Y}^{\tau}, \v, \mathscr{Y}^{\tau}\big) 
			\nonumber\\ & \quad - e^{z(\vartheta_{t}\omega)} \left[ F_{N}(e^{z(\vartheta_{t}\omega)}\| \v\|_{\V}) - F_{N}(e^{z(\vartheta_{t}\omega)}\|\wi \v\|_{\V})  \right] \cdot b\big(\widetilde{\v},\v, \mathscr{Y}^{\tau}\big) \nonumber\\&\quad+e^{-z(\vartheta_{t}\omega)}(\f(t+\tau)-\f_{\infty},\mathscr{Y}^{\tau}).
		\end{align}
		In view of \eqref{HI+SE}, \eqref{FN1}-\eqref{FN2} and Young's inequalities, we find
		\begin{align}
			& \left|e^{z(\vartheta_{t}\omega)} F_{N}(e^{z(\vartheta_{t}\omega)}\|\v\|_{\V}) \cdot b\big(\mathscr{Y}^{\tau}, \v, \mathscr{Y}^{\tau}\big) \right|
			\nonumber\\	& \leq C e^{z(\vartheta_{t}\omega)} F_{N}(e^{z(\vartheta_{t}\omega)}\|\v\|_{\V}) \cdot \|\mathscr{Y}^{\tau}\|_{\V} \|\v\|_{\V} \|\mathscr{Y}^{\tau}\|_{\H}^{\frac12}  \|\mathscr{Y}^{\tau}\|_{\V}^{\frac12},
			\nonumber\\ & \leq C N   \|\mathscr{Y}^{\tau}\|_{\H}^{\frac12}  \|\mathscr{Y}^{\tau}\|_{\V}^{\frac32}
			\leq \frac{\nu}{4}\|\mathscr{Y}^{\tau}\|^2_{\V} + C N^4  \|\mathscr{Y}^{\tau}\|_{\H}^{2}, \label{BC5}\
		\end{align}
		and
		\begin{align}
			&\left|e^{z(\vartheta_{t}\omega)} \left[ F_{N}(e^{z(\vartheta_{t}\omega)}\| \v\|_{\V}) - F_{N}(e^{z(\vartheta_{t}\omega)}\|\wi \v\|_{\V})  \right] \cdot b\big(\widetilde{\v},\v, \mathscr{Y}^{\tau}\big)\right|
			\nonumber\\ & \leq e^{z(\vartheta_{t}\omega)} \cdot \frac{C}{N}F_{N}(e^{z(\vartheta_{t}\omega)}\| \v\|_{\V})  F_{N}(e^{z(\vartheta_{t}\omega)}\|\wi \v\|_{\V})\cdot e^{z(\vartheta_{t}\omega)}\|\mathscr{Y}^{\tau}\|_{\V} \|\widetilde{\v}\|_{\V}\|\v\|_{\V}  \|\mathscr{Y}^{\tau}\|_{\H}^{\frac{1}{2}} \|\mathscr{Y}^{\tau}\|_{\V}^{\frac{1}{2}}
			\nonumber\\ & \leq CN   \|\mathscr{Y}^{\tau}\|_{\H}^{\frac{1}{2}} \|\mathscr{Y}^{\tau}\|_{\V}^{\frac{3}{2}}
			\leq \frac{\nu}{4}\|\mathscr{Y}^{\tau}\|^2_{\V} + C N^4  \|\mathscr{Y}^{\tau}\|_{\H}^{2}. \label{BC51}
		\end{align}
		Applying H\"older's and Young's inequalities, we obtain
		\begin{align}
			\left|e^{-z(\vartheta_{t}\omega)}(\f(t+\tau)-\f_{\infty},\mathscr{Y}^{\tau})\right|& \leq\|\f(t+\tau)-\f_{\infty}\|^2_{\L^2(\mathcal{O})}+Ce^{-2z(\vartheta_{t}\omega)}\|\mathscr{Y}^{\tau}\|^2_{\H}.\label{BC4}
		\end{align}
		Combining \eqref{BC2}-\eqref{BC5}, we arrive at
		\begin{align}\label{BC6}
			\frac{\d }{\d t}\|\mathscr{Y}^{\tau}\|^2_{\H}+\frac{\nu}{2}\|\mathscr{Y}^{\tau}\|^2_{\V}\leq C\big[
			S_1(t)\|\mathscr{Y}^{\tau}\|^2_{\H}+\|\f(t+\tau)-\f_{\infty}\|^2_{\L^2(\mathcal{O})}\big],
		\end{align}
		where $S_1= e^{-2z(\vartheta_{t}\omega)}+\left|z(\vartheta_{t}\omega)\right|.$  Applying Gronwall's inequality to \eqref{BC6} over $(0,T)$, we reach at
		\begin{align}\label{BC7}
			\|\mathscr{Y}^{\tau}(T)\|^2_{\H}\leq \left[\|\mathscr{Y}^{\tau}(0)\|^2_{\H}+C\int_{0}^{T}\|\f(t+\tau)-\f_{\infty}\|^2_{\L^2(\mathcal{O})} \d t\right]e^{C\int_{0}^{T}S_1(t)\d t}.
		\end{align}
		The continuity of $z$  implies that
		\begin{align}\label{BC8}
			\int_{0}^{T}S_{1}(t)\d t<+\infty.
		\end{align}
		Using the fact that $\int_{0}^{T}S_{1}(t)\d t$ is bounded, \eqref{BC8-A} and $\|\mathscr{Y}^{\tau}(0)\|^2_{\H}=\|\v_{\tau}-\widetilde{\v}_0\|_{\H}\to0$ as $\tau\to-\infty$, and it completes the proof.
	\end{proof}

	The next lemma is needed to obtain the increasing random absorbing set and the inequality \eqref{AB3} (see below), which  is  used to prove the backward uniform tail-estimates (Lemma \ref{largeradius}) and the backward flattening estimates (Lemma \ref{Flattening}).
	\begin{lemma}\label{Absorbing}
		Suppose that $\f\in\mathrm{L}^2_{\mathrm{loc}}(\R;\L^2(\mathcal{O}))$ and Assumption \ref{assumpO} is satisfied. Then, for each $(\tau,\omega,D)\in\R\times\Omega\times\mathfrak{D},$ there exists a time $\mathcal{T}:=\mathcal{T}(\tau,\omega,D)>0$ such that
		\begin{align}\label{AB1}
			&\sup_{s\leq \tau}\sup_{t\geq \mathcal{T}}\sup_{\v_{0}\in D(s-t,\vartheta_{-t}\omega)}\bigg[\|\v(s,s-t,\vartheta_{-s}\omega,\v_{0})\|^2_{\H}\nonumber\\&\quad+\frac{\nu}{2}\int_{s-t}^{s}e^{\nu\lambda(\zeta-s)-2\sigma\int^{\zeta}_{s}z(\vartheta_{\upeta-s}\omega)\d\upeta}\|\v(\zeta,s-t,\vartheta_{-s}\omega,\v_{0})\|^2_{\V}\d\zeta\bigg]\nonumber\\&\leq 1+\frac{2}{\nu\lambda}\sup_{s\leq \tau}K(s,\omega),
		\end{align}
		where $K(s,\omega)$ is given by
		\begin{align}\label{AB2}
			K(s,\omega):=\int_{-\infty}^{0}e^{\nu\lambda\zeta+2|z(\vartheta_{\zeta}\omega)|+2\sigma\int_{\zeta}^{0}z(\vartheta_{\upeta}\omega)\d\upeta}\|\f(\zeta+s)\|^2_{\L^2(\mathcal{O})}\d\zeta.
		\end{align}
		Furthermore, for all $\xi> s-t,$ $ t\geq0$ and $\v_{0}\in\H$,
		\begin{align}\label{AB3}
			&\|\v(\xi,s-t,\vartheta_{-s}\omega,\v_{0})\|^2_{\H}+\frac{\nu}{2}\int_{s-t}^{\xi}e^{\nu\lambda(\zeta-\xi)-2\sigma\int^{\zeta}_{\xi}z(\vartheta_{\upeta-s}\omega)\d\upeta}\|\v(\zeta,s-t,\vartheta_{-s}\omega,\v_{0})\|^2_{\V}\d\zeta\nonumber\\&\leq e^{-\nu\lambda(\xi-s+t)+2\sigma\int_{-t}^{\xi-s}z(\vartheta_{\upeta}\omega)\d\upeta}\|\v_{0}\|^2_{\H} \nonumber\\&\quad+ \frac{2}{\nu\lambda}\int\limits_{-t}^{\xi-s}e^{\nu\lambda(\zeta+s-\xi)+2|z(\vartheta_{\zeta}\omega)|+2\sigma\int_{\zeta}^{\xi-s}z(\vartheta_{\upeta}\omega)\d\upeta}\|\f(\zeta+s)\|^2_{\L^2(\mathcal{O})}\d\zeta.
		\end{align}
	\end{lemma}

	\begin{proof}
		Let us rewrite the energy inequality \eqref{EI1} for $\v(\zeta)=\v(\zeta,s-t,\vartheta_{-s}\omega,\v_{0})$, that is,
		\begin{align}\label{AB0}
			\frac{\d}{\d\zeta} \|\v(\zeta)\|^2_{\H}&+ \left(\nu\lambda-2\sigma z(\vartheta_{\zeta-s}\omega)\right)\|\v(\zeta)\|^2_{\H}+\frac{\nu}{2}\|\v(\zeta)\|^2_{\V}\leq \frac{2e^{2|z(\vartheta_{\zeta-s}\omega)|}}{\nu\lambda}\|\f(\zeta)\|^2_{\L^2(\mathcal{O})},
		\end{align}
		for a.e. $\zeta\geq s-t$. In view of variation of constants formula with respect to $\zeta\in(s-t,\xi)$, we get \eqref{AB3}. Putting $\xi=s$ in \eqref{AB3}, we find
		\begin{align}\label{AB4}
			& \|\v(s,s-t,\vartheta_{-s}\omega,\v_{0})\|^2_{\H}+\frac{\nu}{2}\int_{s-t}^{s}e^{\nu\lambda(\uprho-s)-2\sigma\int^{\uprho}_{s}z(\vartheta_{\upeta-s}\omega)\d\upeta}\|\v(\uprho,s-t,\vartheta_{-s}\omega,\v_{0})\|^2_{\V}\d\uprho\nonumber\\&\leq e^{-\nu\lambda t+2\sigma\int_{-t}^{0}z(\vartheta_{\upeta}\omega)\d\upeta}\|\v_{0}\|^2_{\H}  +\frac{2}{\nu\lambda}\int_{-t}^{0}e^{\nu\lambda\uprho+2|z(\vartheta_{\uprho}\omega)|+2\sigma\int_{\uprho}^{0}z(\vartheta_{\upeta}\omega)\d\upeta}\|\f(\uprho+s)\|^2_{\L^2(\mathcal{O})}\d\uprho\nonumber\\&\leq e^{-\nu\lambda t+2\sigma\int_{-t}^{0}z(\vartheta_{\upeta}\omega)\d\upeta}\|\v_{0}\|^2_{\H}  +\frac{2}{\nu\lambda}\int_{-\infty}^{0}e^{\nu\lambda\uprho+2|z(\vartheta_{\uprho}\omega)|+2\sigma\int_{\uprho}^{0}z(\vartheta_{\upeta}\omega)\d\upeta}\|\f(\uprho+s)\|^2_{\L^2(\mathcal{O})}\d\uprho,
		\end{align}
		for all $s\leq\tau$. Since $\v_0\in D(s-t,\vartheta_{-t}\omega)$ and $D$ is backward tempered, it implies from \eqref{Z3} and the definition of backward temperedness \eqref{BackTem} that there exists a time $\mathcal{T}=\mathcal{T}(\tau,\omega,D)$ such that for all $t\geq \mathcal{T}$,
		\begin{align}\label{v_0}
			e^{-\nu\lambda t+2\sigma\int_{-t}^{0}z(\vartheta_{\upeta}\omega)\d\upeta}\sup_{s\leq \tau}\|\v_{0}\|^2_{\H}\leq e^{-\frac{\nu\lambda}{3}t}\sup_{s\leq \tau}\|D(s-t,\vartheta_{-t}\omega)\|^2_{\H}\leq1.
		\end{align}
		Hence, by taking supremum on $s\in(-\infty,\tau]$ in \eqref{AB4},  we achieve the required estimate \eqref{AB1}.
	\end{proof}

	\begin{proposition}\label{IRAS}
		Suppose that $\f\in\mathrm{L}^2_{\mathrm{loc}}(\R;\L^2(\mathcal{O}))$, and Assumptions \ref{assumpO} and \ref{Hypo_f-N} are satisfied. For $K(\tau,\omega),$ the same as in \eqref{AB2}, we have
		\vskip 2mm
		\noindent
		\emph{(i)} There is an increasing $\mathfrak{D}$-pullback absorbing set $\mathcal{K}$ given by
		\begin{align}\label{IRAS1}
			\mathcal{K}(\tau,\omega):=\left\{\u\in\H:\|\u\|^2_{\H}\leq e^{z(\omega)}\left[1+\frac{2}{\nu\lambda}\sup_{s\leq \tau}K(s,\omega)\right]\right\}, \ \text{ for all } \ \tau\in\R.
		\end{align}
		Moreover, $\mathcal{K}$ is backward-uniformly tempered with an arbitrary rate, that is, $\mathcal{K}\in{\mathfrak{D}}$.
		\vskip 2mm
		\noindent
		\emph{(ii)} There is a $\mathfrak{B}$-pullback \textbf{random} absorbing set $\widetilde{\mathcal{K}}$ given by
		\begin{align}\label{IRAS11}
			\widetilde{\mathcal{K}}(\tau,\omega):=\left\{\u\in\H:\|\u\|^2_{\H}\leq e^{z(\omega)}\left[1+\frac{2}{\nu\lambda}K(\tau,\omega)\right]\right\}\in\mathfrak{B}, \ \text{ for all } \ \tau\in\R.
		\end{align}
	\end{proposition}
	\begin{proof}
		(i) Using \eqref{G3}, \eqref{Z5} and \eqref{Z3}, we obtain
		\begin{align}\label{IRAS2}
			\sup_{s\leq \tau}K(s,\omega)&= \sup_{s\leq \tau}\int_{-\infty}^{0}e^{\nu\lambda\zeta+2|z(\vartheta_{\zeta}\omega)|+2\sigma\int_{\zeta}^{0}z(\vartheta_{\upeta}\omega)\d\upeta}\|\f(\zeta+s)\|^2_{\L^2(\mathcal{O})}\d\zeta<+\infty.
		\end{align}
		Hence, the pullback absorption follows from Lemma \ref{Absorbing}. Due to the fact that $\tau\mapsto \sup\limits_{s\leq\tau}K(s,\omega)$ is an increasing function, $\mathcal{K}(\tau,\omega)$ is an increasing $\mathfrak{D}$-pullback absorbing set. Using similar arguments as in \eqref{IRAS3-N}, with the help of \eqref{G3}, \eqref{Z5} and \eqref{Z3}, we deduce
		\begin{align}
			&\lim_{t\to+\infty}e^{-ct}\sup_{s\leq \tau}\|\mathcal{K}(s-t,\vartheta_{-t}\omega)\|^2_{\H}=0,
		\end{align}
		which gives $\mathcal{K}\in{\mathfrak{D}}$.
		\vskip 2mm
		\noindent
		(ii) Since $\widetilde{\mathcal{K}}\subseteq\mathcal{K}\in\mathfrak{D}\subseteq\mathfrak{B}$ and the mapping $\omega\mapsto K(\tau,\omega)$ is $\mathscr{F}$-measurable,  $\widetilde{\mathcal{K}}$ is a $\mathfrak{B}$-pullback \textbf{random} absorbing set.
	\end{proof}

	\subsection{Backward uniform tail-estimates and backward flattening estimates}
	In this subsection, we prove the backward uniform tail-estimates and backward flattening estimates for the solution of \eqref{2-M}.  The following lemma provides the backward uniform tail-estimates for the solution of the system \eqref{2-M}.
	\begin{lemma}\label{largeradius}
		Suppose that Assumptions \ref{assumpO} and \ref{Hypo_f-N} are satisfied. Then, for any $(\tau,\omega,D)\in\R\times\Omega\times\mathfrak{D},$ the solution of the system \eqref{2-M}  satisfies
		\begin{align}\label{ep}
			&\lim_{k,t\to+\infty}\sup_{s\leq \tau}\sup_{\v_{0}\in D(s-t,\vartheta_{-t}\omega)}\|\v(s,s-t,\vartheta_{-s}\omega,\v_{0})\|^2_{\mathbb{L}^2(\mathcal{O}^{c}_{k})}=0,
		\end{align}
		where $\mathcal{O}_{k}=\{x\in\mathcal{O}:|x|\leq k\},$ $k\in\mathbb{N}$. In addition, for any $(\tau,\omega,B)\in\R\times\Omega\times{\mathfrak{B}},$ the solution of \eqref{CNSE-A}  satisfies
		\begin{align}\label{ep-tau}
			&\lim_{k,t\to+\infty}\sup_{\v_{0}\in B(\tau-t,\vartheta_{-t}\omega)}\|\v(\tau,\tau-t,\vartheta_{-\tau}\omega,\v_{0})\|^2_{\L^2(\mathcal{O}^{c}_{k})}=0.
		\end{align}
	\end{lemma}

	\begin{proof}
		Let $\uprho$ be the  smooth function defined in \eqref{337}. Similar to \eqref{p-value} and \eqref{p-value-N}, we obtain from \eqref{2-M} that
		\begin{align}\label{p-value-M}
			p=(-\Delta)^{-1}\left[F_{N}(e^{z(\vartheta_{t}\omega)}\|\v\|_{\V})\cdot e^{2z(\vartheta_{t}\omega)}\sum_{i,j=1}^{3}\frac{\partial^2}{\partial x_i\partial x_j}\big(y_i y_j\big)-\nabla\cdot\f\right],
		\end{align}
		in the weak sense and
		\begin{align}
			\|p\|_{\mathrm{L}^2(\mathcal{O})}\leq C N e^{z(\vartheta_{t}\omega)} \|\v\|_{\V} +C\|\f\|_{\L^2(\mathcal{O})}\label{p-value-N-M}.
		\end{align}
		Taking the inner product to the first equation of \eqref{2-M} with $\uprho^2\left(\frac{|x|^2}{k^2}\right)\v$, we have
		\begin{align}\label{ep1}
			&\frac{1}{2} \frac{\d}{\d t} \int_{\mathcal{O}}\uprho^2\left(\frac{|x|^2}{k^2}\right)|\v|^2\d x\nonumber\\&= \underbrace{\nu \int_{\mathcal{O}}(\Delta\v) \uprho^2\left(\frac{|x|^2}{k^2}\right) \v \d x}_{E_1(k,t)}-\underbrace{F_{N}(e^{z(\vartheta_{t}\omega)}\|\v\|_{\V})\cdot e^{z(\vartheta_{t}\omega)}b\left(\v,\v,\uprho^2\left(\frac{|x|^2}{k^2}\right)\v\right)}_{E_2(k,t)}\nonumber\\&\quad-\underbrace{e^{-z(\vartheta_{t}\omega)}\int_{\mathcal{O}}(\nabla p)\uprho^2\left(\frac{|x|^2}{k^2}\right)\v\d x}_{E_3(k,t)}+ \underbrace{e^{-z(\vartheta_{t}\omega)} \int_{\mathcal{O}}\f\uprho^2\left(\frac{|x|^2}{k^2}\right)\v\d x}_{E_4(k,t)} +\sigma z(\vartheta_{t}\omega)\int_{\mathcal{O}}\uprho^2\left(\frac{|x|^2}{k^2}\right)|\v|^2\d x.
		\end{align}
		Let us now estimate each term on the right hand side of \eqref{ep1}. By integration by parts, divergence free condition on $\v$,  \eqref{poin}, \eqref{p-value-N-M}, H\"older's and Ladyzhenskaya's inequalities, \eqref{FN1} and Young's inequality, we obtain as in \eqref{ep2-N}-\eqref{ep4-N}
		\begin{align}
			\left|E_1(k,t)\right|&\leq-\frac{3\nu\lambda}{4} \int_{\mathcal{O}}\left|\left(\uprho\left(\frac{|x|^2}{k^2}\right) \v\right)\right|^2\d x+\frac{C}{k}\|\v\|^2_{\V},
			\\
			\left|E_2(k,t)\right|&\leq\frac{CN}{k}\|\v\|^2_{\V},
			\\
			\left|E_3(k,t)\right|&\leq \frac{C(1+N)}{k}\bigg[\|\v\|^2_{\V}+e^{2\left|z(\vartheta_{t}\omega)\right|}\|\f\|^2_{\L^2(\mathcal{O})}\bigg],\\
			\left|E_4(k,t)\right|&\leq \frac{\nu\lambda}{4} \int_{\mathcal{O}}\uprho\left(\frac{|x|^2}{k^2}\right)|\v|^2\d x +\frac{e^{2\left|z(\vartheta_{t}\omega)\right|}}{\nu\lambda} \int_{\mathcal{O}}\uprho\left(\frac{|x|^2}{k^2}\right)|\f(x)|^2\d x.\label{ep4}
		\end{align}
		Combining \eqref{ep1}-\eqref{ep4}, we get
		\begin{align}\label{ep5}
			&\frac{\d}{\d t} \|\v\|^2_{\mathbb{L}^2(\mathcal{O}_k^{c})}+ \left(\nu\lambda-2\sigma z(\vartheta_{t}\omega)\right) \|\v\|^2_{\mathbb{L}^2(\mathcal{O}_k^c)}  
			\nonumber\\ &\leq\frac{C(1+N)}{k} \left[ \|\v\|^2_{\V}+e^{2\left|z(\vartheta_{t}\omega)\right|}\|\f\|^2_{\L^2(\mathcal{O})}\right] +\frac{2e^{2\left|z(\vartheta_{t}\omega)\right|}}{\nu\lambda} \int_{\mathcal{O}\cap\{|x|\geq k\}}|\f(x)|^2\d x.
		\end{align}
		Applying the variation of constants formula to the equation \eqref{ep5} on $(s-t,s)$ and replacing $\omega$ by $\vartheta_{-s}\omega$, we find,  for $s\leq\tau, t\geq 0$ and $\omega\in\Omega$,
		\begin{align}\label{ep6}
			&\|\v(s,s-t,\vartheta_{-s}\omega,\v_{0})\|^2_{\mathbb{L}^2(\mathcal{O}_k^{c})} \nonumber\\& \leq e^{-\alpha t+2\sigma\int_{-t}^{0}z(\vartheta_{\upeta}\omega)\d\upeta}\|\v_{0}\|^2_{\H} +\frac{C(1+N)}{k}\bigg[\int_{s-t}^{s}e^{\nu\lambda(\zeta-s)-2\sigma\int^{\zeta}_{s}z(\vartheta_{\upeta-s}\omega)\d\upeta}\|\v(\zeta,s-t,\vartheta_{-s}\omega,\v_{0})\|^2_{\V}\d\zeta
			\nonumber\\&\quad+\int_{-t}^{0}e^{2\left|z(\vartheta_{\zeta-s}\omega)\right|+\nu\lambda(\zeta-s)-2\sigma\int^{\zeta}_{s}z(\vartheta_{\upeta-s}\omega)\d\upeta}\|\f(\zeta+s)\|^2_{\L^2(\mathcal{O})}\d\zeta\bigg]\nonumber\\&\quad+C\int_{s-t}^{s}e^{2|z(\vartheta_{\zeta-s}\omega)|+\nu\lambda(\zeta-s)-2\sigma\int^{\zeta}_{s}z(\vartheta_{\upeta-s}\omega)\d\upeta} \int_{\mathcal{O}\cap\{|x|\geq k\}}|\f(x,\xi)|^2\d x\d\zeta.
		\end{align}
		Now using \eqref{f3-N}, \eqref{Z3}, the definition collections $\mathfrak{D}$ and $\mathfrak{B}$ (see Subsection \ref{CoRS}), and Lemma \ref{Absorbing}, one can immediately complete the proof.
	\end{proof}

	The following lemma provides the backward flattening estimates of solutions to the system \eqref{2-M}.
	\begin{lemma}\label{Flattening}
		Suppose that Assumptions \ref{assumpO} and \ref{Hypo_f-N} are satisfied. Let $(\tau,\omega,D)\in\R\times\Omega\times\mathfrak{D}$, $k\geq1$ be fixed, $\varrho_k$ be given by \eqref{varrho_k} and $\mathrm{P}_i$ be the same as in \eqref{DirectProd}. Then
		\begin{align}\label{FL-P-M}
			\lim_{i,t\to+\infty}\sup_{s\leq \tau}\sup_{\v_{0}\in D(s-t,\vartheta_{-t}\omega)}\|(\I-\P_{i})\bar{\v}(s,s-t,\vartheta_{-s}\omega,\bar{\v}_{0,2})\|^2_{\L^2(\mathcal{O}_{\sqrt{2}k})}=0,
		\end{align}
		where $\bar{\v}=\varrho_k\v$ and $\bar{\v}_{0,2}=(\I-\P_{i})(\varrho_k\v_{0})$. In addition, for any $(\tau,\omega,B)\in\R\times\Omega\times\mathfrak{B}$, we have
        \begin{align}\label{FL-P-M-tau}
			\lim_{i,t\to+\infty}\sup_{\v_{0}\in B(\tau-t,\vartheta_{-t}\omega)}\|(\I-\P_{i})\bar{\v}(\tau,\tau-t,\vartheta_{-\tau}\omega,\bar{\v}_{0,2})\|^2_{\L^2(\mathcal{O}_{\sqrt{2}k})}=0.
		\end{align}
	\end{lemma}
	\begin{proof}
		Multiplying  \eqref{2-M} by $\varrho_k$, we obtain 
		\begin{align}\label{FL1-M}
			&\frac{\d\bar{\v}}{\d t}-\nu\Delta\bar{\v}+F_{N}(e^{z(\vartheta_{t}\omega)}\|\v\|_{\V})\cdot
			e^{z(\vartheta_{t}\omega)}\varrho_k\left[(\v\cdot\nabla)\v\right]+e^{-z(\vartheta_{t}\omega)}\varrho_k\nabla p\nonumber\\&=e^{-z(\vartheta_{t}\omega)}\varrho_k\f +\sigma z(\vartheta_t\omega)\bar{\v}-\nu\v\Delta\varrho_k-2\nu\nabla\varrho_k\cdot\nabla\v.
		\end{align}
		As before, we get
		\begin{align}\label{FL2-M}
			&\frac{1}{2}\frac{\d}{\d t}\|\bar{\v}_{i,2}\|^2_{\L^2(\mathcal{O}_{\sqrt{2}k})} +\nu\|\nabla\bar{\v}_{i,2}\|^2_{\L^2(\mathcal{O}_{\sqrt{2}k})}-\sigma z(\vartheta_{t}\omega)\|\bar{\v}_{i,2}\|^2_{\L^2(\mathcal{O}_{\sqrt{2}k})}\nonumber\\&=-\underbrace{F_{N}(e^{z(\vartheta_{t}\omega)}\|\v\|_{\V})\cdot e^{z(\vartheta_{t}\omega)}\sum_{q,m=1}^{3}\int_{\mathcal{O}_{\sqrt{2}k}}\left(\I-\P_i\right)\bigg[y_{q}\frac{\partial y_{m}}{\partial x_q}\left\{\varrho_k(x)\right\}^2y_{m}\bigg]\d x}_{:=\widehat{J}_1}
			\nonumber\\&\quad-\underbrace{
				\left\{\nu\big(\v\Delta\varrho_k,
				\bar{\v}_{i,2}\big)
				+2\nu\big(\nabla\varrho_k
				\cdot\nabla\v,\bar{\v}_{i,2}\big)
				-\big(e^{-z(\vartheta_{t}\omega)}
				\varrho_k\f,\bar{\v}_{i,2}\big)\right\}
			}_{:=\widehat{J}_2}   -
			\underbrace{\big(e^{-z(\vartheta_{t}\omega)}\varrho_k(x)\nabla p, \bar{\v}_{i,2}\big)}_{:=\widehat{J}_3}.
		\end{align}
		Next, we estimate the terms on the right hand side of \eqref{FL2-M} as follows: Using integration by parts, divergence free condition on $\v$, \eqref{poin-i}, \eqref{FN1}, H\"older's, Ladyzhenskaya's and Young's inequalities, we arrive at
		\begin{align}
			|\widehat{J}_1|&\leq \frac{\nu}{6}\|\nabla\bar{\v}_{i,2}\|^2_{\L^2(\mathcal{O}_{\sqrt{2}k})}+CN^2\lambda^{-1}_{i+1} \|\v\|^2_{\V},\label{FL3-M}\\
			|\widehat{J}_2|&\leq \frac{\nu}{6}\|\nabla\bar{\v}_{i,2}\|^2_{\L^2(\mathcal{O}_{\sqrt{2}k})}+ C\lambda^{-1}_{i+1}\bigg[\|\v\|^2_{\V}+e^{2|z(\vartheta_{t}\omega)|}\|\f\|^2_{\L^2(\mathcal{O})}\bigg],\label{FL4-M}\\
			|\widehat{J}_3|&\leq \frac{\nu}{6}\|\nabla\bar{\v}_{i,2}
			\|^2_{\L^2(\mathcal{O}_{\sqrt{2}k})} +
			C\lambda^{-1}_{i+1}(1+N^2)\left[\|\v\|^2_{\V} + e^{2\left|z(\vartheta_{t}\omega)\right|}\|\f\|^2_{\L^2(\mathcal{O})}\right]
			,\label{FL5-M}
		\end{align}
		where we have used \eqref{p-value-N-M} in \eqref{FL5-M}. Now, combining \eqref{FL2-M}-\eqref{FL5-M} and using \eqref{poin} in the resulting inequality, we arrive at
		\begin{align}\label{FL6-M}
			\frac{\d}{\d t}\|\bar{\v}_{i,2}\|^2_{\L^2(\mathcal{O}_{\sqrt{2}k})} +\left(\nu\lambda-2\sigma z(\vartheta_{t}\omega)\right)\|\bar{\v}_{i,2}\|^2_{\L^2(\mathcal{O}_{\sqrt{2}k})} \leq C\lambda^{-1}_{i+1}(1+N^2)\left[\|\v\|^2_{\V} + e^{2\left|z(\vartheta_{t}\omega)\right|}\|\f\|^2_{\L^2(\mathcal{O})}\right].
		\end{align} 
		In view of the variation of constants formula, by \eqref{FL7}, we find
		\begin{align}\label{FL8-M}
			&\|(\I-\P_{i})\bar{\v}(s,s-t,\vartheta_{-s}\omega,\bar{\v}_{0,2})\|^2_{\L^2(\mathcal{O}_{\sqrt{2}k})}\nonumber\\&\leq e^{-\nu\lambda t+4{\aleph}\int^{0}_{-t}\left|z(\vartheta_{\upeta}\omega)\right|\d\upeta}\|(\I-\P_i)(\varrho_k\v_{0})\|^2_{\L^2(\mathcal{O}_{\sqrt{2}k})}\nonumber\\&\quad+C\lambda^{-1}_{i+1}(1+N^2)\underbrace{\int_{s-t}^{s}e^{\nu\lambda(\zeta-s)-2\sigma\int^{\zeta}_{s}z(\vartheta_{\upeta-s}\omega)\d\upeta}\|\v(\zeta,s-t,\vartheta_{-s}\omega,\v_{0})\|^2_{\V}\d\zeta}_{\widehat{L}_1(s,t)}\nonumber\\&\quad+C\lambda^{-1}_{i+1}(1+N^2)\underbrace{\int_{-t}^{0}e^{2\left|z(\vartheta_{\zeta-s}\omega)\right|+\nu\lambda(\zeta-s)-2\sigma\int^{\zeta}_{s}z(\vartheta_{\upeta-s}\omega)\d\upeta}\|\f(\zeta+s)\|^2_{\L^2(\mathcal{O})}\d\zeta}_{\widehat{L}_2(s,t)}.
		\end{align}
		It implies from \eqref{Z3}, \eqref{G3} and \eqref{AB1} that
		\begin{align}\label{FL9-M}
			\sup_{s\leq \tau}\widehat{L}_1(s,t)<+\infty \ \ \text{ and } \ \ \sup_{s\leq \tau}\widehat{L}_2(s,t)<+\infty,
		\end{align}
		for sufficiently large $t>0$.  Furthermore, we have
		\begin{align}\label{FL11-M}
			\|(\I-\P_i)(\varrho_k\v_{0})\|^2_{\L^2(\mathcal{O}_{\sqrt{2}k})}\leq C\|\v_{0}\|^2_{\H},
		\end{align}
		for all $\v_{0}\in D(s-t,\vartheta_{-t}\omega)$ and $s\leq\tau$. Now, using the definition of collections $\mathfrak{D}$ and $\mathfrak{B}$ (see Subsection \ref{CoRS}), \eqref{Z3}, \eqref{G3}, Lemma \ref{Absorbing}, \eqref{FL9-M}-\eqref{FL11-M} and the fact that $\lambda_{i+1}^{-1}\to 0$ as $i\to+\infty$ in \eqref{FL8-M}, we obtain \eqref{FL-P-M} and \eqref{FL-P-M-tau}, as required.
	\end{proof}

	\subsection{Proof of Theorem \ref{MT1}}\label{thm1.5}
	This subsection is devoted to the proof of the main results of this section, that is, the existence of $\mathfrak{D}$-pullback random attractors and their asymptotic autonomous robustness for the system \eqref{SNSE} with $S(\u)=\u$. 

The proof of this theorem is  divided into the following seven steps:
	\vskip 2mm
	\noindent
	\textbf{Step I:} \textit{$\mathfrak{D}$-pullback time-semi-uniform asymptotic compactness of $\Psi$.} It is enough to prove that for each $(\tau,\omega,D)\in\R\times\Omega\times\mathfrak{D}$, arbitrary sequences $s_n\leq\tau$, $\tau_{n}\to+\infty$ and $\v_{0,n}\in D(s_n-t_n,\vartheta_{-t_n}\omega)$,  the sequence $$\v_n:=\v(s_n,s_n-t_n,\vartheta_{-s_n}\omega,\v_{0,n})$$ is pre-compact in $\H$. Let $E_{N}=\{\v_n:n\geq N\}, \ N=1,2,\ldots.$ In order to prove the pre-compactness of the sequence $\v_n$, it is enough to show that the Kuratowski measure $\kappa_{\H}(E_{N})\to0$ and $N\to+\infty$, (see Lemma \ref{K-BAC}).
	
	Since, $t_n\to\infty$, there exists $\widehat{\mathcal{N}}\in\N$ such that $t_n\geq \mathcal{T}$ for all $n\geq \widehat{\mathcal{N}}$. For each $\eta>0$, by Lemma \ref{largeradius}, there exists $\widehat{\mathcal{N}}_1\geq \widehat{\mathcal{N}}$ and $\widehat{K}_0\geq1$ such that
	\begin{align}\label{MTM1}
		\left\|\uprho\left(\frac{|x|^2}{k^2}\right)\v_n\right\|_{\L^2(\mathcal{O}^{c}_{K})}\leq\frac{\eta}{2}, 
	\end{align}
	for all $n\geq \widehat{\mathcal{N}}_1$ and for all  $K\geq \widehat{K}_0$, where $\mathcal{O}^{c}_{K}=\mathcal{O}\backslash\mathcal{O}_{K}$ and  $\mathcal{O}_{K}=\{x\in\mathcal{O}:|x|\leq K\}.$  By Lemma \ref{Flattening}, there exist $\hat{i}_0\in\N$ and $\widehat{\mathcal{N}}_2\geq\widehat{\mathcal{N}}_1$ such that
	\begin{align}\label{MTM2}
		\|(\I-\P_i)(\varrho_{K}\v_n)\|_{\L^2(\mathcal{O}_{\sqrt{2}K})}\leq\frac{\eta}{2}, 
	\end{align}
	for all $n\geq \widehat{\mathcal{N}}_2$ and for all  $i\geq \hat{i}_0$.
	
	Now, Lemma \ref{Absorbing} provides us that the set $E_{\widehat{\mathcal{N}}_2}$ is bounded in $\H$. Then, the set $\{\varrho_{K}\v_n:n\geq\widehat{\mathcal{N}}_2\}$ is bounded in $\L^2(\mathcal{O}_{\sqrt{2}K})$. Hence, by the finite-dimensional range of $\P_i$, $\P_{i}\{\varrho_{K}\v_n:n\geq\widehat{\mathcal{N}}_2\}$ is pre-compact in $\L^2(\mathcal{O}_{\sqrt{2}K})$, from which we conclude that
	\begin{align}\label{MTM3}
		\kappa_{\L^2(\mathcal{O}_{\sqrt{2}K})}\left(\P_{i}\{\varrho_{K}\v_n:n\geq\widehat{\mathcal{N}}_2\}\right)=0.
	\end{align}
	It follows from \eqref{MTM2}-\eqref{MTM3} and \cite[Theorem 1.4]{Rakocevic} that
	\begin{align}\label{MTM4}
		& \kappa_{\L^2(\mathcal{O}_{\sqrt{2}K})}\left(\{\varrho_{K}\v_n:n\geq\widehat{\mathcal{N}}_2\}\right)\nonumber\\&\leq\kappa_{\L^2(\mathcal{O}_{\sqrt{2}K})}\left(\P_{i}\{\varrho_{K}\v_n:n\geq\widehat{\mathcal{N}}_2\}\right)+\kappa_{\L^2(\mathcal{O}_{\sqrt{2}K})}\left((\I-\P_{i})\{\varrho_{K}\v_n:n\geq\widehat{\mathcal{N}}_2\}\right)\nonumber\\& \leq\frac{\eta}{2}.
	\end{align}
	We infer from \eqref{MTM1} and \eqref{MTM4} that
	\begin{align*}
		\kappa_{\H}(E_{\widehat{\mathcal{N}}_2})\leq\kappa_{\L^2(\mathcal{O}_{\sqrt{2}K})}(\varrho_{K}E_{\widehat{\mathcal{N}}_2})+\kappa_{\L^2(\mathcal{O}^{c}_{K})}\left(\uprho\left(\frac{|x|^2}{k^2}\right)E_{\widehat{\mathcal{N}}_2}\right)\leq \eta,
	\end{align*}
	which shows that $\Psi$ is time-semi-uniformly asymptotically compact in $\H$.
	\vskip 2mm
	\noindent
	\textbf{Step II:} \textit{$\mathfrak{B}$-pullback asymptotically compactness of $\Psi$.} 
    {Using similar arguments as in \textbf{Step I}, we will prove that for each $(\tau,\omega,B)\in\R\times\Omega\times\mathfrak{B}$, arbitrary sequence $t_{n}\to+\infty$ and $\v_{0,n}\in B(\tau-t_n,\vartheta_{-t_n}\omega)$,  the sequence 
	$$\v^n:=\v(\tau,\tau-t_n,\vartheta_{-\tau}\omega,\v_{0,n})$$ 
	is pre-compact in $\H$. We also denote $E^{N}=\{\v^n:n\geq N\}, \ N=1,2,\ldots.$ 
	In view of Lemma \ref{largeradius}, for each $\eta>0$, there exists $\widehat{\mathcal{N}}_3\in\N$ and $\widehat{K}_1\geq 1$ such that
	\begin{align*}
		\left\|\uprho\left(\frac{|x|^2}{k^2}\right)\v^n\right\|_{\L^2(\mathcal{O}^{c}_{K})}\leq\frac{\eta}{2}, 
	\end{align*}
	for all $n\geq \widehat{\mathcal{N}}_3$ and for all  $K\geq \widehat{K}_1$. 
	 In view of Lemma \ref{Flattening}, there exist $\hat{i}_1\in\N$ and $\widehat{\mathcal{N}}_4\geq\widehat{\mathcal{N}}_3$ such that
	\begin{align*}
		\|(\I-\P_i)(\varrho_{K}\v^n)\|_{\L^2(\mathcal{O}_{\sqrt{2}K})}\leq\frac{\eta}{2}, 
	\end{align*}
	for all $n\geq \widehat{\mathcal{N}}_4$ and for all  $i\geq \hat{i}_1$. Since the set $E^{\widehat{\mathcal{N}}_4}$ is bounded in $\H$. Then, the set $\{\varrho_{K}\v^n:n\geq\widehat{\mathcal{N}}_4\}$ is bounded in $\L^2(\mathcal{O}_{\sqrt{2}K})$. Hence, $\P_{i}\{\varrho_{K}\v^n:n\geq\widehat{\mathcal{N}}_4\}$ is pre-compact in $\L^2(\mathcal{O}_{\sqrt{2}K})$, from which we deduce that
	\begin{align*}
		\kappa_{\L^2(\mathcal{O}_{\sqrt{2}K})}\left(\P_{i}\{\varrho_{K}\v^n:n\geq\widehat{\mathcal{N}}_4\}\right)=0.
	\end{align*}
	Therefore,
	\begin{align*}
		&\kappa_{\L^2(\mathcal{O}_{\sqrt{2}K})}\left(\{\varrho_{K}\v^n:n\geq\widehat{\mathcal{N}}_4\}\right)\nonumber\\&\leq\kappa_{\L^2(\mathcal{O}_{\sqrt{2}K})}\left(\P_{i}\{\varrho_{K}\v^n:n\geq\widehat{\mathcal{N}}_4\}\right)+\kappa_{\L^2(\mathcal{O}_{\sqrt{2}K})}\left((\I-\P_{i})\{\varrho_{K}\v^n:n\geq\widehat{\mathcal{N}}_4\}\right)\nonumber\\& \leq\frac{\eta}{2}.
	\end{align*}
	From the above estimates, we obtain
	\begin{align*}
		\kappa_{\H}(E^{\widehat{\mathcal{N}}_4})\leq\kappa_{\L^2(\mathcal{O}_{\sqrt{2}K})}(\varrho_{K}E^{\widehat{\mathcal{N}}_4})+\kappa_{\L^2(\mathcal{O}^{c}_{K})}\left(\uprho\left(\frac{|x|^2}{k^2}\right)E^{\widehat{\mathcal{N}}_4}\right)\leq \eta,
	\end{align*}
	which shows that $\Psi$ is $\mathfrak{B}$-pullback asymptotically compact in $\H$.}
	\vskip 2mm
	\noindent
	\textbf{Step III:} \textit{$\mathfrak{D}$-pullback attractor $\mathscr{A}(\tau,\omega)$.} Proposition \ref{IRAS}  and \textbf{Step I} ensure that $\Psi$ has a $\mathfrak{D}$-pullback absorbing set and $\Psi$ is $\mathfrak{D}$-pullback asymptotically compact in $\H$, respectively. Hence, by the abstract theory established in \cite{SandN_Wang}, $\Psi$ has a unique $\mathfrak{D}$-pullback attractor $\mathscr{A}$ which is given by
	\begin{align}\label{A3}
		\mathscr{A}=\bigcap\limits_{t_0>0}\overline{\bigcup\limits_{t\geq t_0}\Psi(t,\tau-t,\vartheta_{-t}\omega)\mathcal{K}(\tau-t,\vartheta_{-t}\omega)}^{\H}.
	\end{align}
	However, we remark that the $\mathscr{F}$-measurability of $\mathscr{A}$ is unknown, therefore we say that $\mathscr{A}$ is a $\mathfrak{D}$-pullback attractor instead of $\mathfrak{D}$-pullback random attractor.
	\vskip 2mm
	\noindent
	\textbf{Step IV:} \textit{$\mathfrak{B}$-pullback attractor $\widetilde{\mathscr{A}}(\tau,\omega)$.} Proposition \ref{IRAS} (part (ii)) and \textbf{Step II} ensure that $\Psi$ has a $\mathfrak{B}$-pullback random absorbing set and $\Psi$ is $\mathfrak{B}$-pullback asymptotically compact, respectively. Hence, by the abstract theory established in \cite{SandN_Wang}, $\Psi$ has a unique $\mathfrak{B}$-pullback random attractor $\widetilde{\mathscr{A}}$ which is given by
	\begin{align}\label{A4}
	\widetilde{\mathscr{A}}=\bigcap\limits_{t_0>0}\overline{\bigcup\limits_{t\geq t_0}\Psi(t,\tau-t,\vartheta_{-t}\omega)\widetilde{\mathcal{K}}(\tau-t,\vartheta_{-t}\omega)}^{\H}.
	\end{align}
	\vskip 2mm
	\noindent
	\textbf{Step V:} \textit{Time-semi-uniformly compactness of $\mathscr{A}(\tau,\omega)$.} We prove that $\bigcup\limits_{s\leq\tau}\mathscr{A}(s,\omega)$ is pre-compact in $\H$. Let $\{\u_n\}_{n=1}^{\infty}$ be an arbitrary sequence extracted from $\bigcup\limits_{s\leq\tau}\mathscr{A}(s,\omega)$. Then, we can find a sequence $s_n\leq\tau$ such that $\u_n\in\mathscr{A}(s_n,\omega)$ for each $n\in\N$. Now, for the sequence $t_n\to\infty$, by the invariance property of $\mathscr{A},$ we have $\u_n\in\Psi(t_n,s_n-t_n,\vartheta_{-t_{n}}\omega)\mathscr{A}(s_n-t_n,\vartheta_{-t_{n}}\omega)$. It implies that we can find $\u_{0,n}\in\mathscr{A}(s_n-t_n,\vartheta_{-t_{n}}\omega)$ such that $\u_n=\Psi(t_n,s_n-t_n,\vartheta_{-t_{n}}\omega,\u_{0,n})$, where $\u_{0,n}\in\mathscr{A}(s_n-t_n,\vartheta_{-t_{n}}\omega)\subseteq\mathcal{K}(s_n-t_n,\vartheta_{-t_{n}}\omega)$ with $s_n\leq\tau$ and $\mathcal{K}\in\mathfrak{D}$.  Then,  it follows from  the $\mathfrak{D}$-pullback time-semi-uniform asymptotic compactness of $\Psi$ that the sequence $\{\u_n\}_{n=1}^{\infty}$ is pre-compact in $\H$. Hence, $\bigcup\limits_{s\leq\tau}\mathscr{A}(s,\omega)$ is pre-compact in $\H$.
	\vskip 2mm
	\noindent
	\textbf{Step VI:} \textit{$\mathscr{A}(\tau,\omega)=\widetilde{\mathscr{A}}(\tau,\omega)$. This  implies that $\Psi$ has a unique pullback \textbf{random} attractor which is time-semi-uniformly compact in $\H$.} Let us fix $(\tau,\omega)\in\R\times\Omega$. Since, by Proposition \ref{IRAS}, $\mathcal{K}(\tau,\omega)\supseteq\widetilde{\mathcal{K}}(\tau,\omega)$, it follows from \eqref{A3} and \eqref{A4} that $\mathscr{A}(\tau,\omega)\supseteq\widetilde{\mathscr{A}}(\tau,\omega)$. At the same time, since $\mathscr{A}\in\mathfrak{B}\subseteq\mathfrak{D}$, the invariance property of $\mathscr{A}$ and the attraction property of $\widetilde{\mathscr{A}}$ imply that
	\begin{align*}
		\text{dist}_{\H}(\mathscr{A}(\tau,\omega),\widetilde{\mathscr{A}}(\tau,\omega))=\text{dist}_{\H}(\Psi(t,\tau-t,\vartheta_{-t}\omega)\mathscr{A}(\tau-t,\vartheta_{-t}\omega),\widetilde{\mathscr{A}}(\tau,\omega))\to 0,
	\end{align*}
	as $t\to+\infty$. This indicates that $\mathscr{A}(\tau,\omega)\subseteq\widetilde{\mathscr{A}}(\tau,\omega)$. Hence $\mathscr{A}(\tau,\omega)=\widetilde{\mathscr{A}}(\tau,\omega)$, which, in view of the $\mathscr{F}$-measurability of $\widetilde{\mathscr{A}}(\tau,\omega)$, shows that $\mathscr{A}(\tau,\omega)$ is $\mathscr{F}$-measurable.
	\vskip 2mm
	\noindent
	\textbf{Step VII:} \textit{Proof of \eqref{MT2-N} and \eqref{MT3-N}.} By using {Propositions} \ref{IRAS} and \ref{Back_conver} and \emph{time-semi-uniformly compactness} of $\mathscr{A}$, and applying similar arguments as in the proof of \cite[Theorem 5.2]{CGTW}, one can complete the proof. Since, the arguments are similar to the proof of \cite[Theorem 5.2]{CGTW}, we are not repeating here.

	\begin{appendix}
		\renewcommand{\thesection}{\Alph{section}}
		\numberwithin{equation}{section}
		
		\section{GMNSE with a different cut-off} \label{ApA}\setcounter{equation}{0}
		In this section, we discuss another modification of the 3D Navier-Stokes equations given in \eqref{3D-NSE}. In particular, we replace the nonlinear term $(\u\cdot\nabla)\u$ with $F_{N}(\|\u\|_{\L^4(\mathcal{O})})\left[(\u\cdot\nabla)\u\right]$ in the system \eqref{3D-NSE} and obtain the following 3D globally modified Navier-Stokes equations:
		\begin{equation}\label{A.1}
			\left\{
			\begin{aligned}
				\frac{\partial \u}{\partial t}-\nu \Delta\u+F_{N}(\|\u\|_{\L^4(\mathcal{O})})\left[(\u\cdot\nabla)\u\right]+\nabla p&=\boldsymbol{f}, &&\text{ in }\  \mathcal{O}\times(\tau,\infty), \\ \nabla\cdot\u&=0, && \text{ in } \ \ \mathcal{O}\times[\tau,\infty), \\ \u&=0, && \text{ on } \ \ \partial\mathcal{O}\times[\tau,\infty), \\
				\u(\tau)&=\u_0, && \ x\in \mathcal{O},
			\end{aligned}
			\right.
		\end{equation}
		where $F_N(\cdot)$ is the same as defined in Section \ref{sec1} (see \eqref{FN} above). Here, the modifying factor $F_{N}(\|\u\|_{\L^4(\mathcal{O})})$ depends on the norm $\|\u\|_{\L^4(\mathcal{O})}$. Essentially, it prevents large $\L^4(\mathcal{O})$-norm, instead of the $\V$-norm in contrast to the model given in \eqref{2},  dominating the dynamics and leading to explosions. It appears to be a better modification of  3D Navier-Stokes equations in comparison to the system \eqref{2} since  $\|\u\|_{\L^4(\mathcal{O})}\leq C \|\u\|_{\V}$, by an application of Sobolev's inequality.
		
		We notice that the analogous calculations as performed in \cite[Lemma 2.1]{Romito_2009} provide the following result:
		\begin{lemma}
			For any $\u,\boldsymbol{v}\in\L^4(\mathcal{O})$ and each $N>0$, 
			\begin{align}
				0 \leq \|\u\|_{\L^4(\mathcal{O})} F_{N}(\|\u\|_{\L^4(\mathcal{O})}) & \leq N, \label{FN3}\\
				|F_{N}(\|\u\|_{\L^4(\mathcal{O})})-F_{N}(\|\boldsymbol{v}\|_{\L^4(\mathcal{O})})|&\leq \frac{1}{N} F_{N}(\|\u\|_{\L^4(\mathcal{O})})F_{N}(\|\boldsymbol{v}\|_{\L^4(\mathcal{O})})\cdot \|\u-\boldsymbol{v}\|_{\L^4(\mathcal{O})}.\label{FN4} 
			\end{align}
		\end{lemma}
		
		Taking the projection $\mathcal{P}$ on the 3D SGMNSE \eqref{A.1}, one obtains the following abstract form in $\V^*$: 
		\begin{equation}\label{A.2}
			\left\{
			\begin{aligned}
				\frac{\d\u}{\d t}+\nu \A\u+F_{N}(\|\u\|_{\L^4(\mathcal{O})})\cdot\B(\u)&=\mathcal{P}\f, \\
				\u(\tau)&=\u_{0}.
			\end{aligned}
			\right.
		\end{equation}

		Let us now discuss the solvability result of the system \eqref{A.2} in the next theorem.	
		\begin{theorem}\label{thm-weak}
			Let $\u_{0}\in\H$ and $\f\in\mathrm{L}^2_{\mathrm{loc}}(\R;\L^2(\mathcal{O}))$. Then, there exists a unique weak solution $\u\in\mathrm{C}([\tau,+\infty);\H)\cap\mathrm{L}^2_{\mathrm{loc}}(\tau,+\infty;\V)$ to the system \eqref{A.2}. Moreover, the solution satisfies the following estimate:
			\begin{align}\label{EI}
				\|\u(t)\|_{\H}^2+\frac{\nu}{2} \int_{\tau}^{t}e^{-\nu\lambda (t-s)}\|\u(s)\|_{\V}^2\d s \leq  e^{-\nu\lambda (t-\tau)}\|\u_0\|_{\H}^2   +  \frac{2}{\nu} \int_{\tau}^{t}e^{-\nu\lambda (t-s)}\|\f(s)\|_{\H^{-1}(\mathcal{O})}^2\d s,
			\end{align}
		for all $t\geq \tau$.
		\end{theorem}
		\begin{proof}
		 The standard Faedo-Galerkin techniques used in Lemma \ref{Soln} can be used to determine the existence of a weak solution $\u\in\mathrm{C}([\tau,+\infty);\H)\cap\mathrm{L}^2_{\mathrm{loc}}(\tau,+\infty;\V)$ to the system \eqref{A.2}. Here, we  show that there exists at most one weak solution to the system \eqref{A.2}  satisfying the energy estimate \eqref{EI}.
			 
			Taking the inner product of the first equation of the system \eqref{A.2} with $\u$, and using the Cauchy-Schwarz and Young's inequalities, we obtain
			\begin{align}
				\frac{1}{2}\frac{\d }{\d t}\|\u(t)\|^2_{\H} + \nu \|\u(t)\|^2_{\V} = \langle \f(t), \u(t) \rangle \leq \frac{1}{\nu}\|\f(t)\|_{\H^{-1}(\mathcal{O})}^2 + \frac{\nu}{4} \|\u(t)\|^2_{\V},
			\end{align}
		for a.e. $t\geq \tau$. The above estimate implies 
		\begin{align}
			\frac{\d }{\d t}\|\u(t)\|^2_{\H} + \nu \lambda \|\u(t)\|^2_{\H} + \frac{\nu}{2} \|\u(t)\|^2_{\V} \leq \frac{2}{\nu}\|\f(t)\|_{\H^{-1}(\mathcal{O})}^2,
		\end{align}
	for a.e. $t\geq \tau$, where we have used \eqref{poin}. Hence, an application of the variation of constants formula provides energy estimate \eqref{EI}.
			
			Let $\u$ and $\boldsymbol{v}$ be two solutions to the system \eqref{A.2} with the same initial data. Let us define $\w=\u-\boldsymbol{v}$.   Then $\w$  satisfies the following system:
			\begin{equation}\label{A.3}
				\left\{
				\begin{aligned}
					\frac{\d\w}{\d t}+\nu \A\w+F_{N}(\|\u\|_{\L^4(\mathcal{O})})\cdot\B(\u)-F_{N}(\|\boldsymbol{v}\|_{\L^4(\mathcal{O})})\cdot\B(\boldsymbol{v})&=\boldsymbol{0}, \\
					\w(\tau)&=\boldsymbol{0}.
				\end{aligned}
				\right.
			\end{equation}
		The existence of strong solutions (see Theorem \ref{thm-strong} below) ensures that every weak solution satisfies the energy equality.	Taking the inner product of the first equation of the system \eqref{A.3} with $\w$, and then using \eqref{b0}, \eqref{FN3} and \eqref{FN4}, we obtain
			\begin{align*}
				&\quad\frac12 \frac{\d}{\d t}\|\w\|_{\H}^2+\nu \|\w\|_{\V}^2 
				\nonumber\\&= -F_{N}(\|\u\|_{\L^4(\mathcal{O})})\cdot b(\u,\u,\w)+ F_{N}(\|\boldsymbol{v}\|_{\L^4(\mathcal{O})})\cdot b(\boldsymbol{v},\boldsymbol{v},\w)
				\nonumber\\ & = -F_{N}(\|\u\|_{\L^4(\mathcal{O})}) \cdot  b(\w,\w,\u)
				- \big[ F_{N}(\|\u\|_{\L^4(\mathcal{O})}) - F_{N}(\|\boldsymbol{v}\|_{\L^4(\mathcal{O})})\big]\cdot b( \boldsymbol{v}, \u,\w )
				\nonumber\\ & \leq F_{N}(\|\u\|_{\L^4(\mathcal{O})}) \cdot \|\w\|_{\L^4(\mathcal{O})}\|\u\|_{\L^4(\mathcal{O})}\|\w\|_{\V} 
				\nonumber\\ & \quad + \frac{1}{N}F_{N}(\|\u\|_{\L^4(\mathcal{O})}) F_{N}(\|\boldsymbol{v}\|_{\L^4(\mathcal{O})})\cdot \|\w\|_{\L^4(\mathcal{O})} \|\u\|_{\L^4(\mathcal{O})}\|\boldsymbol{v}\|_{\L^4(\mathcal{O})} \|\w\|_{\V}
				\nonumber\\ & \leq C N \|\w\|_{\H}^{\frac14}\|\w\|_{\V}^{\frac74} \leq \frac{\nu}{2} \|\w\|_{\V}^{2} + C N^8 \|\w\|_{\H}^{2},
			\end{align*}
			 Now, an application of Gronwall's inequality gives
			\begin{align*}
				\|\w(t)\|_{\H}^2\leq \|\w(\tau)\|_{\H}^2e^{CN^8T}=0,
			\end{align*}
			which proves the uniqueness, that is, $\u(t)=\boldsymbol{v}(t)$ for all $t\geq\tau$ in $\H$.
		\end{proof}
		
		Next we discuss the existence of strong solutions to the system \eqref{A.2}.
		\begin{theorem}\label{thm-strong}
			If $\u_{0}\in\V$ and $\f\in\mathrm{L}^2_{\mathrm{loc}}(\R;\L^2(\mathcal{O}))$, then every weak solution $\u(\cdot)$ of the system \eqref{A.2} belongs to $\mathrm{C}([\tau,+\infty);\V)\cap\mathrm{L}^2_{\mathrm{loc}}(\tau,+\infty;\D(\A))$.
		\end{theorem}
		\begin{proof}
			Here, we mainly provide the energy estimates which is different from the works \cite{Caraballo+Real+Kloeden_2006,Caraballo+Real+Kloeden_2010}, and the rest of the proof can be completed  analogous to the works \cite{Caraballo+Real+Kloeden_2006,Caraballo+Real+Kloeden_2010}. Let us take the inner product of the first equation of the system \eqref{A.3} with $\A\u$ to obtain
			\begin{align}\label{A.4}
				\frac12\frac{\d}{\d t}\|\u\|_{\V}^2+\nu \|\A\u\|_{\H}^2=(\f,\A\u)+ F_{N}(\|\u\|_{\L^4(\mathcal{O})})\cdot b(\u,\u,\A\u).
			\end{align}
			An application of H\"older's and Young's inequalities imply
			\begin{align}\label{A.5}
				|(\f,\A\u)|\leq \frac{1}{\nu}\|\f\|_{\L^2(\mathcal{O})} + \frac{\nu}{4}\|\A\u\|_{\H}^2.
			\end{align}
			Using H\"older's inequality, \eqref{FN3}, and Gagliardo-Nirenberg's, Sobolev's and Young's inequalities, we obtain 
			\begin{align}\label{A.6}
				|F_{N}(\|\u\|_{\L^4(\mathcal{O})})\cdot b(\u,\u,\A\u)|
				& \leq F_{N}(\|\u\|_{\L^4(\mathcal{O})})\cdot \|\u\|_{\L^4(\mathcal{O})} \|\nabla\u\|_{\L^4(\mathcal{O})}\|\A\u\|_{\H}
				\nonumber\\ & \leq C N \|\A\u\|^{\frac34}_{\H}\|\u\|_{\mathbb{L}^6(\mathcal{O})}^{\frac14}\|\A\u\|_{\H}
				\nonumber\\ & \leq C N \|\A\u\|^{\frac{7}{4}}_{\H}\|\u\|_{\V}^{\frac14}
				\nonumber\\ & \leq \frac{\nu}{4} \|\A\u\|^{2}_{\H} +C N^{8} \|\u\|_{\V}^{2}.
			\end{align}
			Combining \eqref{A.4}-\eqref{A.6} and applying Gronwall's inequality, we obtain
			\begin{align*}
				\|\u(t)\|_{\V}^2 +\frac{\nu}{2}\int_{\tau}^{t}\|\A\u(s)\|_{\H}^2\d s \leq \|\u_0\|_{\V}^2 e^{CN^{8}},
			\end{align*}
			for all $t\geq \tau$. The above estimate guarantees that the strong solution $\u\in\mathrm{C}([\tau,+\infty);\V)\cap\mathrm{L}^2_{\mathrm{loc}}(\tau,+\infty;\D(\A))$. 
		\end{proof}

		\begin{remark}
			The results of this work and other relevant results known in the literature for the system  \eqref{2} (see for instance \cite{Anh+Thanh+Tuyet_2023,Caraballo+Chen+Yang_2023_SAM,Caraballo+Chen+Yang_2023_AMOP,Deugoue+Medjo_2018,Hang+My+Nguyen_2024, Hang+My+Nguyen_2024a, Hang+Nguyen_2024, Kloeden+Langa+Real_2007,Ren_2014,Zhao+Yang_2017}, etc.)  can be established for the system \eqref{A.1} also with a suitable modification in the estimates (for instance, see \eqref{EI}). We also point out that the model \eqref{A.1} has an important feature with respect to stochastic perturbation which has been discussed in the work \cite{Kinra+Mohan_UP}.
		\end{remark}
	\end{appendix}

		\medskip\noindent
		\textbf{Acknowledgments:}  
		The first author would like to thank Hanoi Pedagogical University 2 for providing a fruitful working environment. This research was funded by Vietnam Ministry of Education and Training under grant number B2024-CTT-06. Corresponding author would like to thank the Department of Atomic Energy, Government of India, for financial assistance and the Tata Institute of Fundamental Research - Centre for Applicable Mathematics (TIFR-CAM) for providing a stimulating scientific environment and resources. K. Kinra is also funded by national funds through the FCT - Fundação para a Ciência e a Tecnologia, I.P., under the scope of the projects UID/297/2025 and UID/PRR/297/2025 (Center for Mathematics and Applications - NOVA Math). M. T. Mohan would like to thank the Department of Science and Technology (DST) Science \& Engineering Research Board (SERB), India for a MATRICS grant (MTR/2021/000066).

		\medskip\noindent
		\textbf{Data availability:} No data was used for the research described in the article.

		\medskip\noindent
		\textbf{Conflict of interest:} The authors declare no conflict of interest.


\begin{thebibliography}{99}
		
		\bibitem{Anh+Thanh+Tuyet_2023}  C.T. Anh, N.V. Thanh and  P.T.  Tuyet, Asymptotic behaviour of solutions to stochastic three-dimensional globally modified Navier-Stokes equations, \emph{Stochastics},  \textbf{95} (6)  (2023), 997--1021.
		
		\bibitem{Arnold}	L. Arnold, \emph{Random Dynamical Systems}, Springer-Verlag, Berlin, Heidelberg, New York, 1998.
		
		
		
		\bibitem{Ball1997JNS} J.M. Ball, Continuity properties and global attractors of generalized semiflows and the
		Navier-Stokes equations, \emph{J. Nonl. Sci.} \textbf{7} (1997) 475-502.
		
		
		\bibitem{BCF}  	Z. Brze\'zniak, M. Capi\'nski and F. Flandoli, Pathwise global attractors for stationary random dynamical systems, \emph{Probab. Theory Related Fields}, \textbf{95} (1993), 87-102.

        
		\bibitem{BCLLLR} Z. Brz\'ezniak, T. Caraballo, J.A. Langa, Y. Li, G. Lukaszewicz and J. Real, Random attractors for stochastic 2D Navier-Stokes equations in some unbounded domains, \emph{J. Differential Equations}, \textbf{255} (2013), 3897--3919.

        
		\bibitem{BL} Z. Brz\'ezniak and Y. Li, Asymptotic compactness and absorbing sets for 2D stochastic Navier-Stokes equations in some unbounded domains, \emph{Trans. Amer. Math. Soc.}, \textbf{358} (12) (2006) 5587--5629.

        
		
		
		
		
		
		
		\bibitem{Caraballo+Chen+Yang_2023_SAM}  T. Caraballo, Z.  Chen and  D. Yang,  Random dynamics and limiting behaviors for 3D globally modified Navier-Stokes equations driven by colored noise, \emph{Stud. Appl. Math.},  \textbf{151} (1)  (2023),  247--284.
		
		
		\bibitem{Caraballo+Chen+Yang_2023_AMOP}  T. Caraballo, Z.  Chen and  D. Yang, Stochastic 3D globally modified Navier-Stokes equations: weak attractors, invariant measures and large deviations,
		\emph{Appl. Math. Optim.},  \textbf{88} (3)  (2023),  Paper No. 74, 46 pp.
		
		
		
		
		
		
		\bibitem{CGTW} T. Caraballo, B. Guo, N. Tuan and R. Wang, Asymptotically autonomous robustness of random attractors for a class of weakly dissipative stochastic wave equations on unbounded domains, \emph{Proc. Roy. Soc. Edinburgh Sect. A}, \textbf{151} (6) (2021), 1700-1730.

        
		\bibitem{Caraballo+Kloeden_2013}  T. Caraballo and P.E. Kloeden, The three-dimensional globally modified Navier-Stokes equations: recent developments,  \emph{Recent trends in dynamical systems}, 
		473--492, Springer Proc. Math. Stat., 35, Springer, Basel,  2013. 

        
		\bibitem{Caraballo+Real+Kloeden_2006} T. Caraballo, J.  Real and P.E. Kloeden,  Unique strong solutions and V-attractors of a three dimensional system of globally modified Navier-Stokes equations,
		\emph{Adv. Nonlinear Stud.},  \textbf{6} (3)  (2006),  411--436. 
		
		
		\bibitem{CLR} T. Caraballo, G. Lukaszewicz and J. Real, Pullback attractors for asymptotically compact non-autonomous dynamical systems, \emph{Nonlinear Anal.} \textbf{64} (3) (2006), 484-498.
		
		
		\bibitem{CLR1} T. Caraballo, G. Lukaszewicz and J. Real, Pullback attractors for non-autonomous 2D-Navier–Stokes equations in some unbounded domains, \emph{C. R. Acad. Sci. Paris, Ser. I}, \textbf{342} (4) (2006), 263-268.
		
		\bibitem{Caraballo+Real+Kloeden_2010} T. Caraballo, J.  Real and P.E. Kloeden,  Addendum to the paper ``Unique strong solutions and V-attractors of a three dimensional system of globally modified Navier-Stokes equations'', Advanced Nonlinear Studies 6 (2006), 411–436 [MR2245266],
		\emph{Adv. Nonlinear Stud.},  \textbf{10}(1)  (2010),  245--247.
        
		
		\bibitem{CLR2} A. Carvalho, J.A. Langa and J. Robinson, \emph{Attractors for Infinite-dimensional Non-autonomous Dynamical Systems}, Netherlands: Springer New York, 2013.
		
		
		
		\bibitem{chenp} P. Chen, B. Wang, R. Wang and X. Zhang, Multivalued random dynamics of Benjamin-Bona-Mahony equations driven by nonlinear colored noise on unbounded domains, \emph{Math. Ann.}, \textbf{386} (1--2) (2023), 343--373.
		
		
		
		
		
		\bibitem{CV2} V.V. Chepyzhov and M.I. Vishik, \emph{Attractors for Equations of Mathematical Physics}, American Mathematical Society Colloquium Publications, \textbf{49} American Mathematical Society, Providence, RI, 2002.
		
		
		
		
		
		
		
		
		
		\bibitem{Constantin_2003} P. Constantin, \emph{Near Identity Transformations for the Navier-Stokes Equations, in Handbook of Mathematical Fluid Dynamics}, Vol. I\!I, 117-141, North-Holland, Amsterdam, 2003.
		
		
		
		
		\bibitem{CDF}	H. Crauel, A. Debussche and F. Flandoli, Random attractors, \emph{J. Dynam. Differential Equations} \textbf{9} (2) (1995), 307-341.
		
		
		
		
		\bibitem{CF}	H. Crauel and F. Flandoli, Attractors for random dynamical systems, \emph{Probab. Theory Related Fields}, \textbf{100} (1994), 365-393.
		
		
		
		
		
		
		\bibitem{CLL} H. Cui, J.A. Langa and Y. Li, Measurability of random attractors for quasi strong-to-weak continuous random dynamical systems, \emph{J. Dynam. Differential Equations}, \textbf{30} (2018), 1873-1898.
		
		
		\bibitem{Deugoue+Medjo_2018}   G. Deugou\'e and  T.T. Medjo,  The stochastic 3D globally modified Navier-Stokes equations: existence, uniqueness and asymptotic behavior,	\emph{Commun. Pure Appl. Anal.},  \textbf{17} (6)  (2018), 2593--2621.
		
		
		
		\bibitem{Dia+Lappicy_2021} J.-Y. Dai and P. Lappicy, Ginzburg-Landau patterns in circular and spherical geometries: vortices, spirals, and attractors, \emph{SIAM J. Appl. Dyn. Syst.}, \textbf{20}(4) (2021), 1959--1984.
		
		
		
		
		\bibitem{FAN} X. Fan, Attractors for a damped stochastic wave equation of the sine-Gordon type with sublinear multiplicative noise, \emph{Stoch. Anal. Appl.}, \textbf{24} (2006), 767–793.



\bibitem{Farwig+Kozono+Sohr_2007}  R. Farwig, H. Kozono and H. Sohr, On the Helmholtz decomposition in general unbounded domains, \textit{Arch. Math. (Basel)}, {\bf 88}(3 (2007), 239--248.  

        
		
		
		
		\bibitem{FS} F. Flandoli and B. Schmalfuss, Random attractors for the 3D stochastic Navier-Stokes equation with multiplicative noise, \emph{Stoch. Stoch. Rep.}, \textbf{59} (1-2), 1996, 21–45.
		
		
		
		
		
		
		
		\bibitem{GLW} A. Gu, K. Lu and B. Wang, Asymptotic behavior of random Navier-Stokes equations driven by Wong-Zakai approximations, \emph{Discrete Contin. Dyn. Syst. Ser. B}, \textbf{39} (1) (2019), 185-218.
		
		
		
		
		
		
		\bibitem{Hang+My+Nguyen_2024}  H.T. Hang,  B.K. My and P.T. Nguyen,  Wong-Zakai approximations and attractors for stochastic three-dimensional globally modified Navier-Stokes equations driven by nonlinear noise, \emph{Discrete Contin. Dyn. Syst. Ser. B},  \textbf{29} (2)  (2024),  1069--1104.
			
     \bibitem{Hang+My+Nguyen_2024a} H.T. Hang,  B.K. My and P.T. Nguyen, Dynamics of solutions for the three-dimensional stochastic globally modified Navier-Stokes equations on unbounded domains, \emph {Bull. Korean Math. Soc.}, \textbf{61}(5) (2024), 1369--1393. 

\bibitem{Hang+Nguyen_2024}  H.T. Hang and P.T. Nguyen, Random attractors for three-dimensional stochastic globally modified Navier–Stokes equations driven by additive noise on unbounded domains, \emph{Random Oper. Stoch. Equ.}, \textbf{32}(3) (2024), 223--239.

		\bibitem{Heywood} J.G. Heywood, The Navier-Stokes Equations: on the existence, regularity and decay of solutions, \emph{Ind. Univ. Math. J.}, \textbf{29} (1980), 639-681.


\bibitem{KK+FC1} K. Kinra and F. Cipriano, Well-posedness and asymptotic analysis of a class of 2D and 3D third grade fluids in bounded and unbounded domains, \emph{J. Evol. Equ.}, \textbf{25} (2025), Article No. 71.


\bibitem{KK+FC2} K. Kinra  and F. Cipriano, Random dynamics and invariant measures for a class of non-Newtonian fluids of differential type on 2d and 3d Poincar\'e domains, \textit{Nonlinear Anal.}, \textbf{264}, (2026), 114005.

        

		\bibitem{KM2} K. Kinra and M.T. Mohan, $\mathbb{H}^1$-Random attractors for 2D stochastic convective Brinkman-Forchheimer equations in unbounded domains, \emph{Adv. Differential Equations}, \textbf{28}(9-10) (2023), 807--884.
		
		
		
		
		
		
		
		\bibitem{KM3} K. Kinra and M.T. Mohan, Large time behavior of the deterministic and stochastic 3D convective Brinkman-Forchheimer equations in periodic domains, \emph{J. Dynam. Differential Equations}, (2021), pp. 1-42.
		
		
		
		\bibitem{KM7} K. Kinra and M.T. Mohan, Long term behavior of 2D and 3D non-autonomous random convective Brinkman-Forchheimer equations driven by colored noise,  \emph{J. Dynam. Differential Equations}, \textbf{37}(2) (2025), 1467--1537.
		
		
		\bibitem{Kinra+Mohan_UP} K. Kinra and M.T. Mohan, Random dynamics of solutions for three-dimensional stochastic globally modified Navier-Stokes equations on unbounded Poincar\'e domains, \textit{Submitted}. \url{https://arxiv.org/pdf/2408.10426}
		
		\bibitem{Kinra+Mohan+Wang_2024} K. Kinra, M.T. Mohan and R. Wang, Asymptotically autonomous robustness of non-autonomous random attractors for stochastic convective Brinkman-Forchheimer equations on $\mathbb{R}^3$, \emph{Int. Math. Res. Not. IMRN}, \textbf{2024} (7) (2024), 5850--5893.
		
		\bibitem{KRM} K. Kinra, R. Wang and M.T. Mohan, Asymptotic autonomy of random attractors for non-autonomous stochastic Navier-Stokes equations on bounded domains, \emph{Evol. Equ. Control Theory}, \textbf{13} (2) (2024), 349-381.
		
		
		\bibitem{Kloeden+Langa_2007}  P.E. Kloeden and J.A. Langa, Flattening, squeezing and the existence of random attractors, \emph{Proc. R. Soc. Lond. Ser. A Math. Phys. Eng. Sci.}, \textbf{463} (2077) (2007), 163--181.
		
		
		\bibitem{Kloeden+Langa+Real_2007}  P.E. Kloeden, J.A. Langa and   J. Real, Pullback V-attractors of the 3-dimensional globally modified Navier-Stokes equations, \emph{Commun. Pure Appl. Anal.},  \textbf{6} (4)  (2007),  937--955.
		
		
		
		
		
		
		
		\bibitem{Kuratowski} K. Kuratowski, Sur les espaces complets, \emph{Fund. Math.}, \textbf{1} (15) (1930) 301–309.
		
		
		\bibitem{LGL} Y. Li, A. Gu and J. Li, Existence and continuity of bi-spatial random attractors and application to stochastic semilinear Laplacian equations, \emph{J. Dynam. Differential Equations}, \textbf{258} (2) (2015), 504-534.
		
		
        
		\bibitem{YR} Y. Li and R. Wang, Asymptotic autonomy of random attractors for BBM equations with Laplace-multiplier noise, \emph{J. Appl. Anal. Comput.}, \textbf{10} (4) (2020), 1199-1222.
		
		
		\bibitem{Li+Wang+Kloeden_2023}   Y. Li, F. Wang and P. Kloeden, Enlarged numerical attractors for lattice systems with porous media degeneracies, \emph{SIAM J. Appl. Dyn. Syst.}, \textbf{22}(3) (2023), 2282--2311.
		
		
		
		\bibitem{Lin+Guo+Yang+Miranville_2024} Y. Lin, C. Guo, X. Yang and A. Miranville,	Dynamics of the three-dimensional Brinkman-Forchheimer-extended Darcy model in the whole space, \emph{Discrete Contin. Dyn. Syst.}, \textbf{44} (5) (2024), 1304--1328.
		
		
		
		
		
		\bibitem{Rakocevic} V. Rako$\check{c}$evi\'c, Measures of noncompactness and some applications, \emph{Filomat} (1998), 87-120.
		
		
		\bibitem{Ren_2014} J. Ren, Large time behavior for weak solutions of the 3D globally modified Navier-Stokes equations, \emph{Abstr. Appl. Anal.},  2014, Art. ID 879780, 5 pp.
		
		\bibitem{Robinson2} J.C. Robinson, \emph{Infinite-Dimensional Dynamical Systems, An Introduction to Dissipative Parabolic PDEs and the Theory of Global Attractors}, Cambridge Texts in Applied Mathematics, 2001.
		
		
		\bibitem{Robinson1}J.C. Robinson, \emph{Dimensions, Embeddings and Attractors}, Cambridge Tracts in Mathematics, Cambridge University Press, Cambridge, 2011.
		

        
		
		
		\bibitem{Romito_2009} M. Romito, The uniqueness of weak solutions of the globally modified Navier-Stokes equations,	\emph{Adv. Nonlinear Stud.},  \textbf{9} (2)  (2009), 425--427.
		
		
		

        
		
		
		\bibitem{Schmalfussr} B. Schmalfu{\ss}, Backward cocycle and attractors of stochastic differential equations, In  International Seminar on Applied Mathematics Nonlinear Dynamics: Attractor Approximation and Global Behavior (V. Reitmann, T. Riedrich, and N. Koksch, eds.), Technische Universit\"{a}t Dresden, 1992, 185-192.
		
		
		
		\bibitem{Sultanov_2023}  O.A. Sultanov, Long-term behaviour of asymptotically autonomous Hamiltonian systems with multiplicative noise, \emph{SIAM J. Appl. Dyn. Syst.}, \textbf{22}(3) (2023), 1818--1851.
		
		
		\bibitem{R.Temam}	R. Temam, \emph{Infinite-Dimensional Dynamical Systems in Mechanics and Physics,} vol. 68, Applied Mathematical Sciences,	Springer, 1988.
		
		
		
		
        \bibitem{Temam2}	R. Temam, \emph{Navier-Stokes equations: Theory and numerical analysis}, Studies in Mathematics and its Applications, Vol. 2. North-Holland Publishing Co., Amsterdam-New York-Oxford, 1977.
		
		
		\bibitem{UTE-Wang} B. Wang, Attractors for reaction-diffusion equations in unbounded domains, \emph{Physica D}, \textbf{128} (1999), 41–52.
		
		
		
		
		
		
		\bibitem{PeriodicWang} B. Wang, Periodic random attractors for stochastic Navier-Stokes equations on unbounded domain, \emph{Electron. J. Differential Equations}, \textbf{2012} (59) (2012), 1-18.
		
		
		
		
		\bibitem{SandN_Wang} B. Wang, Sufficient and necessary criteria for existence of pullback attractors for non-compact random dynamical systems, \emph{J. Differential Equations}, \textbf{253} (5) (2012), 1544-1583.
		
		
		
		
		
		
		\bibitem{Wang+Guo+Liu+Nguyen_2024} R. Wang, and B. Guo, W. Liu and D.T. Nguyen, Fractal dimension of random invariant sets and regular random attractors for stochastic hydrodynamical equations, \emph{Math. Ann.}, \textbf{389} (1)  (2024),	671--718.
		
		
		
		\bibitem{Wang+Kinra+Mohan_2023}  R. Wang,  K. Kinra and M.T. Mohan,  Asymptotically autonomous robustness in probability of random attractors for stochastic Navier-Stokes equations on unbounded Poincar\'e domains, \emph{SIAM J. Math. Anal.},  \textbf{55} (4)  (2023),  2644--2676.
		
		\bibitem{WL} S. Wang and Y. Li, Longtime robustness of pullback random attractors for stochastic magneto-hydrodynamics equations, \emph{Physica D}, \textbf{382} (2018), 46–57.
		
		
		
		\bibitem{XC} J. Xu and T. Caraballo, Long time behavior of stochastic nonlocal partial differential equations and Wong-Zakai approximations, \emph{SIAM J. Math. Anal.}, \textbf{54} (3) (2022), 2792-2844.
		
		
		
		
		\bibitem{Zhao+Caraballo_2019} C. Zhao and T. Caraballo, Asymptotic regularity of trajectory attractor and trajectory statistical solution for the 3D globally modified Navier-Stokes equations, \emph{J. Differential Equations},  \textbf{266} (11)  (2019),  7205--7229.
		
		
		\bibitem{Zhao+Yang_2017}  C. Zhao and  L. Yang, Pullback attractors and invariant measures for the non-autonomous globally modified Navier-Stokes equations, \emph{Commun. Math. Sci.},  \textbf{15} (6)  (2017), 1565--1580.
		
	\end{thebibliography}
\end{document}